\documentclass[11pt,reqno]{amsart}
\usepackage{amssymb,mathrsfs,color}
\usepackage{enumerate}
\usepackage{graphicx}
\usepackage{xcolor} 
\usepackage{geometry}
\usepackage{amsfonts,amssymb,amsmath}
\usepackage{slashed}
\usepackage{ mathrsfs }
\usepackage{txfonts}
\usepackage[titletoc,title]{appendix}
\usepackage{wrapfig}
\usepackage{appendix}

\usepackage{amsmath}

\usepackage{cancel}
\usepackage{cite}
\usepackage{bm}
\usepackage{nth}
\usepackage[T1]{fontenc}
\usepackage{hyperref}

\usepackage{marginnote}
\usepackage{cancel}
\usepackage{pgfplots}
\usepackage{upgreek}
\usepackage{slashed}

\definecolor{green}{rgb}{0,0.8,0} 

\definecolor{deepgreen}{cmyk}{1,0,1,0.5}

\newcommand{\Del}[1]{}

\numberwithin{equation}{section}

\newtheorem{theorem}{Theorem}[section]
\newtheorem{lemma}[theorem]{Lemma}
\newtheorem{proposition}[theorem]{Proposition}
\newtheorem{remark}[theorem]{Remark}

\renewcommand{\div}{\mathrm{div}\,}

\newcommand{\bff}{{\bf f}}

\newcommand{\bfn}{{\bf n}}

\renewcommand{\hbar}{{\underline h}}

\newcommand{\bbR}{\mathbb R}

\newcommand{\calB}{\mathcal B}
\newcommand{\calC}{\mathcal C}
\newcommand{\calD}{\mathcal D}
\newcommand{\calE}{\mathcal E}

\newcommand{\scE}{{\mathscr{E}}}

\makeatletter
\newsavebox{\@brx}
\newcommand{\llangle}[1][]{\savebox{\@brx}{\(\m@th{#1\langle}\)}%
	\mathopen{\copy\@brx\kern-0.5\wd\@brx\usebox{\@brx}}}
\newcommand{\rrangle}[1][]{\savebox{\@brx}{\(\m@th{#1\rangle}\)}%
	\mathclose{\copy\@brx\kern-0.5\wd\@brx\usebox{\@brx}}}
\makeatother

\usepackage{nth}
\begin{document}
	\title{On nonlinear instability of liquid Lane-Emden stars}
	\author{Zeming Hao
	\and Shuang Miao}
	\maketitle
	\begin{abstract}
 We establish a dynamical nonlinear instability of liquid Lane-Emden stars in $\bbR^{3}$ whose adiabatic exponents take values in $[1,\frac43)$. Our proof relies on a priori estimates for the free boundary problem of a compressible self-gravitating liquid, as well as a quantitative analysis of the competition between the fastest linear growing mode and the source. 
	\end{abstract}
	\section{Introduction}\label{1}
	A classical model for Newtonian stars is given by the compressible Euler-Poisson system, as an idealized self-gravitating gas or liquid surrounded by vacuum, kept together by self-consistent gravitational force. In this paper, we consider the compressible Euler-Poisson equations for a perfect fluid with no heat conduction and no viscosity. It is given by 
	\begin{align}\label{main eq}
	\begin{split}
	&\rho\left( u_{t}+u\cdot\nabla u\right) =-\nabla P
	-\rho\nabla\psi,\quad\quad\quad \ in \   \mathcal{B}(t) \\
	&\rho_{t}+\div(\rho u)=0 ,\quad \quad\quad\quad\quad\quad\quad\ \ in \   \calB(t)
		\end{split}
	\end{align}
here \(u\) is the fluid velocity, 
\(\rho\) is the fluid density,  $P$ is the fluid pressure, and \(\psi\) is the gravitational potential. All these quantities are functions of spacetime variable $(x,t)$.  \(\psi\) is defined as
\begin{align}\label{potential}
\psi(t,x):=-\int_{\mathcal{B}(t)}\frac{\rho(t,y)}{|x-y|}dy,
\end{align}
so that with\(\chi_{\mathcal{B}(t)}\) denoting the characteristic function of \(\mathcal{B}(t)\)
\begin{align*}
\Delta\psi(t,x)=4\pi \rho\chi_{\mathcal{B}(t)}.
\end{align*}
The fluid is initially supported on a compact domain \(\mathcal{B}_{0}\subset\mathbb{R}^{3}\). The domain $\calB(t)$ occupied by fluid does not have a fixed shape and as the system evolves \(\mathcal{B}(t)\) changes in time. We consider a free boundary problem such that there is no surface tension on $\partial\calB(t)$ and any fluid particle on the boundary always stays on the boundary in the evolution. Therefore the following boundary conditions hold:
	\begin{align}\label{boundary conditions}
 P=0\quad \textrm{on}\quad \partial\mathcal{B}(t),\quad
\partial_{t}+u\cdot\nabla\in  \,T(t,\partial\mathcal{B}(t)).
\end{align}
There are four equations (the scalar continuity equation plus the vector momentum equation) but five unknown functions (\(\rho,P\) and \(u\)) in \eqref{main eq}. To close the system, we need to specify an equation of state. The model we consider is that of a barotropic fluid where  the pressure is a function of the fluid density only and is expressed through the following \emph{polytropic} equation of state
\begin{align}\label{eq of state}
P=K\rho^{\gamma}-C,
\end{align}
for some constants \(K,C>0\) and \(1\le\gamma\le 2\). The fact that the model \eqref{main eq}-\eqref{boundary conditions} describes the motion of a \emph{liquid} is reflected by the boundary condition $P=0$ on $\partial\calB(t)$ and the equation of state \eqref{eq of state}. In particular \eqref{eq of state} implies that on $\partial\calB(t)$ the fluid density is a positive constant $\rho_{0}>0$. Without loss of generality, we can take \(K=C=1\).  As we shall see, the constant \(\gamma\), called the adiabatic index, plays a crucial role in our analysis.

The Euler-Poisson system \eqref{main eq}-\eqref{boundary conditions} has steady-state solutions called Lane-Emden stars. In the case of radial symmetry, the steady-state equation becomes an ODE. Lam \cite{L} has shown that liquid Lane-Emden stars are linearly unstable when $1\le\gamma<\frac43$ and the density at the center is large enough. We shall study the nonlinear dynamical behavior of the Lane-Emden stars in this regime.
\subsection{History and background}
The Euler-Poisson system with the polytropic equation of state has played a crucial role in astrophysics. It is used to model stellar structure and evolution by Chandrasekhar \cite{Chan}, Shapiro, and Teukolsky \cite{ST}. When the pressure and the gravity are perfectly balanced, the Euler-Poisson system with the polytropic equation of state has an important class of time-independent steady state solutions known as the Lane-Emden stars. For \(\gamma=\frac{4}{3}\), astrophysicists Goldreich and Weber \cite{GW} found a special class of expanding and collapsing solutions that model stellar collapse and expansion such as supernova expansion. To understand the physical interest corresponding to these special solutions, it is important to have an existence theory for the Euler and Euler-Poisson equation. 
	
In the gaseous case, where the fluid density vanishes on the boundary, the local existence for the free boundary problem of the Euler-Poisson system \eqref{main eq}-\eqref{boundary conditions} (or Euler system for which the self-gravity is not present) has been extensively studied, see \cite{CS,JMas,IT,GL} and references therein. In the liquid case, the corresponding local existence for the Euler-Poisson system was proved in \cite{GLL} (see also \cite{Lind05,Trak09} for the Euler system, and \cite{Oliy,MSW} for relativistic liquid), where the proof relies on delicate energy estimates and optimal elliptic estimates in terms of boundary regularity. The local well-posedness theory provides a solid foundation for studying the stability of Lane-Emden stars.

There has been development on the problem of gaseous Lane–Emden stars. The linear stability was studied in \cite{Lin,JMak}, in which the authors show that Lane–Emden stars in \(\mathbb{R}^{3}\) are linearly unstable when \(1<\gamma<\frac{4}{3}\) and linearly stable when \(\frac{4}{3}\le\gamma <2\). Jang \cite{Jang08,Jang14} prove the full nonlinear, dynamical instability of the steady profile for \(\frac{6}{5}\le\gamma<\frac{4}{3}\) using a bootstrap argument and nonlinear weighted energy estimates. But the nonlinear stability in the range \(\frac{4}{3}<\gamma <2\) has not been fully understood and remains open. By the linear stability, under the assumption that a global-in-time solution exists, nonlinear stability for \(\frac{4}{3}<\gamma <2\) was established by Rein in \cite{Rein} using a variational approach based on the fact that the steady states are minimizers of an energy functional. See other approaches in \cite{LS09,LS08}. In the critical case \(\gamma=\frac{4}{3}\) the Lane-Emden star is nonlinearly unstable despite the conditional linear stability - in fact, Deng, Liu, Yang, and Yao \cite{DLYY} show that the energy of a steady state is zero and any small perturbation can make the energy positive and cause part of the gas go to infinity. Despite the above advances for gaseous stars,  only recently a few results have been established for the stability problem of Lane–Emden stars for a liquid.

Assuming spherical symmetry, so that \(\vec{u}(x,t)=u(r,t)\hat{r}\) where \(\hat{r}=\frac{x}{|x|}\), the continuity equation in \(\eqref{main eq}\) becomes
\begin{align}
D_{t}\rho+\rho\frac{\partial_{r}(r^{2}u)}{r^{2}}=\partial_{t}\rho+u\partial_{r}\rho+\rho\frac{\partial_{r}(r^{2}u)}{r^{2}}=0.
\end{align}
The momentum equation reads
\begin{align}\label{momentum eq SS}
D_{t}u+\partial_{r}\psi+\frac{1}{\rho}\partial_{r}P=\partial_{t}u+u\partial_{r}u+\partial_{r}\psi+\frac{1}{\rho}\partial_{r}P=0.
\end{align}
The gravitational potential satisfies the Poisson equation
\begin{align}\label{Poisson eq SS}
4\pi\rho=\Delta\psi=\frac{1}{r^{2}}\partial_{r}(r^{2}\partial_{r}\psi).
\end{align}
We put \(\eqref{Poisson eq SS}\) into the momentum equation \(\eqref{momentum eq SS}\) to get
\begin{align}
\partial_{t}u+u\partial_{r}u+\frac{4\pi}{r^{2}}\int_{0}^{r}s^{2}\rho(s)ds+\frac{1}{\rho}\partial_{r}p=0.
\end{align}
Now we write the Euler-Poisson system \(\eqref{main eq}\) in  Lagrangian coordinates. Let \(\eta(y,t)\) be the radial position of the fluid particle at time \(t\) so that 
\begin{align*}
\partial_{t}\eta=u\circ\eta \ \ \ \  \text{with} \ \ \ \ \eta(y,0)=\eta_{0}(y).
\end{align*}
Here \(\eta_{0}\) is not necessarily the identity map and depend on the initial density profile and in fact. Then we have the Lagrangian variables
\begin{align*}
&\upsilon=u\circ\eta \ \ (Lagrangian\ velocity)\\
&f=\rho\circ\eta\ \ (Lagrangian\ density)\\ 
&\varphi=\psi\circ\eta\ \ (Lagrangian\ potential)\\
&J=\frac{\eta^{2}}{y^{2}}\partial_{y}\eta\ \ (Jacobian\ determinant\ (renormalized)).
\end{align*}
So the continuity equation in Lagrangian coordinate is 
\begin{align*}
\partial_{t}f+f\frac{\partial_{t}J}{J}=0\ \ \ \ and\ so \ \ \ \ \partial_{t}\log(fJ)=0\ \ \Rightarrow \ \ fJ=f_{0}J_{0}.
\end{align*}
And the momentum equation becomes
\begin{align*}
\partial_{t}^{2}\eta+\frac{4\pi}{\eta^{2}}\int_{0}^{y}s^{2}(f_{0}J_{0})ds+\frac{1}{f_{0}J_{0}}\frac{\eta^{2}}{y^{2}}\partial_{y}\left( f_{0}J_{0}\frac{3y^{2}}{\partial_{y}\eta^{3}}\right) ^{\gamma}=0.
\end{align*}
Then we consider the linearized equation of Euler-Poisson system \(\eqref{main eq}-\eqref{boundary conditions}\) for perturbation. Assume \(\bar{\rho}\) is a steady state solution. For a small perturbation \(\rho_{0}=\bar{\rho}+\varepsilon\) and \(u_{0}=\upsilon\), we can pick \(\eta_{0}\) such that  \(f_{0}J_{0}=\bar{\rho}\). Let \(\eta(y, t)=y(1+\zeta(y,t))\) and \(\sigma=\log\frac{\rho}{\bar{\rho}}\), we obtain the linearized equation for the perturbation
\begin{align}\label{linearized eq SS}
\left\{
\begin{aligned}
&\partial_{t}\upsilon+\frac{4}{\bar{\rho}}\zeta\partial_{y}\bar{\rho}^{\gamma}-\gamma\frac{1}{\bar{\rho}}\partial_{y}(\bar{\rho}^{\gamma}(3\zeta+y\partial_{y}\zeta))=0\\
&\partial_{t}\sigma+\partial_{y}\upsilon+\frac{2\upsilon}{y}=0\\
&\partial_{t}\zeta-\frac{\upsilon}{y}=0
\end{aligned}
\right.
\end{align}
Lam \cite{L} shows that liquid Lane-Emden stars are linearly stable when \(\frac{4}{3}\le\gamma\le2\), or \(1\le\gamma<\frac{4}{3}\) for stars with small central density. It is also shown in \cite{L} that the above linearized equation admits a growing mode solution of the form \(\zeta(y,t)=e^{\lambda t}\chi(y)\) with \(\lambda>0\) when \(\gamma\in[1,\frac43)\) and the star possesses a large central density. The linear stability of the gaseous Lane-Emden stars does not depend on its central density. However in the liquid case, the stability of the Lane-Emden stars does depend on its central density, see \cite{L} and the references (for instance \cite{hadvzic2021stability,hadvzic2021turning}) therein. The following lemma (see \cite{guo}) is crucial for us to obtain nonlinear instability.
\begin{lemma}\label{1.4.}
	Assume that \(L\) is a linear operator on a Banach space \(X\) with norm \(\arrowvert\arrowvert\cdot\arrowvert\arrowvert\), and \(e^{tL}\) generates a strongly continuous semigroup on \(X\) such that
	\begin{align}
	\arrowvert\arrowvert e^{tL}\arrowvert\arrowvert_{(X,X)}\le C_{L}e^{\lambda t}
	\end{align}
	for some \(C_{L}\) and \(\lambda> 0\). Assume a nonlinear operator \(N(y)\) on \(X\) and another norm \(\arrowvert\arrowvert\arrowvert\cdot\arrowvert\arrowvert\arrowvert\), and constant \(C_{N}\), such that
	\begin{align}
	\arrowvert\arrowvert N(y)\arrowvert\arrowvert\le C_{N}\arrowvert\arrowvert\arrowvert y\arrowvert\arrowvert\arrowvert^{2}
	\end{align}
	for all \(y\in X\) and \(\arrowvert\arrowvert\arrowvert y\arrowvert\arrowvert\arrowvert <\infty\). Assume for any solution \(y(t)\) to the equation
	\begin{align}
	y'=Ly+N(y)
	\end{align}
	with \(\arrowvert\arrowvert\arrowvert y(t)\arrowvert\arrowvert\arrowvert\le \sigma\), there exist \(C_{0},C_{\sigma}>0\) such that for any small \(\epsilon>0\), there exists \(C_{\epsilon}>0\) such that the following sharp energy estimate holds:
	\begin{align}\label{1.15}
	\arrowvert\arrowvert\arrowvert y(t)\arrowvert\arrowvert\arrowvert\le C_{0}\arrowvert\arrowvert\arrowvert y(0)\arrowvert\arrowvert\arrowvert+\int_{0}^{T}\epsilon\arrowvert\arrowvert\arrowvert y(s)\arrowvert\arrowvert\arrowvert+C_{\sigma}\arrowvert\arrowvert\arrowvert y(s)\arrowvert\arrowvert\arrowvert^{\frac{3}{2}}+C_{\epsilon}\arrowvert\arrowvert y(s)\arrowvert\arrowvert ds.
	\end{align}
	Consider a family of initial data \(y^{\delta}(0)=\delta y_{0}\) with \(\arrowvert\arrowvert y_{0}\arrowvert\arrowvert=1\) and \(\arrowvert\arrowvert\arrowvert y_{0}\arrowvert\arrowvert\arrowvert<\infty\) and let \(\theta_{0}\) be a sufficiently small (fixed) number. Then there exists some constant \(C>0\) such that if
	\begin{align}
	0\le t\le T^{\delta}\equiv\frac{1}{\lambda}\log\frac{\theta_{0}}{\delta}
	\end{align}
	we have
	\begin{align}\label{1.17}
	\arrowvert\arrowvert y(t)-\delta e^{Lt}y_{0}\arrowvert\arrowvert\le C[\arrowvert\arrowvert\arrowvert y_{0}\arrowvert\arrowvert\arrowvert^{2}+1]\delta^{2}e^{2\lambda t}.
	\end{align}
	In particular, if there exists a constant \(C_{p}\) such that \(\arrowvert\arrowvert e^{Lt}y_{0}\arrowvert\arrowvert\ge C_{p}e^{\lambda t}\) then at the escape time
	\begin{align}\label{1.18}
	\arrowvert\arrowvert y(T^{\delta})\arrowvert\arrowvert\ge\tau_{0}>0,
	\end{align}
	where \(\tau_{0}\) depends explicitly on \(C_{L},C_{N},C_{0},C_{\sigma},C_{p},\lambda,y_{0},\sigma\) and is independent of \(\delta\).
\end{lemma}
\begin{remark}
	\(y(t)\) represents the perturbation away from the
	trivial steady state and \(\lambda\) is the fastest growing mode of the linear operator \(L\). It is worth noting that The small constant \(\epsilon\) only needs to be less than, say, \(\frac{\lambda}{4}\). 
\end{remark}
\begin{remark}
	In our application, \(\arrowvert\arrowvert\arrowvert\cdot\arrowvert\arrowvert\arrowvert\) represents higher order Sobolev norms (including internal and boundary norms). \(\arrowvert\arrowvert\cdot\arrowvert\arrowvert\) represents \(L^{2}\) norm in the fluid interior.	 \(\arrowvert\arrowvert\arrowvert\cdot\arrowvert\arrowvert\arrowvert\) is a stronger norm than \(\arrowvert\arrowvert\cdot\arrowvert\arrowvert\) and the assumption \(\arrowvert\arrowvert N(y)\arrowvert\arrowvert\le C_{N}\arrowvert\arrowvert\arrowvert y\arrowvert\arrowvert\arrowvert^{2}\) means that the nonlinearity is not bounded in the weaker norm \(\arrowvert\arrowvert\cdot\arrowvert\arrowvert\). This is a key analytical
	difficulty in many instability problems.
\end{remark}
Lemma \(\ref{1.4.}\) provides a general framework to study nonlinear instability, see other approaches \cite{VF,ChSu} for study of nonlinear instability in free boundary problems. The most delicate part of applying this framework is that the higher order Sobolev norm \(\arrowvert\arrowvert\arrowvert\cdot\arrowvert\arrowvert\arrowvert\) does not create a faster growth rate than the weaker counterpart \(\arrowvert\arrowvert\cdot\arrowvert\arrowvert\), over the time scale of \(0\le t\le T^{\delta}\). In other words, roughly speaking, the stronger norm \(\arrowvert\arrowvert\arrowvert\cdot\arrowvert\arrowvert\arrowvert\) is controlled reversely by the weaker norm \(\arrowvert\arrowvert\cdot\arrowvert\arrowvert\). This is the key to close the instability argument.
\subsection{Statement of the main result}
In this paper we  prove the nonlinear instability for the liquid Lane-Emden stars. To state the main theorem, we consider the Euler-Poisson system \eqref{main eq}-\eqref{boundary conditions} in a perturbative form around the equilibrium state \(\bar{\rho}\) (Let \(\rho=\bar{\rho}+\tilde{\rho}\), then the unknowns of the system \eqref{main eq}-\eqref{boundary conditions} become \((u,\tilde{\rho})\)). Our main result establishes the full nonlinear dynamical instability of the Lane-Emden equilibrium. In the statement of the following theorem, for any \(\delta>0\) and \(\theta_{0}>\delta\), we define
\begin{align}
T^{\delta}\equiv\frac{1}{\sqrt{\mu_{0}}}\log\frac{\theta_{0}}{\delta},
\end{align}
where \(\sqrt{\mu_{0}}\) is the sharp linear growth rate.
\begin{theorem}\label{main th}
	Assume that \(1\le\gamma<\frac{4}{3}\) and the stars have large central densities. For any sufficiently small \(\delta>0\), there exists a family of initial data \((u^{\delta}(0),\tilde{\rho}^{\delta}(0))=\delta(u_{0},\tilde{\rho}_{0})\) and \(T^{\delta}>0\) such that the perturbed solutions \((u^{\delta}(t),\tilde{\rho}^{\delta}(t))\) to the Euler-Poisson system \eqref{main eq}-\eqref{boundary conditions} for \(t\in[0,T^{\delta}]\) satisfy
	\begin{align*}
	\arrowvert\arrowvert(u^{\delta}(T^{\delta}),\tilde{\rho}^{\delta}(T^{\delta}))\arrowvert\arrowvert_{L^{2}}\ge \tau_{0}>0,
	\end{align*}
	where \(\tau_{0}\) is independent of \(\delta\).
\end{theorem}
	\begin{remark}
		The above result shows that no matter how small the amplitude of initial perturbed
		data is taken to be, we can find a solution such that the corresponding energy escapes at a time \(T^{\delta}\): there is no stabilization of the system. We conclude from this that liquid Lane-Emden stars with large central density for \(1\le\gamma<\frac{4}{3}\) are nonlinearly unstable.
	\end{remark}
\begin{remark}
	We note that the nonlinear dynamics of any general perturbation is dominated by the fastest linear growing mode \(\sqrt{\mu_{0}}\) up to the time scale of \(T^{\delta}\). This implies that the nonlinear instability is essentially driven by the linear instability. The hypothesis of Theorem \(\ref{main th}\) guarantees the occurrence of the linear instability and a fast linear growth rate.
\end{remark}
\subsection{Main ideas for the proof}
	Assume the exact solution \(y_{\ast}\) of Euler-Poisson system \(\eqref{main eq}-\eqref{boundary conditions}\) is written as
\begin{align*}
	y_{\ast}=\bar{y}+y_{L}+y_{e},
\end{align*}
where \(\bar{y}\) is the steady state, \(y_{L}\) is the solution of linearized equation and \(y_{e}\) is the error. Based on linear instability, we have \(\Arrowvert y_{L}\Arrowvert\sim\delta e^{\lambda t}\), where \(\delta\) is the magnitude of the initial perturbation. Suppose there exists small enough \(\epsilon>0\) and constants \(C_{\sigma},C_{\epsilon}>0\) such that the perturbation solution \(y:=y_{L}+y_{e}\) satisfies the following sharp energy estimate
\begin{align*}
	\arrowvert\arrowvert\arrowvert y(t)\arrowvert\arrowvert\arrowvert\le\delta+\int_{0}^{t}\epsilon\lambda \arrowvert\arrowvert\arrowvert y(s)\arrowvert\arrowvert\arrowvert+C_{\sigma}\arrowvert\arrowvert\arrowvert y(s)\arrowvert\arrowvert\arrowvert^{\frac{3}{2}}+C_{\epsilon}\arrowvert\arrowvert y(s)\arrowvert\arrowvert ds.
\end{align*}
Then according to Lemma \(\ref{1.4.}\), we have \(\Arrowvert y\Arrowvert\sim\frac{1}{2}\delta e^{\lambda t}\), which means linear solution \(y_{L}\) dominate the nonlinear correction \(y_{e}\). Therefore at the escape time \(t=T^{\delta}\), we shall see
\begin{align*}
	\Arrowvert y(T^{\delta})\Arrowvert\gtrsim\frac{1}{2}\delta e^{\lambda T^{\delta}}=\frac{1}{2}\theta_{0},
\end{align*}
where \(\theta_{0}\) is independent of \(\delta\), and instability happen.

Therefore an appropriate a priori estimate for the perturbed solution is the key to prove the nonlinear instability. To obtain such an estimate, we follow the approach in \cite{MSW}. We start by deriving quasilinear system from compressible Euler-Poisson system \eqref{main eq}-\eqref{boundary conditions}. By introducing an acoustical metric $g$ and its corresponding wave operator $\Box_{g}$ (see \eqref{2.4}), we obtain a coupled quasilinear system for fluid velocity \(u\) and logarithmic density perturbation variable \(\varepsilon\) (see \(\eqref{2.5},\eqref{2.15},\eqref{2.20}\) and \(\eqref{2.24}\)):
\begin{equation}\label{1.19}
\left\{
\begin{aligned}
&\left(B^{2}+a\nabla_{n}^{\left( g\right) }\right)u^{i}=-c_{s}^{2}\nabla_{i} B\bm{\uprho}-B\nabla_{i}\psi\quad \textrm{on}\ \partial\calB(t),\\
&\square_{g}u^{i}=-\left(1+c_{s}^{-1}c'_{s}\right) \left( g^{-1}\right) ^{\alpha\beta}\partial_{\alpha}\bm{\uprho}\partial_{\beta}u^{i}+B\nabla_{i}\psi-2c_{s}^{-1}c'_{s}(B\bm{\uprho})\nabla_{i}\psi\quad \textrm{in}\ \calB(t),
\end{aligned}
\right.
\end{equation}
and
\begin{equation}\label{1.20}
\left\{
\begin{aligned}
\square_{g}\varepsilon=&-3c_{s}^{-1}c'_{s}\left( g^{-1}\right) ^{\alpha\beta}\partial_{\alpha}\bm{\uprho}\partial_{\beta}\varepsilon+\mathscr{Q}\quad\textrm{in}\ \calB(t),\quad\quad\varepsilon,\ B\varepsilon=0\ \ \textrm{on}\ \ \partial\calB(t), \\
\square_{g}B\varepsilon=&-(1+3c_{s}^{-1}c'_{s})\left( g^{-1}\right) ^{\alpha\beta}\left( \partial_{\alpha}\bm{\uprho}\right) \left( \partial_{\beta}B\varepsilon\right)-2c_{s}c'_{s}\delta^{ab}\left( \partial_{a}\bm{\uprho}\right) \left( \partial_{b}B\varepsilon\right)+\mathscr{P}\quad \textrm{in}\ \calB(t).
\end{aligned}
\right.
\end{equation}
Here $\mathscr{Q}, \mathscr{P}$ are lower order terms whose precise expressions are given in \eqref{lower order terms}.  $\bm{\uprho}$ is the logarithmic density defined to be $\log\frac{\rho}{\rho_{0}}$. $B$ is the vectorfield $\partial_{t}+u^{a}\partial_{a}$ and $c_{s}$ is the sound speed defined by $c_{s}:=\sqrt{\frac{dP}{d\rho}}$. In \cite{MSW} a similar quasilinear system was shown to be hyperbolic type, and a local-wellposedness result of free boundary hard phase fluids was obtained. For the fluid velocity equations \(\eqref{1.19}\), We multiply the boundary equation and interior equation by $\frac{1}{a}(Bu)$,  \(Bu\) respectively and integrate. By observing the signs of the boundary terms, we obtain the following energy in Lemma \(\ref{3.2.}\)
\begin{align*}
\int_{\calB(t)}\left| \partial_{t,x}u\right| ^{2}dx+\int_{\partial\calB(t)}\frac{1}{a}\left| Bu\right| ^{2}dS.
\end{align*}
For the wave equation for \(B\varepsilon\) with Dirichlet boundary conditions, we choose a suitable multiplier consisting of an appropriate linear combination of \(B\) and the normal \(n\) to $\partial\calB(t)$, and apply integration by parts in Lemma \(\ref{3.3.}\). The energy functional for \(\eqref{1.20}\) controls
\begin{align*}
\int_{\calB(t)}\left| \partial_{t,x}B\varepsilon\right| ^{2}dx+\int_{0}^{t}\int_{\partial\calB(\tau)}\left| \partial_{t,x}B\varepsilon\right| ^{2}dSd\tau.
\end{align*}
To control higher order Sobolev norms, we commute \(B^{k}\) with \(\eqref{1.19}\) and \(\eqref{1.20}\) (see \(\eqref{3.28}-\eqref{3.30}\)). Then we use elliptic estimates and the corresponding wave equation to control the higher order Sobolev norms in terms of the $L^{2}$-norms of $B^{k}$-derivatives (see details in Proposition \ref{5.2.}). For the most challenge gravitational potential in the source, we apply Hilbert transform and techniques in Clifford analysis to convert the derivative of \(\nabla \psi\) of any order into a function equal to zero on the boundary, which is treated in Section \(\ref{4}\). Then the energy functional \(\eqref{5.3}\) can control the \(L^{2}\) norm of the gravitational terms by using elliptic estimates repeatedly and controlling commutator errors. For other applications of Clifford analysis in fluid free boundary problems, see \cite{wu1999well, MS,ZZ}.  

To close the a priori estimate, we need to control the source carefully, which we divide into three types: \emph{the top order linear source}, \emph{the non-top order linear source} and \emph{the nonlinear source}. According to Lemma \ref{1.4.} and Sobolev embedding, the nonlinear source is almost in good forms. The details are given in Section \ref{5}.  For the linear source, we need to prove that their coefficients can be controlled by the fastest linear growth mode \(\mu_{0}\). By the standard Sobolev interpolation inequality, which helps us to transfer large coefficients of strong norm to above \(L^{2}\) norm, we can control the non-top order linear terms. Finally for the top order linear terms, as we shall show in Proposition \(\ref{5.1.}\), the coefficient \(C_{1}(\bar{\rho})\) in front of these terms is in form of
\begin{align}
C_{1}(\bar{\rho})\sim\left( c_{s}^{2}(\bar{\bm{\uprho}})+c_{s}(\bar{\bm{\uprho}})c'_{s}(\bar{\bm{\uprho}})\right)\nabla\bar{\bm{\uprho}} \sim\gamma\bar{\rho}^{\gamma-1}\nabla\bar{\bm{\uprho}}\lesssim-\frac{1}{\bar{\rho}}\partial_{y}\bar{\rho}^{\gamma}.
\end{align}
Observing the influence of the adiabatic index \(\gamma\) on the properties of the steady-state solution, we divide \(\gamma\) into three different intervals. Applying the basic relations \(\eqref{5.49}, \eqref{5.55}\) and asymptotic properties \(\eqref{333}\) satisfied by the steady-state solution, we prove that the coefficient \(C_{1}(\bar{\rho})\) can be dominated by the fastest linear growth mode \(\mu_{0}\) in all three intervals when the density of the star center is large enough. The details for treating the linear source are given in Section \ref{6}. Combined with the previous discussion and the a priori estimate \(\eqref{5.5}\), we finally prove that the liquid Lane-Emden stars are nonlinearly unstable.

\subsection{Outline of the paper} In Section \(\ref{2}\) we derive the quasilinear equations for fluid velocity \(u\) and logarithmic density perturbation variable \(\varepsilon\). In Section \(\ref{3}\), we derive the basic energy inequalities and higher order equations. In Section \(\ref{4}\), we introduce  Hilbert transform and describe how to apply them in the a priori estimates. Based on these, in Section \(\ref{5}\) we prove a sharp nonlinear energy estimate (see \ref{5.5}), and the integral terms of the right-hand side of \(\eqref{5.5}\) contain the three types of sources (i.e. the top order linear terms, the non-top order linear terms and the nonlinear terms). In Section \(\ref{6}\), we show that the coefficients of the linear terms on the right-hand side of \(\eqref{5.5}\) can be dominated by the fastest linear growth mode \(\mu_{0}\), which in fact proves our main result Theorem \(\ref{main th}\) by using a bootstrap argument.

\subsection*{Acknowledgment}
The authors would like to thank Qingtang Su and Yanlin Wang for stimulating discussions on this project. This work was supported by National Key R\& D Program of China 2021YFA1001700,  NSFC grant 12071360, and the Fundamental Research Funds for the Central Universities in China.

\section{Quasilinear equations for compressible Euler-Poisson system}\label{2}
We start by recalling the notation
\begin{align*}
\bm{\uprho}=\log(\rho/\rho_{0}), \ \ B=\partial_{t}+u^{a}\partial_{a},\ \ c_{s}^{2}=\frac{dP}{d\rho},
\end{align*}
where the constant $\rho_{0}>0$ is the boundary value of $\rho$. Then the compressible Euler-Poisson system \(\eqref{main eq}\) write
\begin{align}\label{2.1}
B\bm{\uprho}=&-\div u,\\ \label{2.2}
Bu^{i}=&-c_{s}^{2}\nabla_{i}\bm{\uprho}-\nabla_{i}\psi.
\end{align}
Let \(X:\mathbb{R}\times\calB(0)\to \calB\) be the Lagrangian parametrization of \(\calB=\calB(t)\). We can express the logarithmic density as the sum of the steady state and the perturbation
\begin{align*}
\bm{\uprho}(x,t)=\bar{\bm{\uprho}}(X^{-1}(x,t))+\varepsilon(x,t),\ \ \ \ where\ \ \ \ \bar{\bm{\uprho}}=\log\bar{\rho}
\end{align*}
By a slight abuse of notation, we often write \(\bar{\bm{\uprho}}(x,t)\) instead of \(\bar{\bm{\uprho}}(X^{-1}(x,t))\), therefore we have
\begin{align}\label{2.3}
B\bar{\bm{\uprho}}=0.
\end{align}
In order to obtain quasilinear equations, we introduce the following acoustical metric and the corresponding covariant wave operator
\begin{align}\label{2.4}
\begin{split}
&g:=-dt\otimes dt+c_{s}^{-2}\sum_{a=1}^{3}\left(dx^{a}-u^{a}dt\right)\otimes \left(dx^{a}-u^{a}dt\right),\\
&g^{-1}:=-B\otimes B+c_{s}^{2}\sum_{a=1}^{3}\partial_{a}\otimes\partial_{a},\\
&\square_{g}:=\frac{1}{\sqrt{|\det g|}}\partial_{\alpha} \left\{\sqrt{|\det g|}(g^{-1})^{\alpha\beta}\partial_{\beta}\right\}.
\end{split}
\end{align}
It is straightforward to verify that \(g^{-1}\) is the matrix inverse of \(g\), and in this case the vectorfield \(B\) is timelike, future-directed, orthogonal to \(\calB_{t}\), and with unit-length.

\begin{lemma}[\(\textbf{Quasilinear equation on the free boundary}\)]
	Given the system \(\eqref{2.1}-\eqref{2.2}\), we have the following equation on the boundary:
	\begin{align}\label{2.5}
	\left(B^{2}+a\nabla_{n}^{\left( g\right) }\right)u^{i}=-c_{s}^{2}\nabla_{i}B\varepsilon-B\nabla_{i}\psi,\quad \textrm{on}\quad \partial\calB(t).
	\end{align}
	Here \(a=\sqrt{\partial_{\alpha}\bm{\uprho}\partial^{\alpha}\bm{\uprho}}\).
	\begin{proof}
		According to \eqref{boundary conditions},we get
		\begin{align}\label{2.6}
		B\bm{\uprho}=0,\quad \textrm{on}\quad \partial\calB(t).
		\end{align}
		Taking a \(B\) derivative to the momentum equation \eqref{2.2}, using \(\eqref{2.6}\) and restricting it to the boundary we get
		\begin{align}\label{2.7}
		B^{2}u^{i}+c_{s}^{2}B\nabla_{i} \bm{\uprho}=-B\nabla_{i}\psi.
		\end{align}
		Computing the commutator \(\left[B,\nabla_{i}\right]\bm{\uprho}=B\nabla_{i} \bm{\uprho}-\nabla_{i} B\bm{\uprho}=-\nabla \bm{\uprho}\cdot\nabla u^{i} \) yields
		\begin{align}\label{2.8}
		B^{2}u^{i}-c_{s}^{2}\nabla \bm{\uprho}\cdot\nabla u^{i}=-\nabla_{i} B\bm{\uprho}-B\nabla_{i}\psi.
		\end{align}
		We now simplify the term \(\left( c_{s}^{2}\nabla \bm{\uprho}\cdot\nabla u^{i}\right) \). According to the expression of \(g^{-1}\) ,we have
		\(\partial^{a}=-B^{a}B+c_{s}^{2}\partial_{a}\) and \(\partial^{0}=-B\). Using \(\eqref{2.6}\) we get
		\begin{align*}
		\partial^{a}\bm{\uprho}=c_{s}^{2}\partial_{a}\bm{\uprho},\quad \partial^{0}\bm{\uprho}=0, \quad \textrm{on}\quad \partial\calB(t),
		\end{align*}
		and hence
		\begin{align*}
		c_{s}^{2}\nabla \bm{\uprho}\cdot\nabla u^{i}=\partial^{\alpha}\bm{\uprho}\partial_{\alpha}u^{i}=\partial_{\alpha}\bm{\uprho}\partial^{\alpha}u^{i}.
		\end{align*}
		We use \(\nabla^{\left( g\right) }\) for the (spacetime) gradient (with respect to \(g\)). Let \(n\) be the unit outward pointing (spacetime) normal to \(\partial\calB(t)\). Since \(\bm{\uprho}=0\ \textrm{on}\  \partial\calB(t),\ -\nabla^{\left( g\right) } \bm{\uprho}=an\) on the free boundary \(\partial\calB(t)\), with \(a=\sqrt{\partial_{\alpha}\bm{\uprho}\partial^{\alpha}\bm{\uprho}}\). According to \(\eqref{2.3}\) and simple calculations, we get
		\begin{align}\label{2.9}
		\left(B^{2}+a\nabla^{\left( g\right)}_{n}\right)u^{i}=-c_{s}^{2}\nabla_{i} B\varepsilon-B\nabla_{i}\psi,
		\end{align}
		which completes the proof of \(\eqref{2.5}\).
	\end{proof}
\end{lemma}

\begin{lemma}[\(\square_{g}\) \(\textbf{relative to the Cartesian coordinates}\)]
	The covariant wave operator \(\square_{g}\) acts on scalar functions \(\phi\)  via the following identity:
	\begin{align}\label{2.11}
	\square_{g}\phi=-BB\phi+c_{s}^{2}\delta^{ab}\partial_{a}\partial_{b}\phi+2c_{s}^{-1}c'_{s}\left(B\varepsilon\right)B\phi-\left(\partial_{a}u^{a}\right)B\phi-c_{s}^{-1}c'_{s}\left( g^{-1}\right)^{\alpha\beta}\left(\partial_{\alpha}\bm{\uprho}\right) \partial_{\beta}\phi.
	\end{align}
	\begin{proof}
		It is straightforward to compute using equations \(\eqref{2.4}\) that relative to Cartesian coordinates, we have
		\begin{align}\label{2.12}
		\det g=-c_{s}^{-6}
		\end{align}
		and hence
		\begin{align}\label{2.13}
		\sqrt{|\det g|}g^{-1}=-c_{s}^{-3}B\otimes B+c_{s}^{-1}\sum_{a=1}^{3}\partial_{a}\otimes\partial_{a}.
		\end{align}
		Using\(\eqref{2.4}\), \(\eqref{2.12}\) and \(\eqref{2.13}\), we compute that
		\begin{align}\label{2.14}
		\begin{split}
		\square_{g}\phi=&-c_{s}^{3}\left(B^{\, \alpha}\partial_{\alpha}\left( c_{s}^{-3}\right) \right) B^{\,\beta}\partial_{\beta}\phi-\left( \partial_{\alpha}B^{\,\alpha}\right) B^{\,\beta}\partial_{\beta}\phi-\left( B^{\,\alpha}\partial_{\alpha}B^{\,\beta}\right) \partial_{\beta}\phi\\
		&-B^{\,\alpha}B^{\,\beta}\partial_{\alpha}\partial_{\beta}\phi+c_{s}^{2}\delta^{ab}\partial_{a}\partial_{b}\phi-c_{s}c'_{s}\delta^{ab}\left( \partial_{a}\bm{\uprho}\right) \partial_{b}\phi.
		\end{split}
		\end{align}
		Finally, from \(\eqref{2.14}\), the expression for \(B\), the expression for \(g^{-1}\),  and simple calculations, we arrive at \(\eqref{2.11}\).
	\end{proof}
\end{lemma}
We now establish equation for \(u\).
\begin{lemma}[\(\textbf{Wave equation for}\) \(u\)]
	The compressible Euler-Poisson equations \(\eqref{2.1}-\eqref{2.2}\) imply the following covariant wave equation for the scalar-valued function \(u^{i}\):
	\begin{align}\label{2.15}
	\begin{split}
	\square_{g}u^{i}=-\left(1+c_{s}^{-1}c'_{s}\right) \left( g^{-1}\right) ^{\alpha\beta}\partial_{\alpha}\bm{\uprho}\partial_{\beta}u^{i}+B\nabla_{i}\psi-2c_{s}^{-1}c'_{s}(B\varepsilon)\nabla_{i}\psi.
	\end{split}
	\end{align}
	\begin{proof}
		First, we use \(\eqref{2.11}\) with \(\phi=u^{i}\) to deduce
		\begin{align}\label{2.16}
		\square_{g}u^{i}=-BBu^{i}+c_{s}^{2}\delta^{ab}\partial_{a}\partial_{b}u^{i}+2c_{s}^{-1}c'_{s}(B\varepsilon) Bu^{i}-\left( \partial_{a}u^{a}\right) Bu^{i}-c_{s}^{-1}c'_{s}\left( g^{-1}\right) ^{\alpha\beta}\partial_{\alpha}\bm{\uprho}\partial_{\beta}u^{i}.
		\end{align}
		Next, we use \(\eqref{2.1}-\eqref{2.2}\) to compute that
		\begin{align}\label{2.17}
		\begin{split}
		BBu^{i}=&-c_{s}^{2}\delta^{ia}B\partial_{a}\bm{\uprho}-2c_{s}c'_{s}B\bm{\uprho}\delta^{ia}\partial_{a}\bm{\uprho}-B\nabla_{i}\psi\\
		=&-c_{s}^{2}\delta^{ia}\partial_{a}\left( B\bm{\uprho}\right) +c_{s}^{2}\delta^{ia}\partial_{a}u^{b}\partial_{b}\bm{\uprho}-2c_{s}c'_{s}B\bm{\uprho}\delta^{ia}\partial_{a}\bm{\uprho}-B\nabla_{i}\psi\\
		=&c_{s}^{2}\delta^{ia}\delta_{c}^{b}\partial_{a}\left( \partial_{b}u^{c}\right) +c_{s}^{2}\delta^{ia}\partial_{a}u^{b}\partial_{b}\bm{\uprho}-2c_{s}c'_{s}B\bm{\uprho}\delta^{ia}\partial_{a}\bm{\uprho}-B\nabla_{i}\psi\\
		=&c_{s}^{2}\delta^{bc}\partial_{b}\partial_{c}u^{i}+c_{s}^{2}\delta^{ia}\partial_{c}\left( \partial_{a}u^{c}-\partial_{c}u^{a}\right) +c_{s}^{2}\delta^{ia}\partial_{a}u^{b}\partial_{b}\bm{\uprho}-2c_{s}c'_{s}B\bm{\uprho}\delta^{ia}\partial_{a}\bm{\uprho}-B\nabla_{i}\psi\\
		=&c_{s}^{2}\delta^{bc}\partial_{b}\partial_{c}u^{i}+c_{s}^{2}\delta^{ia}\partial_{c}\left( \partial_{a}u^{c}-\partial_{c}u^{a}\right)+c_{s}^{2}\left( \partial_{i}u^{b}-\partial_{b}u^{i}\right)\partial_{b}\bm{\uprho}\\
		&+c_{s}^{2}\delta^{ab}\partial_{a}u^{i}\partial_{b}\bm{\uprho}-2c_{s}c'_{s}B\bm{\uprho}\delta^{ia}\partial_{a}\bm{\uprho}-B\nabla_{i}\psi\\
		=&c_{s}^{2}\delta^{bc}\partial_{b}\partial_{c}u^{i}+c_{s}^{2}\delta^{ab}\partial_{a}u^{i}\partial_{b}\bm{\uprho}-2c_{s}c'_{s}B\bm{\uprho}\delta^{ia}\partial_{a}\bm{\uprho}-B\nabla_{i}\psi,        
		\end{split}
		\end{align}
		because \(\textrm{curl}u=0\).
		Next, substituting the RHS \(\eqref{2.17}\) for the term \(-BBu^{i}\) on RHS \(\eqref{2.16}\), we arrive at
		\begin{align}\label{2.18}
		\begin{split}
		\square_{g}u^{i}=\left\{-c_{s}^{2}\delta^{ab}\left( \partial_{b}\bm{\uprho}\right) \partial_{a}u^{i}-\left( \partial_{a}u^{a}\right) Bu^{i} \right\}-c_{s}^{-1}c'_{s}\left( g^{-1}\right) ^{\alpha\beta}\partial_{\alpha}\bm{\uprho}\partial_{\beta}u^{i}+B\nabla_{i}\psi-2c_{s}^{-1}c'_{s}(B\bm{\uprho})\nabla_{i}\psi.
		\end{split}
		\end{align}
		To handle the terms \(\left\{\cdot\right\}\) in \(\eqref{2.18}\), we use \(\eqref{2.1}\),\(\eqref{2.2}\) and \(\eqref{2.4}\) to obtain
		\begin{align}\label{2.19}
		-c_{s}^{2}\delta^{ab}\left( \partial_{b}\bm{\uprho}\right) \partial_{a}u^{i}-\left( \partial_{a}u^{a}\right) Bu^{i}=-c_{s}^{2}\delta^{ab}\left( \partial_{b}\bm{\uprho}\right) \partial_{a}u^{i}+\left( B\bm{\uprho}\right) Bu^{i}=-\left( g^{-1}\right) ^{\alpha\beta}\partial_{\alpha}\bm{\uprho}\partial_{\beta}u^{i}.
		\end{align}
		Finally, substituting \(\eqref{2.19}\) into \(\eqref{2.18}\), we conclude the desired equation \(\eqref{2.15}\).
	\end{proof}
\end{lemma}
We now derive equation for \(\varepsilon\).
\begin{lemma}[\(\textbf{Wave equation for}\) \(\varepsilon\)]
	The compressible Euler-Poisson equations \(\eqref{2.1}-\eqref{2.2}\) imply the following covariant wave equation for the logarithmic density perturbation variable \(\varepsilon\):
	\begin{align}\label{2.20}
	\begin{split}
	\square_{g}\varepsilon=&-3c_{s}^{-1}c'_{s}\left( g^{-1}\right) ^{\alpha\beta}\partial_{\alpha}\bm{\uprho}\partial_{\beta}\varepsilon+2\sum_{1\le a<b\le 3}\left\{\partial_{a}u^{a}\partial_{b}u^{b}-\partial_{a}u^{b}\partial_{b}u^{a}\right\}\\
	&-c_{s}^{2}\delta^{ab}\partial_{a}\partial_{b}\bar{\bm{\uprho}}-2c_{s}^{2}c'_{s}\delta^{ab}(\partial_{a}\bm{\uprho})(\partial_{b}\bar{\bm{\uprho}})-\Delta\psi.
	\end{split}
	\end{align}
	\begin{proof}
		First,using \(\eqref{2.11}\) with \(\phi=\varepsilon\) and equation \(\eqref{2.1}\), we compute that
		\begin{align}\label{2.21}
		\square_{g}\varepsilon=-BB\varepsilon+c_{s}^{2}\delta^{ab}\partial_{a}\partial_{b}\varepsilon+2c_{s}^{-1}c'_{s}(B\varepsilon) (B\varepsilon)+\left( \partial_{a}u^{a}\right) ^{2}-c_{s}^{-1}c'_{s}\left( g^{-1}\right) ^{\alpha\beta}\partial_{\alpha}\bm{\uprho}\partial_{\beta}\varepsilon.
		\end{align}
		Next, we use \(\eqref{2.1}-\eqref{2.2}\) to compute that
		\begin{align}\label{2.22}
		\begin{split}
		BB\varepsilon=&-\partial_{a}\left( Bu^{a}\right) +\left( \partial_{a}u^{b}\right) \partial_{b}u^{a}\\
		=&c_{s}^{2}\delta^{ab}\partial_{a}\partial_{b}\bm{\uprho}+\delta^{ab}\left( \partial_{a}c_{s}^{2}\right) \partial_{b}\bm{\uprho}+\left( \partial_{a}u^{b}\right) \partial_{b}u^{a}+\Delta\psi\\
		=&c_{s}^{2}\delta^{ab}\partial_{a}\partial_{b}\bm{\uprho}+2c_{s}c'_{s}\delta^{ab}\partial_{a}\bm{\uprho}\partial_{b}\bm{\uprho}+\left( \partial_{a}u^{b}\right) \partial_{b}u^{a}+\Delta\psi.
		\end{split}
		\end{align}
		Finally, using \(\eqref{2.22}\) to substitute for the term \(-BB\bm{\uprho}\) on RHS \(\eqref{2.21}\) and using the identities
		\begin{align}\label{2.23}
		\left( \partial_{a}u^{a}\right) ^{2}-\left( \partial_{a}u^{b}\right) \partial_{b}u^{a}=2\sum_{1\le a<b\le 3}\left\{\partial_{a}u^{a}\partial_{b}u^{b}-\partial_{a}u^{b}\partial_{b}u^{a}\right\}
		\end{align}
		and \(B\bm{\uprho}B\varepsilon-c_{s}^{2}\delta^{ab}\partial_{a}\bm{\uprho}\partial_{b}\varepsilon=-\left( g^{-1}\right) ^{\alpha\beta}\partial_{\alpha}\bm{\uprho}\partial_{\beta}\varepsilon\), , we arrive at the desired expression \(\eqref{2.20}\).	
	\end{proof}	
\end{lemma}
It's worth noting that there is no zeroth order term on RHS \(\eqref{2.20}\) because \(\bar{\bm{\uprho}}\) satisfies the steady equation. We now establish equation for \(B\varepsilon\).
\begin{lemma}[\(\textbf{Wave equation for}\) \(B\varepsilon\)]
	The compressible Euler-Poisson equations \(\eqref{2.1}-\eqref{2.2}\) imply the following covariant wave equation for \(B\varepsilon\):
	\begin{align}\label{2.24}
	\begin{split}
		\square_{g}B\varepsilon=& -(1+3c_{s}^{-1}c'_{s})\left( g^{-1}\right) ^{\alpha\beta}\left( \partial_{\alpha}\bm{\uprho}\right) \left( \partial_{\beta}B\varepsilon\right)-2c_{s}c'_{s}\delta^{ab}\left( \partial_{a}\bm{\uprho}\right) \left( \partial_{b}B\varepsilon\right)+2c_{s}^{2}\left( \partial_{a}\partial_{b}\bm{\uprho}\right) \left( \partial_{a}u^{b}\right)\\
        &-2c_{s}c'_{s}\delta^{ab}\left( B\varepsilon\right) \left( \partial_{a}\partial_{b}\bm{\uprho}\right)+2c_{s}^{2}\left( \partial_{a}\partial_{b}\bm{\uprho}\right) \left( \partial_{a}u^{b}\right)-2\left( c'_{s}c'_{s}+c_{s}c'\!'_{s}\right) \delta^{ab}\left( B\varepsilon\right) \left( \partial_{a}\bm{\uprho}\right) \left( \partial_{b}\bm{\uprho}\right)\\ 
        &+8c_{s}c'_{s}\left( \partial_{a}\bm{\uprho}\right) \left( \partial_{b}\bm{\uprho}\right) \left( \partial_{a}u^{b}\right)-B\Delta\psi +2\left( \partial_{a}\partial_{b}\psi\right) \left( \partial_{a}u^{b}\right)
        +2\left( \partial_{a}u^{b}\right) \left( \partial_{b}u^{c}\right) \left( \partial_{c}u^{a}\right).
	\end{split}
	\end{align}
	\begin{proof}
		First,using \(\eqref{2.11}\) with \(\phi=B\varepsilon\) to deduce
		\begin{align}\label{2.25}
		\square_{g}B\varepsilon=-BBB\varepsilon+c_{s}^{2}\delta^{ab}\partial_{a}\partial_{b}B\varepsilon+2c_{s}^{-1}c'_{s}\left( B\varepsilon\right) \left( BB\varepsilon\right) -\left( \partial_{a}u^{a}\right) BB\varepsilon-c_{s}^{-1}c'_{s}\left( g^{-1}\right) ^{\alpha\beta}\left( \partial_{\alpha}\bm{\uprho}\right) \left( \partial_{\beta}B\varepsilon\right) .
		\end{align}
		Next, commuting equation \(\eqref{2.22}\) with the operator \(B\), we get
		\begin{align}\label{2.26}
		\begin{split}
		BBB\varepsilon=&2c_{s}c'_{s}\delta^{ab}\left( B\varepsilon\right) \left( \partial_{a}\partial_{b}\bm{\uprho}\right) +c_{s}^{2}\delta^{ab}\left( B\partial_{a}\partial_{b}\bm{\uprho}\right) +2\left( c'_{s}c'_{s}+c_{s}c'\!'_{s}\right) \delta^{ab}\left( B\varepsilon\right) \left( \partial_{a}\bm{\uprho}\right) \left( \partial_{b}\bm{\uprho}\right) \\
		&+B\Delta\psi+\left\{4c_{s}c'_{s}\delta^{ab}\left( B\partial_{a}\bm{\uprho}\right) \left( \partial_{b}\bm{\uprho}\right) +2\left( B\partial_{a}u^{b}\right) \left( \partial_{b}u^{a}\right) \right\}.
		\end{split}
		\end{align}
		Next, using the identity
		\begin{align}\label{2.27}
		\partial_{a}B\phi=B\partial_{a}\phi+\left( \partial_{a}u^{c}\right) \left( \partial_{c}\phi\right) 
		\end{align}
		with \(\phi=\bm{\uprho}\) and \(\phi=u^{b}\) respectively, we deduce the terms \(\left\{\cdot\right\}\) in \(\eqref{2.26}\) and obtain
		\begin{align}\label{2.28}
		\begin{split}
		\left\{\cdot\right\}=&4c_{s}c'_{s}\delta^{ab}\left( \partial_{a}B\varepsilon\right) \left( \partial_{b}\bm{\uprho}\right) -4c_{s}c'_{s}\delta^{ab}\left( \partial_{a}u^{c}\right) \left( \partial_{c}\bm{\uprho}\right) \left( \partial_{b}\bm{\uprho}\right)+2\left( \partial_{a}Bu^{b}\right) \left( \partial_{b}u^{a}\right) -2\left( \partial_{a}u^{b}\right) \left( \partial_{b}u^{c}\right) \left( \partial_{c}u^{a}\right)
		\end{split}
		\end{align}
		Next, using the identity
		\begin{align}\label{2.29}
		\partial_{a}\partial_{b}B\varepsilon=B\partial_{a}\partial_{b}\bm{\uprho}+\left( \partial_{a}\partial_{b}u^{c}\right) \left( \partial_{c}\bm{\uprho}\right)+\left( \partial_{a}u^{c}\right) \left( \partial_{b}\partial_{c}\bm{\uprho}\right)+\left( \partial_{b}u^{c}\right) \left( \partial_{a}\partial_{c}\bm{\uprho}\right),
		\end{align}
		we get
		\begin{align}\label{2.30}
		\begin{split}
		c_{s}^{2}\delta^{ab}\partial_{a}\partial_{b}B\varepsilon=&c_{s}^{2}\delta^{ab}\left( B\partial_{a}\partial_{b}\bm{\uprho}\right) +c_{s}^{2}\delta^{ab}\left( \partial_{a}\partial_{b}u^{c}\right) \left( \partial_{c}\bm{\uprho}\right) +2c_{s}^{2}\delta^{ab}\left( \partial_{a}u^{c}\right) \left( \partial_{b}\partial_{c}\bm{\uprho}\right) \\
		=&c_{s}^{2}\delta^{ab}\left( B\partial_{a}\partial_{b}\bm{\uprho}\right) +c_{s}^{2}\left( \Delta u^{c}\right) \left( \partial_{c}\bm{\uprho}\right) +2c_{s}^{2}\left( \partial_{a}u^{b}\right) \left( \partial_{a}\partial_{b}\bm{\uprho}\right).
		\end{split}
		\end{align}
		Next, we use \(\eqref{2.25},\eqref{2.26},\eqref{2.28}\) and \(\eqref{2.30}\) to derive the equation
		\begin{align}\label{2.31}
		\begin{split}
		\square_{g}B\varepsilon=& -c_{s}^{-1}c'_{s}\left( g^{-1}\right) ^{\alpha\beta}\left( \partial_{\alpha}\bm{\uprho}\right) \left( \partial_{\beta}B\varepsilon\right)-4c_{s}c'_{s}\delta^{ab}\left( \partial_{a}\bm{\uprho}\right) \left( \partial_{b}B\varepsilon\right)+c_{s}^{2}\left( \triangle u^{c}\right) \left( \partial_{c}\bm{\uprho}\right)\\
		&-2c_{s}c'_{s}\delta^{ab}\left( B\varepsilon\right) \left( \partial_{a}\partial_{b}\bm{\uprho}\right)+2c_{s}^{2}\left( \partial_{a}\partial_{b}\bm{\uprho}\right) \left( \partial_{a}u^{b}\right)-2\left( c'_{s}c'_{s}+c_{s}c'\!'_{s}\right) \delta^{ab}\left( B\varepsilon\right) \left( \partial_{a}\bm{\uprho}\right) \left( \partial_{b}\bm{\uprho}\right)\\ 
		&+4c_{s}c'_{s}\left( \partial_{a}\bm{\uprho}\right) \left( \partial_{b}\bm{\uprho}\right) \left( \partial_{a}u^{b}\right)-B\Delta\psi \\
		&-2\left( \partial_{a}Bu^{b}\right) \left( \partial_{b}u^{a}\right)+2\left( \partial_{a}u^{b}\right) \left( \partial_{b}u^{c}\right) \left( \partial_{c}u^{a}\right)+2c_{s}^{-1}c'_{s}\left( B\varepsilon\right) \left( BB\varepsilon\right) +\left( B\varepsilon\right) \left( BB\varepsilon\right) .
		\end{split}
		\end{align}
		Finally, noticing the identity 
		\begin{align}\label{2.32}
		\begin{split}
		-2\left( \partial_{a}Bu^{b}\right) \left( \partial_{b}u^{a}\right)=
		2c_{s}^{2}\left( \partial_{a}\partial_{b}\bm{\uprho}\right) \left( \partial_{a}u^{b}\right) +4c_{s}c'_{s}\left( \partial_{a}\bm{\uprho}\right) \left( \partial_{b}\bm{\uprho}\right) \left( \partial_{a}u^{b}\right)+2\left( \partial_{a}\partial_{b}\psi\right) \left( \partial_{a}u^{b}\right)  ,
		\end{split}
		\end{align}
		according to \(\textrm{curl}(\textrm{curl}u)=\nabla\left( \nabla\cdot u\right) -\Delta u=0\), and \(\left( B\bm{\uprho}\right) \left( BB\varepsilon\right) -c_{s}^{2}\delta^{ab}\left( \partial_{a}\bm{\uprho}\right) \left( \partial_{b}B\varepsilon\right) =-\left( g^{-1}\right) ^{\alpha\beta}\partial_{\alpha}\bm{\uprho}\partial_{\beta}B\varepsilon\), we arrive at the desired expression \(\eqref{2.24}\).
	\end{proof}
\end{lemma}
In conclusion, we reformulate the compressible Euler-Poisson equation in the following quasilinear form:
\begin{equation}\label{2.33}
\left\{
\begin{aligned}
&\left(B^{2}+a\nabla_{n}^{\left( g\right) }\right)u^{i}=-c_{s}^{2}\nabla_{i} B\bm{\uprho}-B\nabla_{i}\psi\quad \textrm{on}\ \partial\calB(t),\\
&\square_{g}u^{i}=-\left(1+c_{s}^{-1}c'_{s}\right) \left( g^{-1}\right) ^{\alpha\beta}\partial_{\alpha}\bm{\uprho}\partial_{\beta}u^{i}+B\nabla_{i}\psi-2c_{s}^{-1}c'_{s}(B\bm{\uprho})\nabla_{i}\psi\quad \textrm{in}\ \calB(t),
\end{aligned}
\right.
\end{equation}
and
\begin{equation}\label{2.34}
\left\{
\begin{aligned}
\square_{g}\varepsilon=&-3c_{s}^{-1}c'_{s}\left( g^{-1}\right) ^{\alpha\beta}\partial_{\alpha}\bm{\uprho}\partial_{\beta}\varepsilon+\mathscr{Q}\quad\textrm{in}\ \calB(t),\quad\quad\varepsilon,\ B\varepsilon=0\ \ \textrm{on}\ \ \partial\calB(t), \\
\square_{g}B\varepsilon=&-(1+3c_{s}^{-1}c'_{s})\left( g^{-1}\right) ^{\alpha\beta}\left( \partial_{\alpha}\bm{\uprho}\right) \left( \partial_{\beta}B\varepsilon\right)-2c_{s}c'_{s}\delta^{ab}\left( \partial_{a}\bm{\uprho}\right) \left( \partial_{b}B\varepsilon\right)+\mathscr{P}\quad \textrm{in}\ \calB(t),
\end{aligned}
\right.
\end{equation}
where \(a\), \(\mathscr{Q}\) and \(\mathscr{P}\) are defined by
\begin{align}\label{lower order terms}
	\begin{split}
a:=&\sqrt{\partial_{\alpha}\bm{\uprho}\partial^{\alpha}\bm{\uprho}},\\
\mathscr{Q}:=&-c_{s}^{2}\delta^{ab}\partial_{a}\partial_{b}\bar{\bm{\uprho}}-2c_{s}^{2}c'_{s}\delta^{ab}(\partial_{a}\bm{\uprho})(\partial_{b}\bar{\bm{\uprho}})-\Delta\psi+2\sum_{1\le a<b\le 3}\left\{\partial_{a}u^{a}\partial_{b}u^{b}-\partial_{a}u^{b}\partial_{b}u^{a}\right\}\\
\mathscr{P}:=&-2c_{s}c'_{s}\delta^{ab}\left( B\varepsilon\right) \left( \partial_{a}\partial_{b}\bm{\uprho}\right)+2c_{s}^{2}\left( \partial_{a}\partial_{b}\bm{\uprho}\right) \left( \partial_{a}u^{b}\right)-2\left( c'_{s}c'_{s}+c_{s}c'\!'_{s}\right) \delta^{ab}\left( B\varepsilon\right) \left( \partial_{a}\bm{\uprho}\right) \left( \partial_{b}\bm{\uprho}\right)-B\Delta\psi\\ 
&+8c_{s}c'_{s}\left( \partial_{a}\bm{\uprho}\right) \left( \partial_{b}\bm{\uprho}\right) \left( \partial_{a}u^{b}\right) +2\left( \partial_{a}\partial_{b}\psi\right) \left( \partial_{a}u^{b}\right)
+2c_{s}^{2}\left( \partial_{a}\partial_{b}\bm{\uprho}\right) \left( \partial_{a}u^{b}\right) +2\left( \partial_{a}u^{b}\right) \left( \partial_{b}u^{c}\right) \left( \partial_{c}u^{a}\right).
\end{split}
\end{align}
We write the right hand side of equation \(\eqref{2.33}\) and \(\eqref{2.34}\) as the sum of the main linear terms and remainder, which \(\mathscr{Q}\) and \(\mathscr{P}\) include the lower order linear and nonlinear terms. In the next section we will discuss the necessary analytic tools to resolve above quasilinear system.

\section{Energy inequality and higher order equations}\label{3}
In this section we first consider the following model
\begin{equation}\label{3.1}
\left\{
\begin{aligned}
&\left(B^{2}+an\right)\phi=f\\
&\square_{g}\phi=g
\end{aligned}
\right.
\end{equation}
here we use the notation \(n\phi\) for \(\nabla_{n}^{\left( g\right) }\phi\) and \(a\) is as in \(\eqref{2.5}\). In subsection \(\ref{222}\) we use \(\nabla\) for \(\nabla^{\left( g\right) }\).

To obtain an energy inequality for the wave operator \(\square_{g}\), we start as usual by choosing a vector field \(X\) (the multiplier) and writing
\begin{align}\label{3.2}
\left( \square_{g}\phi\right) \left( X\phi\right) =\div P +q.
\end{align}
Here \(P\) is an appropriate field whose coefficients are quadratic forms in the components of \(\nabla\phi\), and \(q\) is a quadratic form in these components with
variable coefficients. Then by integrating \(\left( \square_{g}\phi\right) \left( X\phi\right) \) in some spacetime domain \(\mathcal{D}\), and using the Stokes formula, we can obtain boundary terms
\begin{align*}
\int_{\partial\mathcal{D}}\left\langle P,N\right\rangle dv,
\end{align*}
where \(N\) is  unit outside normal, which yield the energy of \(\phi\). We shall applying the above energy identity to our setting.
\subsection{Energy inequality}\label{222}
The energy-momentum tensor \(Q\) is a symmetric 2-tensor defined by 
\begin{align}\label{Q}
Q\left( X,Y\right) =\left( X\phi\right) \left( Y\phi\right) -\frac{1}{2}\left\langle X,Y \right\rangle |\nabla\phi|^{2},\quad Q_{\alpha\beta}=\left( \partial_{\alpha}\phi\right)\left( \partial_{\beta}\phi\right)-\frac{1}{2}g_{\alpha\beta}|\nabla\phi|^{2},
\end{align}
where \(|\nabla\phi|^{2}=\partial_{\alpha}\phi\partial^{\alpha}\phi\), and the deformation tensor of a given vector field \(X\) is the symmetric
2-tensor \(^{\left( X\right)}\pi \) defined by
\begin{align}
^{\left( X\right)}\pi\left( Y,Z\right) =\left\langle D_{Y}X,Z\right\rangle +\left\langle D_{Z}X,Y\right\rangle ,\quad ^{\left( X\right)}\pi_{\alpha\beta}=D_{\alpha}X_{\beta}+D_{\beta}X_{\alpha}.
\end{align}
Next, we introduce a key formula.
\begin{lemma}
	Let \(\phi\) be a given \(C^{2}\) function and \(Q\) be the associated energy–momentum tensor. Let \(X\) be a vector field, and set \(P_{\alpha}=Q_{\alpha\beta}X^{\beta}=\left( X\phi\right) \left( \partial_{\alpha}\phi\right) -\frac{1}{2}X_{\alpha}|\nabla\phi|^{2}\). Then
	\begin{align}\label{key formula}
	\div P =\left( \square_{g}\phi\right)\left( X\phi\right) +\frac{1}{2}Q^{\alpha\beta\left( X\right) }\pi_{\alpha\beta}.
	\end{align}
\end{lemma}
To compute \(^{\left( X\right)}\pi\), we utilize the following formula
\begin{align}\label{pi}
^{\left( X\right)}\pi^{\alpha\beta}=\partial^{\alpha}\left( X^{\beta}\right) +\partial^{\beta}\left( X^{\alpha}\right) -X\left( g^{\alpha\beta}\right) .
\end{align}

Now, using \(\eqref{Q}\), \(\eqref{key formula}\) and \(\eqref{pi}\), we compute that
\begin{align}\label{divP}
\begin{split}
\div P =&\left( \square_{g}\phi\right)\left( X\phi\right) +\frac{1}{2}\left[ \left( \partial_{\alpha}\phi\right) \left( \partial_{\beta}\phi\right) -\frac{1}{2}g_{\alpha\beta}|\nabla\phi|^{2}\right]  \left[ \partial^{\alpha}\left( X^{\beta}\right) +\partial^{\beta}\left( X^{\alpha}\right) -X\left( g^{\alpha\beta}\right)\right] \\
=&\left( \square_{g}\phi\right)\left( X\phi\right)+\partial^{\alpha}\left( X^{\beta}\right) \left( \partial_{\alpha}\phi\right) \left( \partial_{\beta}\phi\right) -\frac{1}{2}X\left( g^{\alpha\beta}\right) \left( \partial_{\alpha}\phi\right) \left( \partial_{\beta}\phi\right) -\frac{1}{2}\left( \partial_{\alpha}X^{\alpha}\right) |\nabla\phi|^{2}+\frac{1}{4}g_{\alpha\beta}X\left( g^{\alpha\beta}\right) |\nabla\phi|^{2}.
\end{split}
\end{align}
Noticing the identity
\begin{align}\label{Xnabla2}
|g|^{-\frac{1}{2}}X\left( |g|^{\frac{1}{2}} \right)=-\frac{1}{2}g_{\alpha\beta}X\left( g^{\alpha\beta}\right),
\end{align}
we arrive at the following identity
\begin{align}\label{key identity}
\left( \square_{g}\phi\right)\left( X\phi\right)=\div P+\frac{1}{2}\left( \partial_{\alpha}X^{\alpha}\right) |\nabla\phi|^{2}+\frac{1}{2}X\left(g^{\alpha\beta} \right)\left(\partial_{\alpha}\phi \right)\left(\partial_{\beta}\phi \right) -\partial^{\alpha}\left(X^{\beta} \right)\left(\partial_{\alpha}\phi \right)\left(\partial_{\beta}\phi \right)+\frac{1}{2}|g|^{-\frac{1}{2}}X\left( |g|^{\frac{1}{2}} \right) |\nabla\phi|^{2} ,
\end{align}
where \(|g|=|\det g|\). Then we establish  main energy identity for \(\eqref{3.1}\).
\begin{lemma}[\(\textbf{Energy identity}\)]\label{3.2.}
	Suppose \(\phi\) satisfies \(\eqref{3.1}\), then the following energy identity holds:
	\begin{align}\label{energy}
	\begin{split}
	&\int_{\calB(T)}\left( \left( B\phi\right) ^{2}+\frac{1}{2}|\nabla\phi|^{2}\right) dx+\int_{\partial\calB(T)}\frac{1}{2a}\left( B\phi\right) ^{2}dS\\
	&=\int_{\calB(0)}\left( \left( B\phi\right) ^{2}+\frac{1}{2}|\nabla\phi|^{2}\right) dx+\int_{\partial\calB(0)}\frac{1}{2a}\left( B\phi\right) ^{2}dS\\
	&+\int_{0}^{T}\int_{\partial\calB(t)}\frac{1}{a}f\left( B\phi\right) dSdt-\int_{0}^{T}\int_{\calB(t)}g\left( B\phi\right) dxdt-\int_{0}^{T}\int_{\partial\calB(t)}\frac{1}{2a^{2}}\left( Ba\right) \left( B\phi\right) ^{2}dSdt\\
	&+\int_{0}^{T}\int_{\partial\calB(t)}\frac{\div\mkern-17.5mu\slash\ \ B}{2a}\left( B\phi\right) ^{2}dSdt+\int_{0}^{T}\int_{\calB(t)}\frac{1}{2}B\left(g^{\alpha\beta} \right)\left(\partial_{\alpha}\phi \right)\left(\partial_{\beta}\phi \right) dxdt\\
	&-\int_{0}^{T}\int_{\calB(t)}\frac{1}{2}\left( B\bm{\uprho}\right) |\nabla\phi|^{2}dxdt-\int_{0}^{T}\int_{\calB(t)}\partial^{\alpha}\left(B^{\beta} \right)\left(\partial_{\alpha}\phi \right)\left(\partial_{\beta}\phi \right) dxdt+\int_{0}^{T}\int_{\calB(t)}\frac{1}{2}|g|^{-\frac{1}{2}}B\left( |g|^{\frac{1}{2}} \right) |\nabla\phi|^{2}dxdt,
	\end{split}
	\end{align}
	where \(\div\mkern-17.5mu\slash\  \) denotes the divergence operator on \(\partial\calB\).
\end{lemma}
\begin{proof}
	Multiplying the first equation in \(\eqref{3.1}\) by \(\frac{1}{a}\left( B\phi\right) \) we get
	\begin{align}\label{1/a B phi}
	\frac{1}{2}B\left[ \frac{1}{a}\left( B\phi\right) ^{2}\right] +\left( n\phi\right) \left( B\phi\right) =\frac{1}{a}f\left( B\phi\right) -\frac{1}{2a^{2}}\left( Ba\right) \left( B\phi\right) ^{2}.
	\end{align}
	we have
	\begin{align}\label{div}
	\int_{0}^{T}\int_{\partial\calB(t)}\left( \frac{1}{2}B\left( \frac{1}{a}\left( B\phi\right) ^{2}\right) +\frac{\div\mkern-17.5mu\slash\ \ B}{2a}\left( B\phi\right) ^{2}\right) dSdt=\int_{\partial\calB(T)}\frac{1}{2a}\left( B\phi\right) ^{2}dS-\int_{\partial\calB(0)}\frac{1}{2a}\left( B\phi\right) ^{2}dS.
	\end{align}
	Integrating \(\eqref{1/a B phi}\) over \(\partial\calB=\cup_{t\in[0,T]}\partial\calB(t)\) and using \(\eqref{div}\) we get
	\begin{align}\label{energy1}
	\begin{split}
	&\int_{\partial\calB(T)}\frac{1}{2a}\left( B\phi\right) ^{2}dS+\int_{0}^{T}\int_{\partial\calB(t)}\left( n\phi\right) \left( B\phi\right)dSdt\\
	&=\int_{\partial\calB(0)}\frac{1}{2a}\left( B\phi\right) ^{2}dS+\int_{0}^{T}\int_{\partial\calB(t)}\frac{1}{a}f\left( B\phi\right) dSdt\\
	&-\int_{0}^{T}\int_{\partial\calB(t)}\frac{1}{2a^{2}}\left( Ba\right) \left( B\phi\right) ^{2}dSdt+\int_{0}^{T}\int_{\partial\calB(t)}\frac{\div\mkern-17.5mu\slash\ \ B}{2a}\left( B\phi\right) ^{2}dSdt.
	\end{split}
	\end{align}
	To treat the second term on the left, we integrate \(\eqref{key identity}\) with \(X=B\) over \(\cup_{t\in[0,T]}\calB(t)\). Using the fact that \(B\) is tangent to \(\partial\calB\) and Stokes formula, we get
	\begin{align}\label{energy2}
	\begin{split}
	&\int_{\calB(T)}\left( \left( B\phi\right) ^{2}+\frac{1}{2}|\nabla\phi|^{2}\right) dx-\int_{0}^{T}\int_{\partial\calB(t)}\left( n\phi\right) \left( B\phi\right)dSdt\\
	&=\int_{\calB(0)}\left( \left( B\phi\right) ^{2}+\frac{1}{2}|\nabla\phi|^{2}\right) dx-\int_{0}^{T}\int_{\calB(t)}g\left( B\phi\right) dxdt+\int_{0}^{T}\int_{\calB(t)}\frac{1}{2}B\left(g^{\alpha\beta} \right)\left(\partial_{\alpha}\phi \right)\left(\partial_{\beta}\phi \right) dxdt\\
	&-\int_{0}^{T}\int_{\calB(t)}\frac{1}{2}\left( B\bm{\uprho}\right) |\nabla\phi|^{2}dxdt-\int_{0}^{T}\int_{\calB(t)}\partial^{\alpha}\left(B^{\beta} \right)\left(\partial_{\alpha}\phi \right)\left(\partial_{\beta}\phi \right) dxdt+\int_{0}^{T}\int_{\calB(t)}\frac{1}{2}|g|^{-\frac{1}{2}}B\left( |g|^{\frac{1}{2}} \right) |\nabla\phi|^{2}dxdt.
	\end{split}
	\end{align}
	The lemma follows by adding \(\eqref{energy2}\) to \(\eqref{energy1}\).
\end{proof}
Because the vectorfield \(B\) is timelike and future-directed, according to the positivity of energy-momentum tensor, the first term on the left of \(\eqref{energy}\) satisfies
\begin{align*}
\left( B\phi\right) ^{2}+\frac{1}{2}|\nabla\phi|^{2}\gtrsim\left| \partial_{t,x}\phi\right|^{2}.
\end{align*}
For the steady-state solution to \eqref{main eq}-\eqref{boundary conditions}, we have the following \emph{Taylor sign condition}
\begin{align}\label{Taylor sign}
\nabla_{\mathcal{N}}p\le-c_{0}<0,\quad on \quad \partial\calB ,\quad\quad where\quad \nabla_{\mathcal{N}}=\mathcal{N}^{a}\partial_{a}.
\end{align}
Since the solution we shall construct in this paper is a small perturbation of the steady-state solution up to time $T^{\delta}$, provided that $\theta_{0}$ is sufficiently small, the condition \eqref{Taylor sign} holds also for the perturbed solution. Therefore \(a\) must be positive and the left hand side of the energy identity \(\eqref{energy}\) controls
\begin{align*}
\int_{\partial\calB(T)}\frac{1}{a}\left| B\phi\right| ^{2}dS+\int_{\calB(T)}\left| \partial_{t,x}\phi\right| ^{2}dx.
\end{align*}
Next, we give two energy estimates for the wave operator:
\begin{lemma}\label{3.3.}
	There is a (future-directed and timelike) vectorfield \(Q\) such that for any \(\phi\) which is constant on \(\partial\calB\),
	\begin{align}\label{estimate1}
	\begin{split}
	&\int_{\calB\left( T\right) }\left| \partial_{t,x}\phi\right|^{2}dx+\int_{0}^{T}\int_{\partial\calB\left( t\right)}\left| \partial_{t,x}\phi\right|^{2}dSdt\\
	&\lesssim\int_{\calB\left( 0\right) }\left| \partial_{t,x}\phi\right|^{2}dx+\left| \int_{0}^{T}\int_{\calB\left( t\right) } \left( \square_{g}\phi\right) \left( Q\phi\right)  dxdt\right| \\
	&+\int_{0}^{T}\int_{\calB\left( t\right) }\left| \frac{1}{2}\left( \partial_{\alpha}Q^{\alpha}\right) |\nabla\phi|^{2}+\frac{1}{2}Q\left(g^{\alpha\beta} \right)\left(\partial_{\alpha}\phi \right)\left(\partial_{\beta}\phi \right) -\partial^{\alpha}\left(Q^{\beta} \right)\left(\partial_{\alpha}\phi \right)\left(\partial_{\beta}\phi \right)+\frac{1}{2}|g|^{-\frac{1}{2}}Q\left( |g|^{\frac{1}{2}} \right) |\nabla\phi|^{2} \right|  dxdt.
	\end{split}
	\end{align}
\end{lemma}
\begin{proof}
	Integrating \(\eqref{key identity}\) with \(X=Q\) over \(\cup_{t\in[0,T]}\calB(t)\), we get
	\begin{align}\label{energyQ}
	\begin{split}
	&\int_{\calB(T)}\left( \left( Q\phi\right) \left( B\phi\right) +\frac{1}{2}Q^{0}|\nabla\phi|^{2}\right) dx-\int_{0}^{T}\int_{\partial\calB(t)}n_{\alpha}\left( \left( Q\phi\right) \left( \partial^{\alpha}\phi\right) -\frac{1}{2}Q^{\alpha}|\nabla\phi|^{2}\right) dSdt\\
	=&\int_{\calB(0)}\left( \left( Q\phi\right) \left( B\phi\right) +\frac{1}{2}Q^{0}|\nabla\phi|^{2}\right) dx-\int_{0}^{T}\int_{\calB(t)}\left( \square_{g}\phi\right) \left( Q\phi\right) dxdt+\int_{0}^{T}\int_{\calB(t)}\frac{1}{2}Q\left(g^{\alpha\beta} \right)\left(\partial_{\alpha}\phi \right)\left(\partial_{\beta}\phi \right) dxdt\\
	&+\int_{0}^{T}\int_{\calB(t)}\frac{1}{2}\left(\partial_{\alpha}Q^{\alpha}\right) |\nabla\phi|^{2}dxdt-\int_{0}^{T}\int_{\calB(t)}\partial^{\alpha}\left(Q^{\beta} \right)\left(\partial_{\alpha}\phi \right)\left(\partial_{\beta}\phi \right) dxdt+\int_{0}^{T}\int_{\calB(t)}\frac{1}{2}|g|^{-\frac{1}{2}}Q\left( |g|^{\frac{1}{2}} \right) |\nabla\phi|^{2}dxdt.
	\end{split}
	\end{align}
	Since \(\phi\) is constant on \(\partial\calB\) , \(\partial_{\alpha}\phi\partial^{\alpha}\phi=\pm\sqrt{\partial_{\alpha}\phi\partial^{\alpha}\phi}\left( n\phi\right)\). We get
	\begin{align}
	\partial_{\alpha}\phi\partial^{\alpha}\phi=\left( n\phi\right) ^{2}.
	\end{align}
	and
	\begin{align}
	n_{\alpha}\left( \left( Q\phi\right) \left( \partial^{\alpha}\phi\right) -\frac{1}{2}Q^{\alpha}\left( \partial_{\beta}\phi\right) \left( \partial^{\beta}\phi\right)\right)=\frac{1}{2}Q^{n}\left( n\phi\right) ^{2}  
	\end{align}
	on \(\partial\calB\), where \(Q^{n}=\left\langle Q,n\right\rangle\). Therefore letting \(Q\) be a  future-directed timelike vectorfield with \(Q^{n}<0\) in \(\eqref{energyQ}\) (for instance \(Q= \nu B-n\) for some large \(\nu\)) we arrive at \(\eqref{estimate1}\).
\end{proof}
This energy estimate can be used for the second equation in \(\eqref{2.34}\).
\begin{lemma}\label{3.4.}
	There exists a (future-directed and timelike) vectorfield \(Q\) such that for any function \(\phi\),
	\begin{align}\label{estimate2}
	\begin{split}
	&\int_{\calB\left( T\right) }\left| \partial_{t,x}\phi\right|^{2}dx+\int_{0}^{T}\int_{\partial\calB\left( t\right)}\left| \partial_{t,x}\phi\right|^{2}dSdt\\
	&\lesssim\int_{\calB\left( 0\right) }\left| \partial_{t,x}\phi\right|^{2}dx+\left| \int_{0}^{T}\int_{\calB\left( t\right) } \left( \square_{g}\phi\right) \left( Q\phi\right)  dxdt\right| +\int_{0}^{T}\int_{\partial\calB\left( t\right) }\left( \left( n\phi\right) ^{2}+\left( B\phi\right) ^{2}\right) dSdt\\
	&+\int_{0}^{T}\int_{\calB\left( t\right) }\left| \frac{1}{2}\left( \partial_{\alpha}Q^{\alpha}\right) |\nabla\phi|^{2}+\frac{1}{2}Q\left(g^{\alpha\beta} \right)\left(\partial_{\alpha}\phi \right)\left(\partial_{\beta}\phi \right) -\partial^{\alpha}\left(Q^{\beta} \right)\left(\partial_{\alpha}\phi \right)\left(\partial_{\beta}\phi \right)+\frac{1}{2}|g|^{-\frac{1}{2}}Q\left( |g|^{\frac{1}{2}} \right) |\nabla\phi|^{2} \right|  dxdt.
	\end{split}
	\end{align}
\end{lemma}
\begin{proof}
	We choose \(Q^{n}=\left\langle Q,n\right\rangle>0\). For instance, let \(Q=\nu B+n\) with \(\nu>0\) chosen so that \(Q\) is future-directed and timelike. Then on \(\partial\calB\), 
	\begin{align*}
	\begin{split}
	&-n_{\alpha}\left( \left( Q\phi\right) \left( \partial^{\alpha}\phi\right) -\frac{1}{2}Q^{\alpha}\left( \partial_{\beta}\phi\right) \left( \partial^{\beta}\phi\right)\right)\\
	=&-\left(\nu B\phi\right) \left( n\phi\right) -\left( n\phi\right) ^{2}+\frac{1}{2}\partial_{\alpha}\phi\partial^{\alpha}\phi
	\ge c_{1}\left| \partial_{t,x}\phi\right|^{2}-c_{2}\left( \left( n\phi\right) ^{2}+\left( B\phi\right) ^{2}\right) 
	\end{split}
	\end{align*}
	for some constants \(c_{1}, c_{2}>0\) depending only on \(B\). The lemma follows by \(\eqref{energyQ}\) with \(Q=\nu B+n\).
\end{proof}
This energy estimate can be used for controlling arbitrary derivatives of an arbitrary function on the
boundary in terms of the normal and material derivatives.
\begin{lemma}\label{3.5.}
	For any \(\phi\),
	\begin{align}\label{3.20}
	c_{s}^{2}\Delta\phi=\square_{g}\phi+BB\phi-\left( 3c_{s}^{-1}c'_{s}+1\right) B\bm{\uprho}B\phi+c_{s}c'_{s}\delta^{ab}\partial_{a}\bm{\uprho}\partial_{b}\phi.
	\end{align}
\end{lemma}
\begin{proof}
This is a direct consequence of \(\eqref{2.11}\).
\end{proof}
 In order to apply Lemma \(\ref{3.5.}\), we introduce the following standard elliptic estimates, whose proof can be found in \cite{Taylor-book1}:
\begin{lemma}\label{3.6.}
	For any \(t>0,k=0,1\), we have
	\begin{align}
	\Arrowvert\phi\Arrowvert_{H^{k+1}\left( \calB_{t}\right) }\lesssim\Arrowvert\Delta\phi\Arrowvert_{L^{2}\left( \calB_{t}\right) }+\Arrowvert\phi\Arrowvert_{H^{k+\frac{1}{2}}\left( \partial\calB_{t}\right) },
	\end{align}
	and
	\begin{align}
	\Arrowvert\phi\Arrowvert_{H^{2}\left( \calB_{t}\right) }\lesssim\Arrowvert\Delta\phi\Arrowvert_{L^{2}\left( \calB_{t}\right) }+\Arrowvert N\phi\Arrowvert_{H^{\frac{1}{2}}\left( \partial\calB_{t}\right) },
	\end{align}
	where \(N\) is a transversal vectorfield to \(\partial\calB_{t}\subseteq\calB_{t}\), and where the implicit constants depend on \(\calB_{t}\).
\end{lemma}

\subsection{Higher order equations}Here we derive the higher order versions of \(\eqref{2.33}\) and \(\eqref{2.34}\). We give some important commutator identities, which are valid for any \(\phi\).
\begin{align}\label{commutator1}
&[B,\partial_{a}]\phi=-\left( \partial_{a}u^{b}\right) \left( \partial_{b}\phi\right) .\\\label{commutator2}
&[B,\partial_{a}\partial_{b}]\phi=-\left( \partial_{a}u^{c}\right) \left( \partial_{b}\partial_{c}\phi\right) -\left( \partial_{b}u^{c}\right) \left( \partial_{a}\partial_{c}\phi\right) -\left( \partial_{a}\partial_{b}u^{c}\right) \left( \partial_{c}\phi\right) \\
&[B,\Delta]\phi=-2\delta^{ab}\left( \partial_{a}u^{c}\right) \left( \partial_{b}\partial_{c}\phi\right) +\delta^{ab}\left(\partial_{a}B\bm{\uprho}\right) \left( \partial_{b}\phi\right)  .
\end{align}
\begin{align}
\begin{split}
[B,B^{2}-c_{s}^{2}\delta^{ab}\partial_{a}\bm{\uprho}\partial_{b}]\phi=&c_{s}^{2}\delta^{ab}\left( \partial_{a}\bm{\uprho}\right) \left( \partial_{b}u^{c}\right) \left( \partial_{c}\phi\right) +c_{s}^{2}\delta^{ab}\left( \partial_{c}\bm{\uprho}\right) \left( \partial_{b}u^{c}\right) \left( \partial_{a}\phi\right)\\
&-c_{s}^{2}\delta^{ab}\left( \partial_{a}B\bm{\uprho}\right) \left( \partial_{b}\phi\right),
\quad \textrm{on}\quad \partial\calB . 
\end{split}
\end{align}
\begin{align}
\begin{split}
[B,\square_{g}]\phi=&-c_{s}^{2}\delta^{ab}\left( \partial_{a}\partial_{b}u^{c}\right) \left( \partial_{c}\phi\right) -2c_{s}^{2}\delta^{ab}\left( \partial_{a}u^{c}\right) \left( \partial_{b}\partial_{c}\phi\right) +2c_{s}c'_{s}B\bm{\uprho}\delta^{ab}\partial_{a}\partial_{b}\phi\\
&+3\left( c_{s}^{-1}c''_{s}-c_{s}^{-2}c'_{s}c'_{s}\right) \left( B\bm{\uprho}\right) ^{2}B\phi+\left( 1+3c_{s}^{-1}c'_{s}\right) BB\bm{\uprho}B\phi\\
&-\left( c'_{s}c'_{s}+c_{s}c''_{s}\right) B\bm{\uprho}\delta^{ab}\partial_{a}\bm{\uprho}\partial_{b}\phi-c_{s}c'_{s}\delta^{ab}\left( \partial_{a}B\bm{\uprho}\right) \left( \partial_{b}\phi\right) \\ \label{3.27}
&+c_{s}c'_{s}\delta^{ab}\left( \partial_{a}\bm{\uprho}\right) \left( \partial_{b}u^{c}\right) \left( \partial_{c}\phi\right)+c_{s}c'_{s}\delta^{ab}\left( \partial_{c}\bm{\uprho}\right) \left( \partial_{b}u^{c}\right) \left( \partial_{a}\phi\right).
\end{split}
\end{align}
The above identities can be obtained by direct calculation. Applying \(\eqref{commutator1}-\eqref{3.27}\) we can calculate the higher order versions of \(\eqref{2.33}\) and \(\eqref{2.34}\), which we record in
the following lemmas.

\begin{lemma}\label{3.7.}
	For any \(k\ge 0\)
	\begin{align}\label{3.28}
	\left( B^{2}+an\right) B^{k}u=-c_{s}^{2}\nabla B^{k+1}\varepsilon+F_{k}
	\end{align}
	where \(F_{k}\) is a linear combination (coefficients are related to \(c_{s}\)) of terms of the forms\\
	\begin{enumerate}[\quad (1)]    		
		\item \(\left( \nabla B^{k_{1}}u\right)...\left( \nabla B^{k_{m}}u\right)\left( \nabla B^{k_{m+1}}\bm{\uprho}\right)\), where \(k_{1}+\cdot\cdot\cdot+k_{m+1}\le k-1\).\label{k1}
		\item \(\left( \nabla B^{k_{1}}u\right)...\left( \nabla B^{k_{m}}u\right)\left( \nabla B^{k_{m+1}}B\bm{\uprho}\right)\), where \(k_{1}+\cdot\cdot\cdot+k_{m+1}\le k-1\).\label{k2}
		\item \(\left( B^{k+1}\nabla\psi\right) \)\label{k3}
	\end{enumerate}
\end{lemma}
\begin{proof}
	We proceed inductively. For $k=0$ the statement already contained in the first equation in \eqref{2.33}. 
	\begin{align*}
	B\left( anB^{j}u\right)=B\left( -c_{s}^{2}\delta^{ab}\partial_{a}\bm{\uprho}\partial_{b}B^{j}u\right) =&anB^{j+1}u-c_{s}^{2}\delta^{ab}\left( \partial_{a}B\bm{\uprho}\right) \left( \partial_{b}B^{j}u\right)\\
	&+c_{s}^{2}\delta^{ab}\left( \partial_{b}u^{c}\right) \left( \left( \partial_{a}\bm{\uprho}\right) \left( \partial_{c}B^{j}u\right) +\left( \partial_{c}\bm{\uprho}\right) \left( \partial_{a}B^{j}u\right) \right)  ,
	\end{align*}
	so \([B,B^{2}+an]B^{j}u\) has the right form.  Next, in view of \(\eqref{commutator1}\), \(B\) applied to the terms in (1), (2) and (3) with \(k\) replaced by \(j\), as well as \(\nabla B^{j+1}\varepsilon\), also has the desired form.
\end{proof}
\begin{lemma}\label{3.8.}
	For any \(k\ge 0\)
	\begin{align}\label{3.29}
	\square_{g}B^{k}u=G_{k}
	\end{align}
	where \(G_{k}\) is a linear combination (coefficients are related to \(c_{s}\) and its derivatives) of terms of the forms\\
	\begin{enumerate}[\quad (1)]
		\item \(\left( \nabla B^{k_{1}}u\right)...\left( \nabla B^{k_{m}}u\right)\left( \nabla B^{k_{m+1}}\bm{\uprho}\right)\), where \(k_{1}+\cdot\cdot\cdot+k_{m+1}\le k\).\label{j1}
		\item \(\left( \nabla B^{k_{1}}u\right)...\left( \nabla B^{k_{m}}u\right)\left( \nabla^{(2)}B^{k_{m+1}}u\right) \), where \(k_{1}+\cdot\cdot\cdot+k_{m+1}\le k-1\).\label{j2}
		\item \(\left( \nabla B^{k_{1}}u\right)...\left( \nabla B^{k_{m}}u\right) \left( B^{k_{m+1}}\nabla\psi\right) \), where \(k_{1}+\cdot\cdot\cdot+k_{m}\le k\) and \(k_{1}+\cdot\cdot\cdot+k_{m+1}\le k+1\)\label{j3}
	\end{enumerate}
\end{lemma}
\begin{proof}
	Again we proceed inductively. For \(k=0\) the statement already contained in the second equation in \(\eqref{2.33}\). Assume it holds for \(k=j\) and let us prove it for \(k=j+1\). By computing, \([B,\square_{g}]B^{j}u\) has the right form contained in (1) and (2). Similarly, \(B\) applied to above forms have the desired forms by \(\eqref{commutator1}\) and \(\eqref{commutator2}\).
\end{proof}
\begin{lemma}\label{3.9.}
	For any \(k\ge 0\)
	\begin{align}\label{3.30}
	\square_{g}B^{k+1}\varepsilon=H_{k}
	\end{align}
	where \(H_{k}\) is a linear combination (coefficients are related to \(c_{s}\) and its derivatives) of terms of the forms\\
	\begin{enumerate}[\quad (1)]
		\item \(\left( \nabla B^{k_{1}}u\right)...\left( \nabla B^{k_{m}}u\right)\left( \nabla B^{k_{m+1}}\bm{\uprho}\right) \left( \nabla B^{k_{m+2}}B\bm{\uprho}\right)\), where \(k_{1}+\cdot\cdot\cdot+k_{m}\le k-1\) and \(k_{1}+\cdot\cdot\cdot+k_{m+2}\le k\).
		\label{i1}\item \(\left( \nabla B^{k_{1}}u\right)...\left( \nabla B^{k_{m}}u\right)\left( \nabla B^{k_{m+1}}\bm{\uprho}\right) \left( \nabla B^{k_{m+2}}\bm{\uprho}\right)\), where \(k_{1}+\cdot\cdot\cdot+k_{m+2}\le k\).
		\label{i2}\item \(\left( \nabla B^{k_{1}}u\right)...\left( \nabla B^{k_{m}}u\right)\left( \nabla^{\left( 2\right) } B^{k_{m+1}}\bm{\uprho}\right) \), where \(k_{1}+\cdot\cdot\cdot+k_{m+1}\le k\).
		\label{i3}\item \(\left( \nabla B^{k_{1}}u\right)...\left( \nabla B^{k_{m}}u\right)\left( \nabla B^{k_{m+1}}\bm{\uprho}\right) \left( \nabla^{\left( 2\right) } B^{k_{m+2}}u\right) \), where \(k_{1}+\cdot\cdot\cdot+k_{m+2}\le k-1\).
		\label{i4}\item \(\left( \nabla B^{k_{1}}u\right)...\left( \nabla B^{k_{m}}u\right)\), where \(k_{1}+\cdot\cdot\cdot+k_{m}\le k\).
		\label{i5}\item \(\left( \nabla B^{k_{1}}u\right)...\left( \nabla B^{k_{m}}u\right)\left(B^{k_{m+1}}\nabla\nabla\psi \right) \), where \(k_{1}+\cdot\cdot\cdot+k_{m+1}\le k\).	\label{i6}
		\item \(\left( B^{k+1}\Delta\psi\right) \).\label{i7}
	\end{enumerate}
\end{lemma}
\begin{proof}
	For \(k=0\) the statement already contained in the second equation in \(\eqref{2.34}\). Assume it holds for \(k=j\) and let us prove it for \(k=j+1\). By computing, \([B,\square_{g}]B^{j+1}\varepsilon\) has the right form. Similarly, \(B\) applied to above forms have the desired forms by \(\eqref{commutator1}\) and \(\eqref{commutator2}\).
\end{proof}
Compared to the system studied in \cite{MSW}, the gravitational potential \(\psi\) on the right-hand side of equation \(\eqref{3.28}-\eqref{3.30}\) will also affect the energy estimate. We shall use the Hilbert transform to control its contribution.

\section{Analytic tools}\label{4}
In this section we first recall basic algebraic properties of the Clifford algebra \(\mathcal{C}=Cl_{0,2}(\mathbb{R})\). The algebra \(\calC\) is the associative algebra generated by the four basis elements \((1,e_{1},e_{2},e_{3})\) over \(\bbR\), satisfying the relatins
\begin{align*}
1e_{i}=e_{i},\quad e_{i}e_{j}=-e_{j}e_{i},\quad i\ne j,\ i,j=1,2,3,\quad e_{1}e_{2}=e_{3},\quad e_{i}^{2}=-1,\ i=1,2,3.
\end{align*}
Every element \(a\in\calC\) has a unique representation \(a=a^{0}+\sum_{i=1}^{3}a^{i}e_{i}\). We identify the real numbers \(\bbR\) with Clifford numbers using the relation \(a\mapsto a1\), and identify vectors in \(\bbR^{3}\) with Clifford vectors using the relation \(\vec{a}\mapsto a^{i}e_{i}\). The Clifford differentiation operator \(\calD\), acting on Clifford algebra-valued functions, is defines as 
\begin{align*}
\calD=\sum_{i=1}^{3}\partial_{x^{i}}e_{i},
\end{align*}
where \(x=(x^{1},x^{2},x^{3})\) are the usual rectangular coordinates in \(\bbR^{3}\). Let \(\Omega\) be a \(C^{2}\), bounded, and simply-connected domain in \(\bbR^{3}\) with boundary \(\Sigma\). We say that a function \(f\) defined on \(\Omega\) is Clifford analytic, if \(\calD f=0\). It is straightforward to verify that a vector-valued function \(f\) is Clifford analytic is equivalent to \(f\) being curl and divergence free. For the convenience of following discussion, we introduce the Lagrangian parametrization of the surface. Let \(\xi:\bbR\times S_{R}\to\partial\calB\) be the Lagrangian parametrization of \(\partial\calB=\partial\calB(t)\), satisfying \(\xi(0,p)=p\) for all \(p\) in \(S_{R}\) and 
\begin{align}
\xi_{t}(t,p)=u(t,\xi(t,p)).
\end{align}
\(n(t,p)={\bfn}(t,\xi(t,p))\) denote the exterior unit normal to \(\partial\calB(t)\). In arbitrary (orientation preserving) local coordinates \((\alpha,\beta)\) on \(S_{R}\) we have
\begin{align}
n=\frac{N}{|N|},\quad where \quad N=\xi_{\alpha}\times\xi_{\beta}.
\end{align}
If \({\bf f}:\calB \to \bbR\) is a (possibly time-dependent) differentiable function, and \(f={\bff}\circ\xi\), then by a slight abuse of notation we write
\begin{align}
\nabla f=\left( \nabla{\bff}\right)\circ\xi,\quad df=\left(d{\bff} \right)  \circ\xi,
\end{align}
where \(d\) denotes the exterior differentiation operators on \(\partial\calB\). With this
notation, and using the fact that \(N=\xi_{\alpha}\times\xi_{\beta}\)
\begin{align}
n\times\nabla f:=\left( {\bfn}\times\nabla{\bff}\right) \circ\xi=\frac{\xi_{\beta}f_{\alpha}-\xi_{\alpha}f_{\beta}}{|N|}.
\end{align}
For a Clifford algebra-valued function \(f\) (possibly time-dependent) on \(\Sigma\) we define the Hilbert
transform of \(f\) as
\begin{align}
H_{\Sigma}f(\xi)=p.v.\int_{\Sigma}K(\xi'-\xi)n(\xi')f(\xi')dS(\xi'),\quad  \xi\in\Sigma,
\end{align} 
where 
\begin{align}
K(x):=-\frac{1}{2\pi}\frac{x}{|x|^{3}},\quad\quad x\in \bbR^{3},
\end{align}
and \(K(\xi'-\xi)n(\xi')f(\xi')\) are usual Clifford product. Then we introduce an important property of Hilbert transform described in the following Lemma.
\begin{lemma}
	\cite{clifford}\cite{wu1999well}If \(f\) is the restriction
	to \(\Sigma\) of a Clifford analytic function \(\bff\) defined in a neighborhood of \(\Omega\), then \(f(\xi)=H_{\Sigma}f(\xi)\). Similarly, if \(f\) is the restriction
	to \(\Sigma\) of a Clifford analytic function \(\bff\) defined in a neighborhood of \(\Omega^{c}\), then \(f(\xi)=-H_{\Sigma}f(\xi)\). Finally the operator \(H_{\Sigma}:L^{2}(\Sigma,dS)\to L^{2}(\Sigma,dS)\)  are bounded and linear.
\end{lemma}
In our application, we express the gravitational term on $\partial\calB(t)$ in terms of Hilbert transform (see \cite{MS}). By the boundedness of the operator \(H_{\Sigma}\), we can control the gravitational term on the boundary. We achieve this process through the following Lemma. 

\begin{lemma}\label{4.2.}
	Suppose the functions \(\phi:\bbR^{3}\to\bbR\) and \(\Phi:\overline{\calB}\to\bbR\) satisfy
	\begin{equation}\label{4.7}
	\Delta \phi=\left\{
	\begin{aligned}
	&\Delta \Phi \quad \quad in\  \calB(t),\\
	&0\quad\quad \ \ \ in \  \overline{\mathcal{B} (t)} ^{c}
	\end{aligned}
	\right.
	\end{equation}
	Then
	\begin{equation}
	\nabla \phi=\frac{1}{2}(I-H_{\partial\calB_{t}})\nabla\Phi,\quad on\ \partial\calB_{t}.
	\end{equation}
\end{lemma}
\begin{proof}
	According to \(\eqref{4.7}\), we have
	\begin{align}
	\nabla\cdot\left( \nabla\phi-\nabla\Phi\right) =0\quad and \quad \nabla\times\left( \nabla\phi-\nabla\Phi\right) =0,\quad\quad in\  \calB(t).
	\end{align}
	This indicates that \((\nabla\phi-\nabla\Phi)\) is a Clifford analytic function in \(\calB(t)\) and hence
	\begin{align}
	\begin{split}
	&H_{\partial\calB_{t}}\left( \nabla\phi-\nabla\Phi \right) =\nabla\phi-\nabla\Phi\\
	\Rightarrow \quad \quad &(I-H_{\partial\calB_{t}})\nabla\phi=(I-H_{\partial\calB_{t}})\nabla\Phi.
	\end{split}
	\end{align}
	Similarly since \(\nabla\phi\) is curl and divergence-free outside of \(\calB(t)\),
	we get
	\begin{align}
	(I+H_{\partial\calB_{t}})\nabla\phi=0.
	\end{align} 
	The desired result follows because 
	\begin{align}
	\nabla\phi=\frac{1}{2}(I-H_{\partial\calB_{t}})\nabla\phi+\frac{1}{2}(I+H_{\partial\calB_{t}})\nabla\phi=\frac{1}{2}(I-H_{\partial\calB_{t}})\nabla\Phi.
	\end{align}
\end{proof}
This theorem allows us to use the Hilbert transform to estimate \(\Arrowvert\nabla\psi\Arrowvert_{L^{2}\left( \partial\calB_{t}\right) }\). In the application we usually take \(\Phi=0\) on \(\partial\calB(t)\). 
\begin{lemma}
	Given the Euler-Poisson system \(\eqref{main eq}-\eqref{boundary conditions}\), we have the following equation for the gravitational potential:
	\begin{align}
	\partial_{t}\psi=\nabla\cdot\int_{\calB(t)}\frac{\rho u}{|x-y|}dy
	\end{align}
\end{lemma}
\begin{proof}
	According to the definition of the gravitational potential \(\psi\), we have
	\begin{align}
	\begin{split}
	\partial_{t}\psi&=-\frac{\partial}{\partial t}\int_{\calB(t)}\frac{\rho}{|x-y|}dy\\
	&=-\int_{\calB(t)}\frac{\partial_{t}\rho}{|x-y|}+\div\left( \frac{\rho u}{|x-y|}\right) dy\\
	&=-\int_{\calB(t)}\frac{\partial_{t}\rho}{|x-y|}+\frac{\rho\div u}{|x-y|}+\frac{u\cdot\rho}{|x-y|}+\rho u\cdot\nabla_{y}\left(\frac{1}{|x-y|} \right) dy.
	\end{split}
	\end{align}
	Using the momentum equation in \(\eqref{main eq}\), the first three terms under the integral cancel. Therefore
	\begin{align}
	\partial_{t}\psi=-\int_{\calB(t)}\rho u\cdot\nabla_{y}\left(\frac{1}{|x-y|} \right) dy=\int_{\calB(t)}\rho u\cdot\nabla_{x}\left(\frac{1}{|x-y|} \right) dy=\nabla\cdot\int_{\calB(t)}\frac{\rho u}{|x-y|} dy.
	\end{align}
\end{proof}
The above result allows us to estimate \(\Arrowvert\nabla_{t}\psi\Arrowvert_{L^{2}\left( \partial\calB_{t}\right) }\) using Lemma \(\ref{4.2.}\) hence \(\Arrowvert B^{k}\psi\Arrowvert_{L^{2}\left( \partial\calB_{t}\right) }\) as well.

\begin{proposition}\label{4.4.}
	Suppose the bootstrap assumptions \(\eqref{5.4}\) hold. Then the following estimates hold for any function \(f\) defined on \(\partial\calB\) and any \(k\le l\):
	\begin{align}
	\Arrowvert[B^{k},H_{\partial\calB_{t}}]f\Arrowvert_{L^{2}(\partial\calB_{t})}\lesssim\sum_{i+j\le k-1, j\le 10}(\Arrowvert B^{i}u\Arrowvert_{H^{1}(\partial\calB_{t})}\Arrowvert \nabla B^{j}f\Arrowvert_{L^{\infty}(\partial\calB_{t})}+\Arrowvert B^{j}u\Arrowvert_{W^{2,\infty}(\partial\calB_{t})}\Arrowvert B^{i}f\Arrowvert_{L^{2}(\partial\calB_{t})}).
	\end{align}
\end{proposition}
Before stating the proof the Proposition \(\eqref{4.4.}\) we record some important Lemmas. For any two Clifford algebra-valued functions \(f\) and \(g\) define
\begin{align}
Q(f,g):=\frac{1}{|N|}\left( f_{\alpha}g_{\beta}-f_{\beta}g_{\alpha}\right), 
\end{align}
here \(|N|=|\xi_{\alpha}\times\xi_{\beta}|\) which makes \(Q(f,g)\) coordinate-invariant. If \(f\) and \(g\) are vector-valued we also define 
\begin{align*}
\vec{Q}(f,g):=\frac{1}{|N|}(f_{\alpha}\times g_{\beta}-f_{\beta}\times g_{\alpha}).
\end{align*}
\begin{lemma}\label{4.5.}
	If \(f,g\) and \(h\) are scalar-valued then 
	\begin{align}\label{4.18}
	\int_{\partial\calB(t)}Q(f,g)hdS=-\int_{\partial\calB(t)}fQ(h,g)dS.
	\end{align}
\end{lemma}
\begin{proof}
	In the scalar case we have
	\begin{align}
	Q(f,g)dS=(f_{\alpha}g_{\beta}-f_{\beta}g_{\alpha})d\alpha\wedge d\beta=df\wedge dg,
	\end{align}
	where \(d\) denotes the exterior differentiation operator on \(\partial\calB(t)\). The identity \(\eqref{4.18}\) follows by Stokes' Theorem
	\begin{align}
	\begin{split}
	0=\int_{\partial\calB(t)}d(fhdg)&=\int_{\partial\calB(t)}hdf\wedge dg+\int_{\partial\calB(t)}fdh\wedge dg\\
	&=\int_{\partial\calB(t)}Q(f,g)hdS+\int_{\partial\calB(t)}Q(f,g)hdS.
	\end{split}
	\end{align}
\end{proof}
\begin{lemma}
	Let \(f\) be a Clifford algebra-valued function. Then
	\begin{align}\label{4.21}
	[\partial_{t}^{k},H_{\Sigma}]f\sim\sum_{i+j+p+l\le k-1}\int_{\Sigma}\partial_{t}^{i}K(\xi'-\xi)\partial_{t}^{j}(\xi_{t}-\xi'_{t})Q\left( \partial_{t}^{p}\xi',\partial_{t}^{l}f'\right)dS'. 
	\end{align}
Here ``$\sim$" means we drop the numeric constant coefficients in front of the integrals in the sum.
\end{lemma}
\begin{proof}
	We first prove the  following commutator formulas for the Hilbert transform
	\begin{align}\label{4.22}
	[\partial_{t},H_{\Sigma}]f=\int_{\Sigma}K(\xi'-\xi)(\xi_{t}-\xi'_{t})Q\left( \xi',f'\right)dS'. 
	\end{align}
	Suppose \(\eta\in \bbR^{3}\) are arbitrary vectors and \(\xi\neq\xi'\), and let \(K=K(\xi'-\xi)\). Then in local coordinates \((\alpha,\beta)\) on \(\Sigma\), we have
	\begin{align}\label{4.23}
	-(\eta\cdot\nabla)K(\xi'_{\alpha'}\times\xi'_{\beta'})+(\xi'_{\alpha'}\cdot\nabla)K(\eta\times\xi'_{\beta'})+(\xi'_{\beta'}\cdot\nabla)K(\xi'_{\alpha'}\times\eta)=0.	
	\end{align}
	Now let’s prove \(\eqref{4.22}\). By definition, we have
	\begin{align}\label{4.24}
	\begin{split}
	[\partial_{t},H_{\Sigma}]f&=\partial_{t}(H_{\Sigma}f)-H_{\Sigma}(\partial_{t}f)\\
	&=\int\int\partial_{t}\left(K(\xi'-\xi)(\xi'_{\alpha'}\times\xi'_{\beta'}) \right) (f'-f)d\alpha'd\beta'\\
	&=\int\int\partial_{t}\left(K(\xi'-\xi)\right) (\xi'_{\alpha'}\times\xi'_{\beta'}) (f'-f)d\alpha'd\beta'\\
	&+\int\int K(\xi'-\xi)(\xi'_{t\alpha'}\times\xi'_{\beta'}+\xi'_{\alpha'}\times\xi'_{t\beta'})(f'-f)d\alpha'd\beta'
	\end{split}
	\end{align}
	Notice that
	\begin{align}
	\partial_{t}(K(\xi'-\xi))=((\xi'_{t}-\xi_{t})\cdot\nabla)K(\xi'-\xi),\quad \partial_{\alpha'}K(\xi'-\xi)=(\xi'_{\alpha'}\cdot\nabla)K(\xi'-\xi)
	\end{align}
	and
	\begin{align}
	\partial_{\beta'}K(\xi'-\xi)=(\xi'_{\beta'}\cdot\nabla)K(\xi'-\xi).
	\end{align}
	In \(\eqref{4.23}\) we take \(\eta=\xi'_{t}-\xi_{t}\) and apply to \(\eqref{4.24}\). We get
	\begin{align}\label{4.27}
	\begin{split}
	[\partial_{t},H_{\Sigma}]f&=\int\int{\partial_{\alpha'}K((\xi'_{t}-\xi_{t})\times\xi'_{\beta'})+\partial_{\beta'}K(\xi'_{\alpha'}\times(\xi'_{t}-\xi_{t}))}(f'-f)d\alpha'd\beta'\\
	&+\int\int K(\xi'-\xi)(\xi'_{t\alpha'}\times\xi'_{\beta'}+\xi'_{\alpha'}\times\xi'_{t\beta'})(f'-f)d\alpha'd\beta'.
	\end{split}
	\end{align}
	Applying integration by parts to the first term on the right hand side of \(\eqref{4.27}\), we obtain
	\begin{align}
	\begin{split}
	[\partial_{t},H_{\Sigma}]f&=-\int\int K(\xi'-\xi)\left((\xi'_{t}-\xi_{t})\times\xi'_{\beta'}f'_{\alpha'}+\xi'_{\alpha'}\times(\xi'_{t}-\xi_{t})f'_{\beta'} \right) d\alpha'd\beta'\\
	&-\int\int K(\xi'-\xi)(\xi'_{t\alpha'}\times\xi'_{\beta'}+\xi'_{\alpha'}\times\xi'_{t\beta'})(f'-f)d\alpha'd\beta'\\
	&+\int\int K(\xi'-\xi)(\xi'_{t\alpha'}\times\xi'_{\beta'}+\xi'_{\alpha'}\times\xi'_{t\beta'})(f'-f)d\alpha'd\beta'\\
	&=\int\int K(\xi'-\xi) ((\xi'_{t}-\xi_{t})\times(\xi'_{\beta'}f'_{\alpha'}-\xi'_{\alpha'}f'_{\beta'}))d\alpha'd\beta'\\
	&=\int_{\Sigma}K(\xi'-\xi)(\xi_{t}-\xi'_{t})Q\left( \xi',f'\right)dS'.
	\end{split}
	\end{align}
	Finally the estimate \([\partial_{t}^{k},H_{\Sigma}]f\) follows by writing
	\begin{align}
	[\partial_{t}^{k},H_{\Sigma}]f=[\partial_{t},H_{\Sigma}]\partial_{t}^{k-1}f+\partial_{t}[\partial_{t},H_{\Sigma}]\partial_{t}^{k-2}f+\dots+\partial_{t}^{k-1}[\partial_{t},H_{\Sigma}]f.
	\end{align}
\end{proof}
We next turn to estimates on singular
integral operators. Let \(J:S_{R}\to \bbR_{k}\), \(F:\bbR^{k}\to \bbR\), \(A:S_{R}\to \bbR\) be smooth functions. We want to estimate
singular integrals of the following forms:
\begin{align}\label{4.30}
C_{1}f(p):=p.v.\int_{S_{R}}F\left(\frac{J(p)-J(q)}{|p-q|} \right) \frac{\Pi^{N}_{i=1}(A_{i}(p)-A_{i}(q))}{|p-q|^{N+2}}f(q)dS(q),
\end{align}
where \(dS\) denotes the surface measure on \(S_{R}\), and where we assume that the
kernel
\begin{align}
k_{1}(p,q)=F\left(\frac{J(p)-J(q)}{|p-q|} \right) \frac{\Pi^{N}_{i=1}(A_{i}(p)-A_{i}(q))}{|p-q|^{N+2}}
\end{align}
is odd, that is, \(k_{1}(p,q)=-k_{1}(q,p)\).
\begin{lemma}\label{4.7.}
	\cite{wu2011global}With the same notation as \(\eqref{4.30}\), we have
	\begin{align}\label{4.32}
	\Arrowvert C_{1}f\Arrowvert_{L^{2}(S_{R})}\le C \prod_{i=1}^{N}\left( \Arrowvert\nabla\mkern-10.5mu\slash A_{i}        \Arrowvert_{L^{\infty}(S_{R})}+R^{-1}\Arrowvert A_{i}        \Arrowvert_{L^{\infty}(S_{R})}\right)\Arrowvert f\Arrowvert_{L^{2}(S_{R})},
	\end{align}
	and
	\begin{align}\label{4.33}
	\begin{split}
	\Arrowvert C_{1}f\Arrowvert_{L^{2}(S_{R})}\le&C\left(\Arrowvert\nabla\mkern-10.5mu\slash A_{1}        \Arrowvert_{L^{2}(S_{R})}+R^{-1}\Arrowvert A_{1}        \Arrowvert_{L^{2}(S_{R})} \right) \\
	&\times\prod_{i=2}^{N}\left(\Arrowvert\nabla\mkern-10.5mu\slash A_{i}        \Arrowvert_{L^{\infty}(S_{R})}+R^{-1}\Arrowvert A_{i}        \Arrowvert_{L^{\infty}(S_{R})} \right) \Arrowvert f\Arrowvert_{L^{\infty}(S_{R})},
	\end{split}
	\end{align}
	where \(\nabla\mkern-10.5mu\slash\) denotes the covariant differentiation operator with respect
	to the standard metric on \(S_{R}\) and the constants depend on \(F\), \(\Arrowvert\nabla\mkern-10.5mu\slash J\Arrowvert_{L^{\infty}}\). 
\end{lemma}

\begin{proof}[Proof of Proposition \(\eqref{4.4.}\)]
	We consider \(\Arrowvert	[\partial_{t}^{k},H_{\partial\calB_{t}}]f\Arrowvert_{L^{2}(S_{R})}\) in three limiting cases. First we notice the components of the term
	\begin{align}\label{4.34}
	\int_{\partial\calB_{t}}\partial_{t}^{k-1}K(\xi'-\xi)(\xi_{t}-\xi'_{t})Q\left(\xi',f' \right) dS'
	\end{align}
	satisfy condition of Lemma \(\ref{4.7.}\) by the appropriate transformation. Therefore \(\eqref{4.34}\) can be regarded as lower order terms according to \(\eqref{4.32}\). For the components of the term 
	\begin{align}\label{4.35}
	\int_{\partial\calB_{t}}K(\xi'-\xi)\partial_{t}^{k-1}(\xi_{t}-\xi'_{t})Q\left(\xi',f' \right) dS',
	\end{align}
	we use  Lemma \(\ref{4.7.}\) to bound the \(L^{2}(S_{R})\) norms of \(\eqref{4.34}\) by the right-hand side of \(\eqref{4.33}\). So
	\begin{align}\label{4.36}
	\left| \left| \int_{\partial\calB_{t}}K(\xi'-\xi)\partial_{t}^{k-1}(\xi_{t}-\xi'_{t})Q\left(\xi',f' \right) dS'\right| \right| _{L^{2}(S_{R})}\lesssim(\Arrowvert \nabla\mkern-10.5mu\slash \partial_{t}^{k-1}u\Arrowvert_{L^{2}(S_{R})}+R^{-1}\Arrowvert  \partial_{t}^{k-1}u\Arrowvert_{L^{2}(S_{R})})\Arrowvert \nabla\mkern-10.5mu\slash f\Arrowvert_{L^{\infty}(S_{R})}.
	\end{align} 
	For the term
	\begin{align}\label{4.37}
	\int_{\partial\calB_{t}}K(\xi'-\xi)(\xi_{t}-\xi'_{t})Q\left(\xi',\partial_{t}^{k-1}f' \right) dS',
	\end{align}
	We rewrite the term \(\eqref{4.37}\) as 
	\begin{align}
	\begin{split}
	\int_{\partial\calB_{t}}K^{i}(\xi'-\xi)(\xi^{j}_{t}-(\xi'_{t})^{j})Q\left((\xi')^{p},\partial_{t}^{k-1}f' \right) dS'{ e_{i}e_{j}e_{p}}:=A_{ijp}{ e_{i}e_{j}e_{p}},
	\end{split}
	\end{align}
	suppose \(f\) is scalar function. For any \(i,j,p=1,2,3\), let \((\xi')^{i}-\xi^{i}=(\zeta')^{i}\), \(\xi^{j}_{t}-(\xi'_{t})^{j}=(\eta')^{j}\), \(\partial_{t}^{k-1}f'=\tilde{f'}\)
	and using integration-by-parts formula of Lemma \(\ref{4.5.}\) in the componentwise, we get
	\begin{align}\label{4.39}
	\begin{split}
	&A_{ijp}\sim\int_{\partial\calB_{t}}\frac{(\zeta')^{i}(\eta')^{j}}{|\zeta|^{3}}Q((\xi')^{p},\tilde{f'})\\
	=&\int_{\partial\calB_{t}}Q\left( (\xi')^{p},\frac{(\zeta')^{i}(\eta')^{j}}{|\zeta|^{3}}\right)\tilde{f'}dS' \\
	=&\int_{\partial\calB_{t}}\frac{(\zeta')^{i}}{|\zeta|^{3}}\tilde{f'}Q((\xi')^{p},(\eta')^{j})dS'+\int_{\partial\calB_{t}}\frac{(\eta')^{j}}{|\zeta|^{3}}\tilde{f'}Q((\xi')^{p},(\zeta')^{i})dS'+\int_{\partial\calB_{t}}\frac{(\zeta')^{i}(\eta')^{j}\zeta}{|\zeta|^{5}}\cdot\tilde{f'}Q((\xi')^{p},\zeta)dS'\\
	\lesssim&(\Arrowvert \nabla\mkern-10.5mu\slash u\Arrowvert_{L^{\infty}(S_{R})}+R^{-1}\Arrowvert  u\Arrowvert_{L^{\infty}(S_{R})})\Arrowvert\partial_{t}^{k-1}f\Arrowvert_{L^{2}(S_{R})}.
	\end{split}
	\end{align}
	The last estimate follows from \(\eqref{4.32}\). All other cases can be considered combinations of \(\eqref{4.34}\), \(\eqref{4.35}\) and \(\eqref{4.37}\). Based on above discussion, combining \(\eqref{4.21}\), \(\eqref{4.36}\) and \(\eqref{4.39}\) , we finally obtain
	\begin{align}
	\begin{split}
    \Arrowvert	[\partial_{t}^{k},H_{\partial\calB_{t}}]f\Arrowvert_{L^{2}(S_{R})}
	\lesssim&\sum_{i+j\le k-1, j\le 10}(\Arrowvert\partial_{t}^{i}u\Arrowvert_{H^{1}(S_{R})}\Arrowvert \nabla\mkern-10.5mu\slash \partial_{t}^{j}f\Arrowvert_{L^{\infty}(S_{R})}+\Arrowvert\partial_{t}^{j}u\Arrowvert_{W^{2,\infty}(S_{R})}\Arrowvert \partial_{t}^{i}f\Arrowvert_{L^{2}(S_{R})})\\
	\lesssim&\sum_{i+j\le k-1, j\le 10}(\Arrowvert B^{i}u\Arrowvert_{H^{1}(\partial\calB_{t})}\Arrowvert \nabla B^{j}f\Arrowvert_{L^{\infty}(\partial\calB_{t})}+\Arrowvert B^{j}u\Arrowvert_{W^{2,\infty}(\partial\calB_{t})}\Arrowvert B^{i}f\Arrowvert_{L^{2}(\partial\calB_{t})}).
	\end{split}
	\end{align}
	which completes the proof of Proposition \(\eqref{4.4.}\).
	
\end{proof}
 \section{Priori estimate}\label{5}

  For any function \(\phi\) we define the energies
\begin{align}
\begin{split}
&E[\phi,t]:=\int_{\calB(t)}|\partial_{t,x}\phi|^{2}dx+\int_{\partial\calB(t)}\frac{1}{a}|B\phi|^{2}dS,\\
&\underline{E}[\phi,t]:=\int_{\calB(t)}|\partial_{t,x}\phi|^{2}dx.
\end{split}
\end{align}
Higher order energies are defined as
\begin{align}
E_{j}[\phi,t]=E[B^{j}\phi,t],\quad E_{\le k}[\phi,t]=\sum_{j=0}^{k}E_{j}[\phi,t],\quad\underline{E}_{j}[\phi,t]=\underline{E}[B^{j}\phi,t],\quad\underline{E}_{\le k}[\phi,t]=\sum_{j=0}^{k}\underline{E}_{j}[\phi,t].
\end{align}
To simplify notation we introduce the unified energy
\begin{align}\label{5.3}
\scE_{l}(t):=\sum_{2j+k\le l+2}\left( \Arrowvert B^{k}u\Arrowvert^{2}_{H^{j}\left( \calB_{t}\right) }+   \Arrowvert B^{k+1}\varepsilon\Arrowvert^{2}_{H^{j}\left( \calB_{t}\right) }   \right)+\underline{E}_{\le l+1}[\varepsilon,t]+E_{\le l}[u,t].
\end{align}
Our goal in this section is to prove the following a priori estimate.
 \begin{proposition}\label{5.1.}
	Suppose \(u,\varepsilon\) is a solution to \(\eqref{2.33}-\eqref{2.34}\) with
	\begin{align}\label{5.4}
	\scE_{l}(t)\le C_{l},\quad |\nabla^{(m)}X(\cdot,t)|\le C_{X},\quad|J(t)-1|\le \frac12,\quad\left| \tilde{J}(t)-1\right| \le \frac12,
	\end{align}
	for some constants \(C_{l},C_{X},C_{J},C_{\tilde{J}}>0\) and \(l\) sufficiently large satisfying \(0\le m\ll l\), where \(J(t)\) and \(\tilde{J}(t)\) are the Jacobian of the Lagrangian coordinate transformation from \(\calB(0)\) to \(\calB(t)\) and \(\partial\calB(0)\) to \(\partial\calB(t)\) respectively.
	Then we have 
	\begin{align}\label{5.5}
	\scE_{l}(t)\le C_{0}\scE_{l}(0)+\int_{0}^{t}C_{1}(\bar{\rho})\scE_{l}(s)+C_{2}(\bar{\rho})\calE_{l-1}(s)+C_{h}\scE_{l}^{\frac{3}{2}}(s)ds,
	\end{align}
	for some positive constants \(C_{0},C_{1}(\bar{\rho}), C_{2}(\bar{\rho}),C_{h}\) and \(t\in[0,T]\).	
\end{proposition}
\begin{remark}
	Combined with the discussion in Section \(\ref{1}\), the estimate \(\eqref{5.5}\) and \(\eqref{6.1}\), which we shall prove in the next section, actually show our main conclusion Theorem \(\ref{main th}\), that liquid Lane-Emden stars are nonlinearly unstable, where \(\scE_{l}\) corresponds to a stronger energy norm \(\arrowvert\arrowvert\arrowvert\cdot\arrowvert\arrowvert\arrowvert^{2}\) and \(\Arrowvert\varepsilon,u\Arrowvert^{2}_{L^{2}}\) corresponds to the weaker norm \(\arrowvert\arrowvert\cdot\arrowvert\arrowvert^{2}\).
\end{remark}
\begin{remark}
	The bootstrap assumptions \(\eqref{5.4}\) implies that the norm in the Lagrangian coordinate system is equivalent to the norm in the Cartesian coordinate system, so Proposition \(\ref{5.1.}\) still holds in the Lagrangian coordinate system. It is worth noting that when the conclusion \(\eqref{5.5}\) holds, we can close the bootstrap assumptions using the fundamental theorem of calculus.
\end{remark}

To prove Proposition \(\ref{5.1.}\) we need to show that higher order energies \(\underline{E}_{\le l+1}[\varepsilon,t]\) and \(E_{\le l}[u,t]\) give pointwise control on lower order derivatives of \(u\) and \(\varepsilon\) and \(L^{2}\) control of lower order Sobolev norms of \(u,\varepsilon\). The result is stated in the following proposition.

\begin{proposition}\label{5.2.}
	Under the assumptions of Proposition \(\eqref{5.1.}\), for any \(t\in[0,T]\), we have
	\begin{align}\label{5.6}
	\begin{split}
	&\sum_{k+2p\le l+2}\Arrowvert\partial_{t,x}^{p}B^{k}u\Arrowvert^{2}_{L^{2}\left( \calB_{t}\right) }+	\sum_{k+2p\le l+2}\Arrowvert\partial_{t,x}^{p}B^{k+1}\varepsilon\Arrowvert^{2}_{L^{2}\left( \calB_{t}\right)} \\
	&\lesssim \underline{E}_{\le l+1}[\varepsilon,t]+E_{\le l}[u,t]+	\sum_{k+2p\le l+2}\Arrowvert\partial_{t,x}^{p}B^{k}u\Arrowvert^{2}_{L^{2}\left( \calB_{0}\right) }+	\sum_{k+2p\le l+2}\Arrowvert\partial_{t,x}^{p}B^{k+1}\varepsilon\Arrowvert^{2}_{L^{2}\left( \calB_{0}\right)}+\int_{0}^{t}\scE_{l}(\tau)d\tau.
	\end{split}
	\end{align}
	The implicit coefficient in this estimate depends on the constants of the bootstrap assumption \(\eqref{5.4}\).
\end{proposition}

Before we give the proof of the proposition, we need some preparation. First, we introduce some notations:
\begin{align}
\nabla_{i}:=\partial_{i},\quad\underline{n}:=\frac{(\partial_{1}\bm{\uprho},\partial_{2}\bm{\uprho},\partial_{3}\bm{\uprho})}{\sqrt{\sum_{i=1}^{3}(\partial_{i}\bm{\uprho})^{2}}},\quad\underline{n}_{i}:=\delta_{ij}\underline{n}^{j},\quad
\nabla\mkern-10.5mu\slash_{i}:=\partial_{i}-\underline{n}_{i}\underline{n}^{j}\partial_{j}.
\end{align}
Note that \(\nabla\mkern-10.5mu\slash_{i},i=1,2,3\)  are defined globally, are tangential to \(\partial\calB_{t}\), and span \(T\partial\calB_{t}\).
\begin{lemma}\label{5.5.}
	For any smooth function \(\phi\), the following estimate holds:
	\begin{align}\label{4.12}
	\begin{split}
	\Arrowvert\phi\Arrowvert_{H^{j}\left( \calB_{t}\right)}\lesssim\Arrowvert\phi\Arrowvert_{H^{j-1}\left( \calB_{t}\right)}+\Arrowvert\nabla^{(j-2)}\Delta\phi\Arrowvert_{L^{2}\left( \calB_{t}\right)}
	+\Arrowvert[\nabla\mkern-10.5mu\slash,\nabla^{(j-2)}]\phi\Arrowvert_{H^{1}\left( \calB_{t}\right)}+\Arrowvert\nabla\mkern-10.5mu\slash\phi\Arrowvert_{H^{j-1}\left( \calB_{t}\right)}.
	\end{split}
	\end{align}
\end{lemma}
\begin{proof}
	Using the first estimate in Lemma \(\ref{3.6.}\) and trace theorem
	\begin{align*}
	\Arrowvert\phi\Arrowvert_{H^{j}\left( \calB_{t}\right)}\lesssim&\Arrowvert\phi\Arrowvert_{H^{j-1}\left( \calB_{t}\right)}+\Arrowvert\Delta\nabla^{(j-2)}\phi\Arrowvert_{L^{2}\left( \calB_{t}\right)}+\Arrowvert\nabla\mkern-10.5mu\slash\nabla^{(j-2)}\phi\Arrowvert_{H^{\frac{1}{2}}\left( \partial\calB_{t}\right)}+\Arrowvert\nabla^{(j-2)}\phi\Arrowvert_{H^{\frac{1}{2}}\left( \partial\calB_{t}\right)}\\
	\lesssim&\Arrowvert\phi\Arrowvert_{H^{j-1}\left( \calB_{t}\right)}+\Arrowvert\Delta\nabla^{(j-2)}\phi\Arrowvert_{L^{2}\left( \calB_{t}\right)}+\Arrowvert\nabla\mkern-10.5mu\slash\nabla^{(j-2)}\phi\Arrowvert_{H^{1}\left( \calB_{t}\right)}.
	\end{align*}
	The desired estimate follows after commuting various operators.
\end{proof}
\begin{lemma}\label{5.6.}
	Under the bootstrap assumption \(\eqref{5.4}\), we have
	\begin{align}
	\Arrowvert\nabla^{a}B^{k}u\Arrowvert_{L^{\infty}\left( \calB_{t}\right)}+	\Arrowvert\nabla^{a}B^{k+1}\varepsilon\Arrowvert_{L^{\infty}\left( \calB_{t}\right)}\lesssim \scE_{l}^{\frac{1}{2}}(t),\quad\quad 0\le a \le p-2,\quad k\le l-2p-2, t\in[0,T].
	\end{align}
\end{lemma}
\begin{proof}
	This follows from the Sobolev embedding \(H^{2}(\calB_{t})\hookrightarrow L^{\infty}(\calB_{t})\).
\end{proof}
\begin{lemma}\label{5.7.}
	Under the bootstrap assumption \(\eqref{5.4}\), if \(2a+k \le l+1, t\in[0,T]\), then
	\begin{align}
	\Arrowvert\nabla^{a}B^{k}u\Arrowvert_{L^{2}\left( \calB_{t}\right)}+	\Arrowvert\nabla^{a}B^{k+1}\varepsilon\Arrowvert_{L^{2}\left( \calB_{t}\right)}\lesssim 	\Arrowvert\nabla^{a}B^{k}u\Arrowvert_{L^{2}\left( \calB_{0}\right)}+	\Arrowvert\nabla^{a}B^{k+1}\varepsilon\Arrowvert_{L^{2}\left( \calB_{0}\right)}+\int_{0}^{t}\scE_{l}^{\frac{1}{2}}(s)ds.
	\end{align}
		The implicit coefficient in this estimate depends on the constants of the bootstrap assumption \(\eqref{5.4}\).
\end{lemma}
\begin{proof}
	We rcall the Lagrangian parametrization of \(X\), that is
	\begin{align*}
	\partial_{\tau}X(\tau,y)=Bu(X(\tau,y)),\quad\quad X(0,y)=y.
	\end{align*}
	If \(p_{t}\) is a point on \(\calB_{t}\), we let \(p_{0}\) be the point on \(\calB_{0}\) such that \(X(t,p_{0})=p_{t}\). For any function \(\phi\)
	\begin{align}
	\phi(p_{t})-\phi(p_{0})=\int_{0}^{t}B\phi(p_{\tau})d\tau.
	\end{align}
	According to the bootstrap assumption \(\ref{5.4}\), we can bound the Jacobian of the Lagrangian coordinate transformation from \(\calB_{0}\) to \(\calB_{t}\), therefore we get
	\begin{align*}
	\Arrowvert\phi\Arrowvert_{L^{2}\left( \calB_{t}\right)}\lesssim \Arrowvert\phi\Arrowvert_{L^{2}\left( \calB_{0}\right)}+\int_{0}^{t}\Arrowvert B\phi\Arrowvert_{L^{2}\left( \calB_{s}\right)}ds.
	\end{align*}
	We apply this estimate to \(\phi=\nabla^{a}B^{k}u\) as well \(\phi=\nabla^{a}B^{k+1}\varepsilon\). Then as long as \(2a+k\le l+1\), we have
	\begin{align*}
	\Arrowvert B\phi\Arrowvert_{L^{2}\left( \calB_{s}\right)}\lesssim \scE_{l}^{\frac{1}{2}}(s),
	\end{align*}
	which completes the proof of Lemma \(\ref{5.7.}\).
\end{proof}

\begin{proof}[Proof of Proposition \(\eqref{5.2.}\)]
	Note that we only need to consider \(\partial_{x}^{p}B^{k}u\). Indeed, using induction on the order of \(\partial_{t}\), for \(\partial_{t}\partial_{x}^{p-1}B^{k}u\), we have
	\begin{align}
	\begin{split}
	\partial_{t}\partial_{x}^{p-1}B^{k}u=&B\partial_{x}^{p-1}B^{k}u-u^{a}\partial_{a}\partial_{x}^{p-1}B^{k}u\\
	=&[B,\partial_{x}^{p-1}]B^{k}u+\partial_{x}^{p-1}B^{k+1}u-u^{a}\partial_{a}\partial_{x}^{p-1}B^{k}u.
	\end{split}
	\end{align}
	If we can estimate \(\partial_{t}^{p'}\partial_{x}^{p-p'}B^{k}u\), for \(\partial_{t}^{p'+1}\partial_{x}^{p-p'-1}B^{k}u\), we have
	\begin{align}
	\partial_{t}^{p'+1}\partial_{x}^{p-p'-1}B^{k}u=\partial_{t}\partial_{t}^{p'}\partial_{x}^{p-p'-1}B^{k}u.
	\end{align}
	The induction argument follows exactly the same way as we treat the case when \(p'=0\). The argument for \(\varepsilon\) is the same. Turning to \(\partial_{x}^{p}B^{k}u\), we will use an induction argument on \(p\). When \(p=1\), the result follows
	directly by definition. Now we assume that the estimate holds for index less or equal to \(1\le p \le \frac{M+2}{2}-1\), that is, 
	\begin{align}\label{5.14}
\begin{split}
&\sum_{q\le p}\sum_{k+2q\le l+2}\Arrowvert\partial_{t,x}^{q}B^{k}u\Arrowvert^{2}_{L^{2}\left( \calB_{t}\right) }+	\sum_{q\le p}\sum_{k+2q\le l+2}\Arrowvert\partial_{t,x}^{q}B^{k+1}\varepsilon\Arrowvert^{2}_{L^{2}\left( \calB_{t}\right)} \\
&\lesssim \underline{E}_{\le l+1}[\varepsilon,t]+E_{\le l}[u,t]+	\sum_{k+2q\le l+2}\Arrowvert\partial_{t,x}^{q}B^{k}u\Arrowvert^{2}_{L^{2}\left( \calB_{0}\right) }+	\sum_{k+2q\le l+2}\Arrowvert\partial_{t,x}^{q}B^{k+1}\varepsilon\Arrowvert^{2}_{L^{2}\left( \calB_{0}\right)}+\int_{0}^{t}\scE_{l}(\tau)d\tau.
\end{split}
\end{align}
and prove the estimates for \(p+1\), that is,
	\begin{align}\label{5.15}
\begin{split}
&\sum_{k\le l-2p}\Arrowvert\partial_{t,x}^{p+1}B^{k}u\Arrowvert^{2}_{L^{2}\left( \calB_{t}\right) }+	\sum_{k\le l-2p}\Arrowvert\partial_{t,x}^{p+1}B^{k+1}\varepsilon\Arrowvert^{2}_{L^{2}\left( \calB_{t}\right)} \\
&\lesssim \underline{E}_{\le l+1}[\varepsilon,t]+E_{\le l}[u,t]+	\sum_{k+2q\le l+2}\Arrowvert\partial_{t,x}^{q}B^{k}u\Arrowvert^{2}_{L^{2}\left( \calB_{0}\right) }+	\sum_{k+2q\le l+2}\Arrowvert\partial_{t,x}^{q}B^{k+1}\varepsilon\Arrowvert^{2}_{L^{2}\left( \calB_{0}\right)}+\int_{0}^{t}\scE_{l}(\tau)d\tau.
\end{split}
\end{align}
We start with the estimate for \(\Arrowvert\nabla_{x}^{p+1}B^{k+1}\varepsilon\Arrowvert^{2}_{L^{2}\left( \calB_{t}\right) } \) and in fact first estimate \(\Arrowvert\nabla_{x}^{(2)}\nabla\mkern-10.5mu\slash^{(p-1)}B^{k+1}\varepsilon\Arrowvert^{2}_{L^{2}\left( \calB_{t}\right) }\). To apply Lemma \(\eqref{3.5.}\) to \(\phi:=\nabla\mkern-10.5mu\slash^{p-1}B^{k+1}\varepsilon\) we need to estimate \(\Arrowvert\Delta\nabla\mkern-10.5mu\slash^{(p-1)}B^{k+1}\varepsilon\Arrowvert^{2}_{L^{2}\left( \calB_{t}\right)}\). Using the notation
of Lemma \(\eqref{3.9.}\), we have
\begin{align}\label{5.16}
\begin{split}
\Delta\nabla\mkern-10.5mu\slash^{p-1}B^{k+1}\varepsilon\sim&\nabla\mkern-10.5mu\slash^{p-1}H_{k}+[\nabla\mkern-10.5mu\slash^{p-1},\square_{g}]B^{k+1}\varepsilon+\left( \nabla\bm{\uprho}\right) \left( \nabla\nabla\mkern-10.5mu\slash^{p-1}B^{k+1}\varepsilon\right) \\
&+\nabla\nabla\mkern-10.5mu\slash^{p-1}B^{k+2}\varepsilon+\nabla[B,\nabla\mkern-10.5mu\slash^{p-1}]B^{k+1}\varepsilon.
\end{split}
\end{align}
Except for \(\nabla\mkern-10.5mu\slash^{p-1}H_{k}\) the \(L^{2}\left( \calB_{t}\right)\) norms of all the terms on the right-hand side of \(\eqref{5.16}\) are bounded by the
right-hand side of \(\eqref{5.15}\) using the induction hypothesis \(\eqref{5.14}\). Here for the terms where derivatives hit
the coefficients of \(\nabla\mkern-10.5mu\slash\) it suffices to observe that these coefficients are functions of \(\nabla\bm{\uprho}\). Next we
investigate the structure of \(\nabla\mkern-10.5mu\slash^{p-1}H_{k}\). In view of Lemma \(\eqref{3.9.}\), the remainder terms in \(\nabla\mkern-10.5mu\slash^{p-1}H_{k}\) are 
\begin{align}
\nabla^{p+1}B^{k}\varepsilon,\ \ \nabla^{p+1}B^{k-1}u \ \ and \ \ \nabla^{p-1}B^{k}\nabla\nabla\psi.
\end{align}
The \(L^{2}\left( \calB_{t}\right)\) norm of all other term appearing in \(\nabla\mkern-10.5mu\slash^{p-1}H_{k}\) can be bounded by the right-hand side of \(\eqref{5.15}\) using the induction hypothesis \(\eqref{5.14}\). For the first two terms above, since \(k\le l-2p\) we can use Lemma \(\ref{5.7.}\) to bound the \(L^{2}\left( \calB_{t}\right)\) norms of these terms by the right-hand side of \(\eqref{5.15}\) as well. For the last term, we can use Elliptic estimate and Hilbert transform (which is introduced in the Section \(\eqref{4}\)) frequently to bound its \(L^{2}\left( \calB_{t}\right)\) norms. 

Next we turn to the gravity term. In view of Lemma \(\eqref{3.6.}\),we have
\begin{align}\label{5.18}
\begin{split}
\Arrowvert\nabla^{p-1}B^{k}\nabla\nabla\psi\Arrowvert_{L^{2}\left( \calB_{t}\right) }\lesssim&\Arrowvert\nabla\nabla\nabla^{p-1}B^{k}\psi\Arrowvert_{L^{2}\left( \calB_{t}\right) }+\Arrowvert\nabla^{p-1}[B^{k},\nabla\nabla]\psi\Arrowvert_{L^{2}\left( \calB_{t}\right) }\\
\lesssim&\Arrowvert\nabla^{p-1}B^{k}\psi\Arrowvert_{H^{\frac{3}{2}}\left( \partial\calB_{t}\right)}+\Arrowvert\nabla^{p-1}[\Delta,B^{k}]\psi\Arrowvert_{L^{2}\left( \calB_{t}\right)}\\
&+\Arrowvert\nabla^{p-1}B^{k}\Delta\psi\Arrowvert_{L^{2}\left( \calB_{t}\right)}+\Arrowvert\nabla^{p-1}[B^{k},\nabla\nabla]\psi\Arrowvert_{L^{2}\left( \calB_{t}\right) }\\
:=&\Arrowvert\nabla^{p-1}B^{k}\psi\Arrowvert_{H^{\frac{3}{2}}\left( \partial\calB_{t}\right)}+R_{1}
\end{split}
\end{align}
We focus on the boundary term,
\begin{align}\label{5.19}
\begin{split}
\Arrowvert\nabla^{p-1}B^{k}\psi\Arrowvert_{H^{\frac{3}{2}}\left( \partial\calB_{t}\right)}&\lesssim\Arrowvert\nabla\nabla^{p-1}B^{k}\psi\Arrowvert_{H^{\frac{1}{2}}\left( \partial\calB_{t}\right)}+\Arrowvert\nabla^{p-1}B^{k}\psi\Arrowvert_{H^{\frac{1}{2}}\left( \partial\calB_{t}\right)}\\
&\lesssim\Arrowvert B^{k}\nabla\nabla^{p-1}\psi\Arrowvert_{H^{\frac{1}{2}}\left( \partial\calB_{t}\right)}+\Arrowvert [B^{k},\nabla\nabla^{p-1}]\psi\Arrowvert_{H^{\frac{1}{2}}\left( \partial\calB_{t}\right)}+\Arrowvert\nabla^{p-1}B^{k}\psi\Arrowvert_{H^{\frac{1}{2}}\left( \partial\calB_{t}\right)}\\
&\lesssim\Arrowvert B^{k}\nabla\nabla^{p}\psi\Arrowvert_{L^{2}\left( \partial\calB_{t}\right)}+R_{2}.
\end{split}
\end{align}
Let function \(\Phi\) satisfy \(\Delta\Phi=4\pi G\nabla^{p}\rho,\  in \   \calB(t),\ \ \Phi=0,\ on \ \partial\calB(t) \), then
\begin{align}\label{5.20}
\begin{split}
\Arrowvert B^{k}\nabla\nabla^{p}\psi\Arrowvert_{L^{2}\left( \partial\calB_{t}\right)}&=\frac{1}{2}\Arrowvert B^{k}\left( I-H_{\partial\calB_{t}}\right)\nabla\Phi \Arrowvert_{L^{2}\left( \partial\calB_{t}\right)}\\
&\lesssim\Arrowvert H_{\partial\calB_{t}} B^{k}\nabla\Phi \Arrowvert_{L^{2}\left( \partial\calB_{t}\right)}+\Arrowvert [B^{k},H_{\partial\calB_{t}}]\nabla\Phi \Arrowvert_{L^{2}\left( \partial\calB_{t}\right)}.
\end{split}
\end{align}
For the first term above, we have
\begin{align}\label{5.21}
\begin{split}
\Arrowvert H_{\partial\calB_{t}} B^{k}\nabla\Phi \Arrowvert_{L^{2}\left( \partial\calB_{t}\right)}&\lesssim\Arrowvert B^{k}\nabla\Phi \Arrowvert_{L^{2}\left( \partial\calB_{t}\right)}\\
&\lesssim\Arrowvert \nabla B^{k}\Phi \Arrowvert_{L^{2}\left( \partial\calB_{t}\right)}+\Arrowvert [B^{k},\nabla]\Phi \Arrowvert_{L^{2}\left( \partial\calB_{t}\right)}\\
&\lesssim\Arrowvert  B^{k}\Phi \Arrowvert_{H^{2}\left( \calB_{t}\right)}+\Arrowvert [B^{k},\nabla]\Phi \Arrowvert_{H^{1}\left( \calB_{t}\right)}\\
&\lesssim\Arrowvert  B^{k}\Delta\Phi \Arrowvert_{L^{2}\left( \calB_{t}\right)}+\Arrowvert  [B^{k},\Delta]\Phi \Arrowvert_{L^{2}\left( \calB_{t}\right)}+\Arrowvert [B^{k},\nabla]\Phi \Arrowvert_{H^{1}\left( \calB_{t}\right)}\\
&\lesssim\Arrowvert  \nabla^{p}B^{k}\rho   \Arrowvert_{L^{2}\left( \calB_{t}\right)}+\Arrowvert  [B^{k},\nabla^{p}]\rho   \Arrowvert_{L^{2}\left( \calB_{t}\right)}+R_{3}.
\end{split}
\end{align}
The first two terms on the right-hand side of \(\eqref{5.21}\) are bounded by the right-hand side of \(\eqref{5.15}\) using the induction hypothesis \(\eqref{5.14}\). According to trace theorem and commutator identities, the top order terms about \(u,\varepsilon\) in \(R_{1},R_{2}\) and \(R_{3}\) can be bounded as well. For the remainder terms in \(R_{1},R_{3}\) 
\begin{align}\label{5.22}
\sum_{1\le i\le 2}\sum_{j\le k-1}\Arrowvert\nabla^{p-1}\nabla^{i}B^{j}\psi\Arrowvert_{L^{2}\left( \calB_{t}\right)}\quad and \quad \sum_{1\le i\le 2}\sum_{j\le k-1}\Arrowvert\nabla^{p-1}\nabla^{i}B^{j}\Phi\Arrowvert_{L^{2}\left( \calB_{t}\right)},
\end{align}
we can use elliptic estimates repeatedly to control. For the remainder terms in \(R_{2}\), we can choose the appropriate function \(\Phi\) to bound them in the same way as in the treatment of \(B^{k}\nabla\nabla^{p}\psi\) in \(\eqref{5.20}\). The fact that \([B^{k},H_{\partial\calB_{t}}]\nabla\Phi\) are the lower order terms and some remaining details are proved in Section \(\ref{4}\). Based
on this discussion, Using \(\eqref{5.16}\) and Lemma \(\ref{3.5.}\), for any \(k\le l-2p\) we obtain
\begin{align}\label{5.23}
\begin{split}
&\Arrowvert\nabla^{(2)}\nabla\mkern-10.5mu\slash^{p-1}B^{k+1}\varepsilon\Arrowvert^{2}_{L^{2}\left( \calB_{t}\right) }\\
\lesssim&\underline{E}_{\le l+1}[\varepsilon,t]+E_{\le l}[u,t]+	\sum_{k+2q\le l+2}\Arrowvert\partial_{t,x}^{q}B^{k}u\Arrowvert^{2}_{L^{2}\left( \calB_{0}\right) }+	\sum_{k+2q\le l+2}\Arrowvert\partial_{t,x}^{q}B^{k+1}\varepsilon\Arrowvert^{2}_{L^{2}\left( \calB_{0}\right)}+\int_{0}^{t}\scE_{l}(\tau)d\tau.
\end{split}
\end{align}
Next we apply Lemma \(\ref{5.5.}\) to \(\phi:=\nabla\nabla\mkern-10.5mu\slash^{p-2}B^{k+1}\varepsilon\) to get
\begin{align}\label{5.24}
\begin{split}
\Arrowvert\nabla\nabla\mkern-10.5mu\slash^{p-2}B^{k+1}\varepsilon\Arrowvert_{H^{2}\left( \calB_{t}\right)}\lesssim&\Arrowvert\nabla\nabla\mkern-10.5mu\slash^{p-2}\Delta B^{k+1}\varepsilon\Arrowvert_{L^{2}\left( \calB_{t}\right)}+\Arrowvert[\nabla\nabla\mkern-10.5mu\slash^{p-2},\Delta]B^{k+1}\varepsilon\Arrowvert_{L^{2}\left( \calB_{t}\right)}+\Arrowvert\nabla\mkern-10.5mu\slash^{p-2}B^{k+1}\varepsilon\Arrowvert_{H^{2}\left( \calB_{t}\right)}\\
&+\Arrowvert[\nabla,\nabla\mkern-10.5mu\slash]\nabla\mkern-10.5mu\slash^{p-2}B^{k+1}\varepsilon\Arrowvert_{H^{1}\left( \calB_{t}\right)}+\Arrowvert\nabla\mkern-10.5mu\slash^{p-1}B^{k+1}\varepsilon\Arrowvert_{H^{2}\left( \calB_{t}\right)}.
\end{split}
\end{align}
By \(\eqref{5.23}\) and the arguments leading to it, all the terms on the right-hand side of \(\eqref{5.24}\) except 
\begin{align*}
\Arrowvert\nabla\nabla\mkern-10.5mu\slash^{p-2}\Delta B^{k+1}\varepsilon\Arrowvert_{L^{2}\left( \calB_{t}\right)}
\end{align*}
are bounded by the right-hand side of \(\eqref{5.15}\). The term \(\Arrowvert\nabla\nabla\mkern-10.5mu\slash^{p-2}\Delta B^{k+1}\bm{\uprho}\Arrowvert_{L^{2}\left( \calB_{t}\right)}\) is bounded in the same
way as in the treatment of \(\nabla\mkern-10.5mu\slash^{p-1}H_{k}\) above. Summarizing we have obtained
\begin{align}\label{5.25}
\begin{split}
&\Arrowvert\nabla^{(3)}\nabla\mkern-10.5mu\slash^{p-2}B^{k+1}\varepsilon\Arrowvert^{2}_{L^{2}\left( \calB_{t}\right) }\\
\lesssim&\underline{E}_{\le l+1}[\varepsilon,t]+E_{\le l}[u,t]+	\sum_{k+2q\le l+2}\Arrowvert\partial_{t,x}^{q}B^{k}u\Arrowvert^{2}_{L^{2}\left( \calB_{0}\right) }+	\sum_{k+2q\le l+2}\Arrowvert\partial_{t,x}^{q}B^{k+1}\varepsilon\Arrowvert^{2}_{L^{2}\left( \calB_{0}\right)}+\int_{0}^{t}\scE_{l}(\tau)d\tau.
\end{split}
\end{align}
Repeating the argument inductively for \(\phi:=\nabla^{2}\nabla\mkern-10.5mu\slash^{p-3}B^{k+1}\varepsilon,\nabla^{3}\nabla\mkern-10.5mu\slash^{p-4}B^{k+1}\varepsilon,\dots\) we finally obtain
\begin{align}\label{5.26}
\begin{split}
&\Arrowvert\nabla^{p}B^{k+1}\varepsilon\Arrowvert^{2}_{L^{2}\left( \calB_{t}\right) }\\
\lesssim&\underline{E}_{\le l+1}[\varepsilon,t]+E_{\le l}[u,t]+	\sum_{k+2q\le l+2}\Arrowvert\partial_{t,x}^{q}B^{k}u\Arrowvert^{2}_{L^{2}\left( \calB_{0}\right) }+	\sum_{k+2q\le l+2}\Arrowvert\partial_{t,x}^{q}B^{k+1}\varepsilon\Arrowvert^{2}_{L^{2}\left( \calB_{0}\right)}+\int_{0}^{t}\scE_{l}(\tau)d\tau.
\end{split}
\end{align}
Next we use the second estimate in Lemma \(\ref{3.6.}\) to estimate \(\Arrowvert\nabla^{p+1}B^{k}u\Arrowvert_{L^{2}\left( \calB_{t}\right) },\ k+2p+2\le l+2\), under
the induction hypothesis \(\eqref{5.14}\). The second estimate in Lemma \(\ref{3.6.}\) gives
\begin{align}\label{5.27}
\Arrowvert\nabla^{p+1}B^{k}u\Arrowvert_{L^{2}\left( \calB_{t}\right) }\lesssim\Arrowvert\Delta\nabla^{p-1}B^{k}u\Arrowvert_{L^{2}\left( \calB_{t}\right) }+
\Arrowvert\delta^{ab}\left(\nabla_{a}\bm{\uprho} \right) \nabla_{b}\left( \nabla^{p-1}B^{k}u \right)  \Arrowvert_{H^{\frac{1}{2}}\left( \partial\calB_{t}\right) }.
\end{align}
The term \(\Delta\nabla^{p-1}B^{k}u\) has the similar structure to the corresponding term in \(\eqref{5.16}\) and can be handled using
similar considerations, so we concentrate on the boundary contribution \(\Arrowvert\delta^{ab}\left(\nabla_{a}\bm{\uprho} \right) \nabla_{b}\left( \nabla^{p-1}B^{k}u \right)  \Arrowvert_{H^{\frac{1}{2}}\left( \partial\calB_{t}\right) }\). Using the trace theorem and Lemma \(\ref{3.7.}\), we have
\begin{align}
\begin{split}
&\Arrowvert\delta^{ab}\left(\nabla_{a}\bm{\uprho} \right) \nabla_{b}\left( \nabla^{p-1}B^{k}u \right)  \Arrowvert_{H^{\frac{1}{2}}\left( \partial\calB_{t}\right) }\\
\lesssim&\Arrowvert[\delta^{ab}\left(\nabla_{a}\bm{\uprho} \right) \nabla_{b},\nabla^{p-1}]B^{k}u   \Arrowvert_{H^{1}\left( \calB_{t}\right) }+\Arrowvert\nabla^{p-1}B^{k+2}u   \Arrowvert_{H^{1}\left( \calB_{t}\right) }+\Arrowvert\nabla^{p}B^{k+1}\varepsilon   \Arrowvert_{H^{1}\left( \calB_{t}\right) }+\Arrowvert\nabla^{p-1}F_{k}   \Arrowvert_{H^{\frac{1}{2}}\left( \partial\calB_{t}\right) }.
\end{split}
\end{align}
Except for the last term \(\Arrowvert\nabla^{p-1}F_{k}   \Arrowvert_{H^{\frac{1}{2}}\left( \partial\calB_{t}\right) }\), all other terms on the right above are bounded by the right-hand
side of \(\eqref{5.15}\) using the induction hypothesis \(\eqref{5.14}\) and \(\eqref{5.26}\). For \(\Arrowvert\nabla^{p-1}F_{k}   \Arrowvert_{H^{\frac{1}{2}}\left( \partial\calB_{t}\right) }\), in view of Lemma \(\ref{3.7.}\) the highest order terms are
\begin{align*}
\Arrowvert \nabla^{p+1}B^{k-1}u\Arrowvert_{L^{2}\left( \calB_{t}\right) },\quad \Arrowvert \nabla^{p+1}B^{k}\varepsilon\Arrowvert_{L^{2}\left( \calB_{t}\right) }\quad and \quad \Arrowvert \nabla^{p-1}B^{k+1}\nabla\psi\Arrowvert_{H^{\frac{1}{2}}\left( \partial\calB_{t}\right) }.
\end{align*}
The term \(\Arrowvert \nabla^{p+1}B^{k}\varepsilon\Arrowvert_{L^{2}\left( \calB_{t}\right) }\) was already bounded in \(\eqref{5.26}\), and \(\Arrowvert \nabla^{p+1}B^{k-1}u\Arrowvert_{L^{2}\left( \calB_{t}\right) }\) can be handled using
Lemma \(\ref{5.7.}\). Noticing the following result
\begin{align}
\Arrowvert \nabla^{p-1}B^{k+1}\nabla\psi\Arrowvert_{H^{\frac{1}{2}}\left( \partial\calB_{t}\right) }\sim \Arrowvert \nabla^{p-1}B^{k}\nabla_{t,x}\nabla\psi\Arrowvert_{H^{\frac{1}{2}}\left( \partial\calB_{t}\right) }.
\end{align}
Therefore we can bound \(\Arrowvert \nabla^{p-1}B^{k+1}\nabla\psi\Arrowvert_{H^{\frac{1}{2}}\left( \partial\calB_{t}\right) }\) in the same
way as in the treatment of \(\eqref{5.19}-\eqref{5.22}\). Putting everything together we have proved that
\begin{align}\label{5.30}
\begin{split}
&\Arrowvert\delta^{ab}\left(\nabla_{a}\bm{\uprho} \right) \nabla_{b}\left( \nabla^{p-1}B^{k}u \right)  \Arrowvert_{H^{\frac{1}{2}}\left( \partial\calB_{t}\right) }\\
&\lesssim\underline{E}_{\le l+1}[\varepsilon,t]+E_{\le l}[u,t]+	\sum_{k+2q\le l+2}\Arrowvert\partial_{t,x}^{q}B^{k}u\Arrowvert^{2}_{L^{2}\left( \calB_{0}\right) }+	\sum_{k+2q\le l+2}\Arrowvert\partial_{t,x}^{q}B^{k+1}\varepsilon\Arrowvert^{2}_{L^{2}\left( \calB_{0}\right)}+\int_{0}^{t}\scE_{l}(\tau)d\tau.
\end{split}
\end{align}
Combining \(\eqref{5.27}\) and \(\eqref{5.30}\), we finally obtain
\begin{align}\label{5.31}
\begin{split}
&\Arrowvert\nabla^{p}B^{k}u\Arrowvert^{2}_{L^{2}\left( \calB_{t}\right) }\\
\lesssim&\underline{E}_{\le l+1}[\varepsilon,t]+E_{\le l}[u,t]+	\sum_{k+2q\le l+2}\Arrowvert\partial_{t,x}^{q}B^{k}u\Arrowvert^{2}_{L^{2}\left( \calB_{0}\right) }+	\sum_{k+2q\le l+2}\Arrowvert\partial_{t,x}^{q}B^{k+1}\varepsilon\Arrowvert^{2}_{L^{2}\left( \calB_{0}\right)}+\int_{0}^{t}\scE_{l}(\tau)d\tau,
\end{split}
\end{align}
which completes the proof of Proposition \(\ref{5.2.}\).
\end{proof}
The following we only need to show that \(\underline{E}_{\le l+1}[\varepsilon,t]\) and \(E_{\le l}[u,t]\) are bounded by the right-hand side of \(\eqref{5.5}\). we will complete this work by using energy estimates \(\eqref{energy}\), \(\eqref{estimate1}\) and \(\eqref{estimate2}\). In Lemma \(\ref{3.3.}\), the main term of the right-hand side of \(\eqref{estimate1}\) is
\begin{align*}
\left| \int_{0}^{t}\int_{\calB\left( s\right) } \left( \square_{g}\phi\right) \left( Q\phi\right)  dxds\right|,
\end{align*}
and the ramainder terms are naturally controlled by the right-hand side \(\eqref{5.5}\). For the above term with \(\phi=B^{k+1}\varepsilon\), we need consider all possible case of \(\square_{g}B^{k+1}\varepsilon\) stated in Lemma \(\ref{3.9.}\). Similarly in Lemma \(\ref{3.2.}\), for the right-hand side of \(\eqref{energy}\) with \(\phi=B^{k}u\), what we need to controll is
\begin{align*}
\left| \int_{0}^{t}\int_{\partial\calB(\tau)}F_{k}\left( B^{k+1}u\right) dSd\tau\right| +\left| \int_{0}^{t}\int_{\calB(s)}G_{k}\left( B^{k+1}u\right) dxds\right| .
\end{align*}
To simplify notation we define as
\begin{align*}
\calE_{k}(t):=\underline{E}_{\le l+1}[\varepsilon,t]+ E_{\le l}[u,t].
\end{align*}
We first consider two special cases intorduced in the following Lemma.
\begin{lemma}\label{5.8.}
	Suppose the bootstrap assumption \(\eqref{5.4}\) hold. Then there exist some positive constants \(C_{0}\), \(C_{1}(\bar{\rho})\), \(C_{2}(\bar{\rho})\), \(C_{h}\) for any \(k\le l\) and \(t\in[0,T]\), we have
	\begin{align}\label{5.32}
	\int_{0}^{t}\int_{\calB(\tau)}|\nabla^{2}B^{k}\varepsilon|^{2}dxd\tau\le C_{0}\scE_{l}(0)+\int_{0}^{t}C_{1}(\bar{\rho})\scE_{l}(s)+C_{2}(\bar{\rho})\calE_{l-1}(s)+C_{h}\scE_{l}^{\frac{3}{2}}(s)ds.
	\end{align}
\end{lemma}
\begin{proof}
	We proceed inductively. For \(k=0\) this follows for instance by \(\eqref{3.20}\) and \(\eqref{2.20}\) using the elliptic estimates. Iductively, suppose the statement of the lemma holds for \(k\le j-1\le l-1\) and let us prove it for \(k=j\). Applying Lemma \(\ref{3.5.}\) and elliptic estimates with \(\phi=B^{j}\varepsilon\), we need to bound the \(\Arrowvert H_{j-1}\Arrowvert^{2} _{L^{2}(\calB_{s})}\), with \(H_{j-1}\) as in Lemma \(\ref{3.9.}\). The contribution from line \(\eqref{i1},\eqref{i2},\eqref{i3},\eqref{i4},\eqref{i5}\) and \(\eqref{i7}\) in Lemma \(\ref{3.9.}\) can be controlled by the energy \(\scE_{l}\) (the top order terms are \(\nabla^{2}B^{k-2}u\ and\ \nabla^{2}B^{k-1}\varepsilon\)). Then we consider the contribution of line \(\eqref{i6}\) and focus on \(\Arrowvert B^{k}\nabla\nabla\psi\Arrowvert^{2}_{L^{2}(\calB_{s})}\) in particular. Using similar processing methods like \(\eqref{5.18}-\eqref{5.22}\), \(\Arrowvert B^{k-1}\nabla\nabla\psi\Arrowvert^{2}_{L^{2}(\calB_{s})}\) can be bounded by
	\begin{align}
	\sum_{i\le k-2}\int_{\calB(s)}|\nabla^{2}B^{i}u|^{2}dx+\sum_{i\le k-1}\int_{\calB(s)}\left(|\nabla B^{i}\varepsilon|^{2} +|\nabla B^{i}u|^{2}\right) dx,
	\end{align}
	which completes the proof of Lemma \(\ref{5.8.}\).
\end{proof}
The next lemma allows us to estimate the \(L^{2}(\partial\calB_{s})\) norm of \(\nabla B^{k-1}u\).
\begin{lemma}\label{5.9.}
	Suppose the bootstrap assumption \(\eqref{5.4}\) hold. Then there exist some positive constants \(C_{0}\), \(C_{1}(\bar{\rho})\), \(C_{2}(\bar{\rho})\), \(C_{h}\) for any \(k\le l-1\) and \(t\in[0,T]\), we have
	\begin{align}\label{5.34}
	\int_{0}^{t}\int_{\partial\calB(\tau)}|\nabla B^{k}u|^{2}dSd\tau\le C_{0}\scE_{l}(0)+\int_{0}^{t}C_{1}(\bar{\rho})\scE_{l}(s)+C_{2}(\bar{\rho})\calE_{l-1}(s)+C_{h}\scE_{l}^{\frac{3}{2}}(s)ds.
	\end{align}
\end{lemma}
\begin{proof}
	For \(k=0\) this follows for instance from the trace theorem. Proceeding inductively, we assume \(\eqref{5.34}\) holds for \(k\le j-1\le l-2\) and prove it for \(k=j\). Applying Lemma \(\ref{3.4.}\) with \(\phi=B^{j}u\), what term we need to process is 
	\begin{align*}
	\int_{0}^{t}\int_{\calB(s)}|G_{j}|^{2}dxds+\int_{0}^{t}\int_{\partial\calB(\tau)}|nB^{j}u|^{2}dSd\tau,
	\end{align*}
	and the raminder terms are bounded by the right-hand side of \(\eqref{5.34}\). For the contribution \(nB^{j}u\) on \(\calB\) we apply Lemma \(\ref{3.7.}\) to write
	\begin{align}\label{5.35}
	-nB^{j}u=\frac{1}{a}BB^{j+1}u+\frac{1}{a}\nabla B^{j+1}\varepsilon-\frac{1}{a}F_{j}.
	\end{align}
	The first two terms on the right-hand side of \(\eqref{5.35}\) are bounded by using the trace theorem and Lemma \(\ref{5.8.}\). The contributin of \(F_{j}\) can be bounded by the right-hand side of \(\eqref{5.34}\) using the Hilbert transform and the elliptic estimates. Finally we consider the contribution of \(G_{j}\). The line \(\eqref{j1}\) and \(\eqref{j2}\) in Lemma \(\ref{3.8.}\) are bounded by
	\begin{align}\label{5.36}
	\sum_{i\le j}\int_{0}^{t}\int_{\calB(s)}|\nabla B^{i}\varepsilon|^{2}dxds+\sum_{i\le j}\int_{0}^{t}\int_{\calB(s)}|\nabla B^{i}u|^{2}dxds+\sum_{i\le j-1}\int_{0}^{t}\int_{\calB(s)}|\nabla^{2} B^{i}u|^{2}dxds,
	\end{align}
	hence by the right-hand side of \(\eqref{5.34}\). For the term \(B^{j+1}\nabla\psi\) from line \(\eqref{j3}\) in \(G_{j}\), we actually consider
	\begin{align}
	\int_{0}^{t}\int_{\calB(s)}|B^{j}\nabla_{t,x}\nabla\psi |^{2}dxds,
	\end{align}
	which is bounded by \(\eqref{5.34}\) in the same way as in the treatment of \(\eqref{5.18}-\eqref{5.22}\).
\end{proof}
We turn to the proof of Proposition \(\ref{5.1.}\)
\begin{proof}[Proof of Proposition \(\eqref{5.1.}\)]
	We proceed inductively on \(k\) to show the pirori estimete \(\eqref{5.5}\), that is, there exist some positive constants \(C_{0}\), \(C_{1}(\bar{\rho})\), \(C_{2}(\bar{\rho})\), \(C_{h}\) for any \(j\le l\) and \(t\in[0,T]\)
	\begin{align}\label{5.38}
	\underline{E}_{\le j+1}[\varepsilon,t]+E_{\le j}[u,t]\le C_{0}\scE_{l}(0)+\int_{0}^{t}C_{1}(\bar{\rho})\scE_{l}(s)+C_{2}(\bar{\rho})\calE_{l-1}(s)+C_{h}\scE_{l}^{\frac{3}{2}}(s)ds.
	\end{align}
	For small enough \(j\), the estimate \(\eqref{5.38}\) is clearly true. Now we assume that \(\eqref{5.38}\) holds for \(j\le k-1\le l-1\) and show it for \(j=k\).\\
	{\bf Step\ 1:} First we show that
	\begin{align}\label{5.39}
	\underline{E}_{\le k+1}[\varepsilon,t]\le C_{0}\scE_{l}(0)+\int_{0}^{t}C_{1}(\bar{\rho})\scE_{l}(s)+C_{2}(\bar{\rho})\calE_{l-1}(s)+C_{h}\scE_{l}^{\frac{3}{2}}(s)ds.
	\end{align}
	As we have discussed before, the method is to apply energy estimate \(\eqref{estimate1}\) with \(\phi=B^{k+1}\varepsilon\) to equation \(\eqref{3.30}\), which can be done because \(B^{k+1}\varepsilon=0\) on \(\partial\calB\). It is obvious that the main term of the right-hand side of \(\eqref{estimate1}\) is
    \begin{align*}
    \left| \int_{0}^{t}\int_{\calB\left( s\right) } H_{k} \left( QB^{k+1}\varepsilon\right)  dxds\right|,
    \end{align*}
	where \(H_{k}\) is as in Lemma \(\ref{3.9.}\). If \(H_{k}\) is of the form \(\eqref{i1},\eqref{i2},\eqref{i3},\eqref{i5}\) and \(\eqref{i7}\) in Lemma \(\ref{3.9.}\), then we can use Cauchy-Schwarz to bound above contribution by
	\begin{align*}
	\sum_{i\le k}\int_{0}^{t}\int_{\calB(s)}|\nabla B^{i}u|^{2}dxds+\sum_{i\le k+1}\int_{0}^{t}\int_{\calB(s)}|\nabla B^{i}\varepsilon|^{2}dxds+\sum_{i\le k}\int_{0}^{t}\int_{\calB(s)}|\nabla^{2} B^{i}\varepsilon|^{2}dxds.
	\end{align*}
	According to Lemma \(\ref{5.8.}\) the above three terms are bounded by the right-hand side of \(\eqref{5.39}\). If \(H_{k}\) is of the form \(\eqref{i6}\) in Lemma \(\ref{3.9.}\) we use a similar approach like \(\eqref{5.18}-\eqref{5.22}\). It is worth noting that we should use Lemma \(\eqref{3.6.}\) first to bound \(\Arrowvert B^{k}\nabla\psi\Arrowvert_{H^{1}\left( \calB_{s}\right) }\). Because we can bound \(\Arrowvert\nabla B^{k-1}u\Arrowvert_{L^{2}\left( \partial\calB_{s}\right) }\) rather than \(\Arrowvert\nabla^{2}B^{k-1}u\Arrowvert_{L^{2}\left( \calB_{s}\right) }\) according to Lemma \(\eqref{5.9.}\). Finally for the contribution from line \(\eqref{i4}\) in Lemma \(\ref{3.9.}\) we use a few integration by parts. We treat the most difficult case when \(k_{m+2}=k-1\), and write the resulting expression as
	\begin{align}
	\tilde{F}^{ab}\nabla_{a}\nabla_{b}B^{k-1}u=\tilde{F}^{ab}g_{a\beta}g^{\alpha\beta}\nabla_{\alpha}\nabla_{b}B^{k-1}u:=F^{\alpha b}\nabla_{\alpha}\nabla_{b}B^{k-1}u.
	\end{align}
	Replacing \(H_{k}\) by this expression and making some simple calculations, we have
	\begin{align}\label{5.41}
	\begin{split}
	(F^{\alpha b}\nabla_{\alpha}\nabla_{b}B^{k-1}u)(QB^{k+1}\varepsilon)=&\nabla_{\lambda}[(F^{\lambda b}\nabla_{b}B^{k-1}u)(QB^{k+1}\varepsilon)-(F^{\alpha b}\nabla_{b}B^{k-1}u)(B^{\lambda}Q\nabla_{\alpha}B^{k}\varepsilon)]\\
	&-(\nabla_{\lambda}F^{\lambda b})(\nabla_{b}B^{k-1}u)(QB^{k+1}\varepsilon)-(F^{\lambda b}\nabla_{b}B^{k-1}u)(\nabla_{\lambda}Q^{\alpha})(\nabla_{\alpha}B^{k+1}\varepsilon)\\
	&+(BF^{\alpha b})(\nabla_{b}B^{k-1}u)(Q\nabla_{\alpha}B^{k}\varepsilon)+(F^{\alpha b}\nabla_{b}B^{k}u)(Q\nabla_{\alpha}B^{k}\varepsilon)\\
	&-(\nabla_{b}B^{\lambda})(F^{\alpha b}\nabla_{\lambda}B^{k-1}u)(Q\nabla_{\alpha}B^{k}\varepsilon)+(F^{\alpha b}\nabla_{b}B^{k-1}u)(BQ^{\lambda})(\nabla_{\lambda}\nabla_{\alpha}B^{k}\varepsilon)\\
	&+(F^{\alpha b}\nabla_{b}B^{k-1}u)(Q^{\lambda}[B,\nabla_{\lambda}\nabla_{\alpha}]B^{k}\varepsilon)+(F^{\alpha b}\nabla_{b}B^{k-1}u)(\nabla_{\lambda}B^{\lambda})(Q\nabla_{\alpha}B^{k}\varepsilon).
	\end{split}
	\end{align}
	Except for the first line, the other terms can be bounded by the same arguments as above. For the first line we integrate by parts. The resulting on \(\calB_{t}\) involve at most one top order term, hence these terms can be bounded by using Cauchy-Schwarz inequality with a small constant and using the induction hypothesis. Finally, since \(B\) is tangential to \(\partial\calB\), the integration of boundary term can be bounded by Cauchy-Schwarz inequality with a small constant as well. This finishes the proof of \(\eqref{5.39}\).\\
	{\bf Step\ 2:} Here we show that
	\begin{align}\label{5.42}
	\left| \int_{0}^{t}\int_{\calB(s)}G_{k}\left( B^{k+1}u\right) dxds\right|\le C_{0}\scE_{l}(0)+\int_{0}^{t}C_{1}(\bar{\rho})\scE_{l}(s)+C_{2}(\bar{\rho})\calE_{l-1}(s)+C_{h}\scE_{l}^{\frac{3}{2}}(s)ds.
	\end{align}
	where \(G_{k}\) is as in Lemma \(\ref{3.8.}\). The form \(\eqref{j1}\) and \(\eqref{j3}\) can be bounded in the same way as in Step 1 above. For the form \(\eqref{j2}\) of \(G_{k}\), we also treat the hardest case when \(k_{m+1}=k-1\), and express \(G_{k}\) as
	\begin{align*}
	\tilde{G}^{ab}\nabla_{a}\nabla_{b}B^{j-1}u=\tilde{G}^{ab}g_{a\beta}g^{\alpha\beta}\nabla_{\alpha}\nabla_{b}B^{j-1}u:=G^{\alpha b}\nabla_{\alpha}\nabla_{b}B^{j-1}u.
	\end{align*}
	Then we have
	\begin{align}\label{5.43}
	\begin{split}
	(G^{\alpha b}\nabla_{\alpha}\nabla_{b}B^{j-1}u)(B^{j+1}u)=&\nabla_{\lambda}[(G^{\lambda b}\nabla_{b}B^{j-1}u)(B^{j+1}u)-(G^{\alpha b}\nabla_{b}B^{j-1}u)(B^{\lambda}\nabla_{\alpha}B^{j}u)]\\
	&-(\nabla_{\alpha}G^{\alpha b})(\nabla_{b}B^{j-1}u)(B^{j+1}u)+(BG^{\alpha b})(\nabla_{b}B^{j-1}u)(\nabla_{\alpha}B^{j}u)\\
	&+(G^{\alpha b}\nabla_{b}B^{j}u)(\nabla_{\alpha}B^{j}u)-(G^{\alpha b}\nabla_{\lambda}B^{j-1}u)(\nabla_{b}B^{\lambda})(\nabla_{\alpha}B^{j}u)\\
	&+(G^{\alpha b}\nabla_{b}B^{j-1}u)(\nabla_{\lambda}B^{\lambda})(\nabla_{\alpha}B^{j}u)-(G^{\alpha b}\nabla_{b}B^{j-1}u)(\nabla_{\alpha}B^{\lambda})(\nabla_{\lambda}B^{j}u).
	\end{split}
	\end{align}
	Integrating \(\eqref{5.43}\) over \(\cup_{s\in[0,t]}\calB(s)\), the contribution of the right-hand side of \(\eqref{5.43}\) can be bounded by using integration by parts and Cauchy-Schwarz inequality as well as in the treatment of \(H_{k}\). This completes the proof of \(\eqref{5.42}\).\\
	{\bf Step\ 3:} Then we show that
	\begin{align}\label{5.44}
		E_{\le k}[u,t]\le C_{0}\scE_{l}(0)+\int_{0}^{t}C_{1}(\bar{\rho})\scE_{l}(s)+C_{2}(\bar{\rho})\calE_{l-1}(s)+C_{h}\scE_{l}^{\frac{3}{2}}(s)ds.
	\end{align}
	We apply the Lemma \(\ref{3.2.}\) with \(\phi=B^{k}u\). we note that except for the integration of boundary term
	\begin{align}
	\int_{0}^{t}\int_{\partial\calB(\tau)}F_{k}\left(B^{k+1}u\right)dSd\tau ,
	\end{align}
	 all other terms on the right-hand side of \(\eqref{energy}\) are bounded by using Lemma \(\ref{estimate1}\) and \(\eqref{5.42}\). The contribution of the form \(\eqref{k1}\) and \(\eqref{k2}\) in Lemma \(\ref{3.7.}\) are bounded by using trace theorem, \(\eqref{5.32}\), \(\eqref{5.34}\) and induction hypothesis. For the terms of the form \(\eqref{k3}\), we consider the most difficult case
	 \begin{align}
	 \int_{0}^{t}\int_{\partial\calB(\tau)}|B^{k}\nabla_{t,x}\nabla\psi|^{2}dSd\tau,
	 \end{align}
	 which is bounded in the same way as in the treatment of  \(\eqref{5.20}-\eqref{5.22}\).Note that \(\eqref{5.6}\), \(\eqref{5.39}\) and \(\eqref{5.44}\) complete the proof of the Proposition \(\ref{5.1.}\).\\
\end{proof}

\section{Nonlinear instability}\label{6}
Here we describe how to close the bootstrap argument in Lemma \(\ref{1.4.}\) from linear instability to nonlinear instability. As discussed before, we know that the emergence of nonlinear instability requires the stronger norm to be controlled reversely by the weaker norm, that is, we need to prove that there exists small enough \(\epsilon>0\) and constant \(C_{\epsilon}>0\), such that
\begin{align}\label{6.1}
\int_{0}^{t}C_{1}(\bar{\rho})\scE_{l}(s)+C_{2}(\bar{\rho})\calE_{l-1}(s) ds\le\int_{0}^{t}\epsilon\mu_{0}\scE_{l}(s)+C_{\epsilon}\left(\Arrowvert \varepsilon \Arrowvert^{2}_{L^{2}(\calB_{s})}+\Arrowvert u \Arrowvert^{2}_{L^{2}(\calB_{s})} \right) ds.
\end{align}
 We recall the conclusion of reference \cite{L}, when \(1\le\gamma<\frac{4}{3}\) and stars with large central density, liquid Lane-Emden stars are linearly unstable, and the fastest linear growth mode \(\mu_{0}\) in increases with the increase of center density (that is, \(\mu_{0}\) is related to steady state \(\bar{\rho}\)). Observing the structure of \(F_{k},G_{k},H_{k}\), we notice that the coefficients of linear term with the hightest derivative of \(\varepsilon\) and \(u\) are dependent on the steady state \(\bar{\rho}\), and roughly \(C_{1}(\bar{\rho})\) increases as the steady state \(\bar{\rho}\) increases. Therefore, we need to obtain the exact magnitude relationship between the linear term coefficient \(C_{1}(\bar{\rho})\) and the linear growth mode \(\mu_{0}\), which is the main goal of this section. In \cite{L}, the author transformed \(\eqref{linearized eq SS}\) into a Sturm-Liouville type equation with Robin type boundary condition, and then proved the existence of the fastest linear growth mode \(\mu_{0}\). We outline the results here. By setting \(\zeta(y,t)=e^{\lambda t}\chi(y)\) we rewrite \(\eqref{linearized eq SS}\) as
\begin{align}\label{5.48}
\begin{split}
L\chi:=-\gamma\partial_{y}(\bar{\rho}^{\gamma}y^{4}\partial_{y}\chi)+(4-3\gamma)y^{3}\chi\partial_{y}\bar{\rho}^{\gamma}=&-\lambda^{2}y^{4}\bar{\rho}\chi \\
3\chi(R)+R\partial_{y}\chi(R)=&0.
\end{split}
\end{align}
Given \(\chi\in C^{2}([0,R])\) satisfying the boundary condition of \(\eqref{5.48}\), we use integration by parts
\begin{align}\label{111}
\langle L\chi,\chi\rangle=&\int_{0}^{R}\gamma\bar{\rho}^{\gamma}y^{4}(\partial_{y}\chi)^{2}+(4-3\gamma)y^{3}\chi^{2}\partial_{y}\bar{\rho}^{\gamma}dy+3\gamma R^{3}\chi(R)^{2},
\end{align}
where \(\left\langle \cdot\right\rangle \) represent the usual \(L^{2}\) inner product. Let \(H^{1}_{r}(B_{R}(\bbR^{5}))\) denote the subspace of spherically symmetric functions in \(H^{1}(B_{R}(\bbR^{5}))\). We can conclude
\begin{align}\label{5.49}
\inf_{\Arrowvert\chi\Arrowvert_{y^{4}\bar{\rho}}=1}\langle L\chi,\chi\rangle:=\mu_{\ast}=\inf[\mu:\exists\chi\ne 0 s.t. L\chi=\mu y^{4}\bar{\rho}\chi],
\end{align}
where \(\Arrowvert\chi\Arrowvert^{2}_{\omega}=\langle \chi,\omega\chi\rangle \). Moreover there exists \(\chi_{\ast}\in H^{1}_{r}(B_{R}(\bbR^{5}))\) that allows the infimum to be reached. This shows that if there exist \(\chi\) such that \(\langle L\chi,\chi\rangle <0\), then by \(\eqref{5.49}\) there exist \(-\mu_{0}<0\) and \(\chi_{\ast}\) such that \(L\chi_{\ast}=-\mu_{0}y^{4}\bar{\rho}\chi_{\ast}\). The following we need to estimate the magnitude of \(\mu_{0}\). We know that the family of gaseous steady states are self-similar, so that the family is given by \(\bar{\rho}_{\kappa}(y)=\kappa\bar{\rho}_{\ast}(\kappa^{1-\gamma/2}y)\) where \(\bar{\rho}_{\ast}\) is a steady state. The corresponding liquid star has \(R_{\kappa}=\kappa^{-(1-\gamma/2)}\bar{\rho}^{-1}_{\ast}(1/\kappa)\). We deal with three cases individually.\\
{\bf Case\ 1:} \(\frac{6}{5}<\gamma<\frac{4}{3}\)

With \(\eqref{111}\), we have
\begin{align}
\begin{split}
\langle L_{\kappa}\chi,\chi\rangle=&\int_{0}^{R_{\kappa}}\gamma\bar{\rho}_{\kappa}^{\gamma}y^{4}(\partial_{y}\chi)^{2}+(4-3\gamma)y^{3}\chi^{2}\partial_{y}\bar{\rho}_{\kappa}^{\gamma}dy+3\gamma R_{\kappa}^{3}\chi(R_{\kappa})^{2}\\
=&R_{\kappa}^{3}k^{\gamma}\int_{0}^{1}\gamma\bar{\rho}_{\ast}(\bar{\rho}_{\ast}^{-1}(1/\kappa)z)^{\gamma}z^{d+1}(\partial_{z}\tilde{\chi})^{2}+(4-3\gamma)z^{d}\bar{\rho}_{\ast}^{-1}(1/\kappa)\tilde{\chi}^{2}(\bar{\rho}_{\ast}^{\gamma})'(\bar{\rho}_{\ast}^{-1}(1/\kappa)z)dz+3\gamma R_{\kappa}^{3}\chi(R_{\kappa})^{2}.
\end{split}
\end{align}
When \(\gamma>\frac{6}{5}\), the gaseous steady state \(\bar{\rho}_{\ast}\) has compact support. Then \(\bar{\rho}_{\ast}^{-1}(1/\kappa)\to\bar{\rho}^{-1}_{\ast}(0)=R_{\ast}\) as \(\kappa\to\infty\). According to dominated convergence we have
\begin{align}\label{5.51}
\begin{split}
\langle L_{\kappa}1,1\rangle=&(4-3\gamma)R_{\kappa}^{3}\kappa^{\gamma}\bar{\rho}_{\ast}^{-1}(1/\kappa)\int_{0}^{1}z^{d}(\bar{\rho}_{\ast}^{\gamma})'(\bar{\rho}_{\ast}^{-1}(1/\kappa)z)dz+3\gamma R_{\kappa}^{3}\\
\sim&-\kappa^{\frac{5}{2}\gamma-3}\to-\infty\quad\quad as\quad\quad \kappa\to\infty,
\end{split}
\end{align}
and
\begin{align}\label{5.52}
\langle 1,y^{4}\bar{\rho}_{\kappa}1\rangle=\int_{0}^{R_{\kappa}}y^{4}\bar{\rho}_{\kappa}(y)dy=R_{\kappa}^{5}\kappa\int_{0}^{1}z^{4}\bar{\rho}_{\ast}(\bar{\rho}_{\ast}^{-1}(1/\kappa)z)dz\sim\kappa^{\frac{5}{2}-4}.
\end{align}
Combining \(\eqref{5.49}\), \(\eqref{5.51}\) and \(\eqref{5.52}\), we have \(\mu_{0}\sim\kappa\) when \(\kappa\) is large enough. Then we turn to consider the magnitude of \(C_{1}(\bar{\rho}_{\kappa})\). According to equations \(\eqref{2.33}\) and \(\eqref{2.34}\), we can obtain
\begin{align}\label{5.54}
C_{1}(\bar{\rho}_{\kappa})\sim\left( c_{s}^{2}(\bar{\bm{\uprho}}_{\kappa})+c_{s}(\bar{\bm{\uprho}}_{\kappa})c'_{s}(\bar{\bm{\uprho}}_{\kappa})\right)\nabla\bar{\bm{\uprho}}_{\kappa} \sim\gamma\bar{\rho}_{\kappa}^{\gamma-1}\nabla\bar{\bm{\uprho}}_{\kappa}\lesssim-\frac{1}{\bar{\rho}_{\kappa}}\partial_{y}\bar{\rho}^{\gamma}_{\kappa}.
\end{align}
In spherical symmetry, the steady state satisfies equation
\begin{align}\label{5.55}
\frac{4\pi}{y^{2}}\int_{0}^{y}s^{2}\bar{\rho}(s)ds+\frac{1}{\bar{\rho}}\partial_{y}\bar{\rho}^{\gamma}=0.
\end{align}
Let \(y=\nu R_{\kappa}\), where \(0<\nu\le 1\), we have
\begin{align}
|C_{1}(\bar{\rho}_{\kappa})|\lesssim\frac{1}{\nu^{2}R_{\kappa}^{2}}\int_{0}^{\nu R_{\kappa}}s^{2}\bar{\rho}_{\kappa}(s)ds=\nu R_{\kappa}\kappa\int_{0}^{1}z^{2}\bar{\rho}_{\ast}(\bar{\rho}_{\ast}^{-1}(1/\kappa)\nu z)dz\lesssim\kappa^{\frac{\gamma}{2}}.
\end{align}
Since \(\frac{\gamma}{2}<1\), there exists small enough \(\epsilon>0\) such that \(|C_{1}(\bar{\rho}_{\kappa})|\le \epsilon\mu_{0}\) when \(\kappa\) is large enough. \\
{\bf Case\ 2:} \(\gamma=\frac{6}{5}\)

From the explicit formula which has been proven in \cite{heinzle2002finiteness}, we have
\begin{align}\label{22}
\begin{split}
\bar{\rho}_{\kappa}(y)=&\left( \kappa^{-\frac{2}{5}}+\frac{2\pi}{9}\kappa^{\frac{2}{5}}y^{2}\right) ^{-\frac{5}{2}}\\
\partial_{y}\bar{\rho}_{\kappa}^{\gamma}(y)=&-\frac{4\pi}{3}\kappa^{\frac{2}{5}}y\left( \kappa^{-\frac{2}{5}}+\frac{2\pi}{9}\kappa^{\frac{2}{5}}y^{2}\right) ^{-4}\\
R_{\kappa}=&\frac{3}{\sqrt{2\pi}}\kappa^{-\frac{2}{5}}\left( \kappa^{\frac{2}{5}}-1\right) ^{\frac{1}{2}}.
\end{split}
\end{align}
From \(\eqref{111}\) we have
\begin{align}
\begin{split}
\langle L_{\kappa}1,1\rangle=&\int_{0}^{R_{\kappa}}(4-3\gamma)y^{3}\partial_{y}\bar{\rho}_{\kappa}^{\gamma}dy+3\gamma R_{\kappa}^{3}\\
=&-\frac{4\pi}{3}(4-3\gamma)\kappa^{\frac{2}{5}}\int_{0}^{R{\kappa}}y^{4}\left( \kappa^{-\frac{2}{5}}+\frac{2\pi}{9}\kappa^{\frac{2}{5}}y^{2}\right) ^{-4}dy+3\gamma R_{\kappa}^{3}\\
=&-\frac{4\pi}{3}(4-3\gamma)\kappa^{-\frac{3}{5}}\int_{0}^{\kappa^{\frac{1}{5}}R{\kappa}}z^{4}\left( \kappa^{-\frac{2}{5}}+\frac{2\pi}{9}z^{2}\right) ^{-4}dz+3\gamma R_{\kappa}^{3}\\
=&-3(4-3\gamma)\left( \frac{9}{2\pi}\right) ^{\frac{3}{2}}\kappa^{-\frac{3}{5}}\int_{\kappa^{-\frac{2}{5}}}^{\kappa^{-\frac{2}{5}}+\frac{2\pi}{9}\kappa^{\frac{2}{5}}R_{\kappa}^{2}}\frac{\left( s-\kappa^{-\frac{2}{5}}\right) ^{\frac{3}{2}}}{s^{4}}ds+3\gamma R_{\kappa}^{3}\\
\le&-3(4-3\gamma)\left( \frac{9}{4\pi}\right) ^{\frac{3}{2}}\kappa^{-\frac{3}{5}}\int_{2\kappa^{-\frac{2}{5}}}^{\kappa^{-\frac{2}{5}}+\frac{2\pi}{9}\kappa^{\frac{2}{5}}R_{\kappa}^{2}}s^{-\frac{5}{2}}ds+3\gamma R_{\kappa}^{3}\\
=&-2(4-3\gamma)\left( \frac{9}{4\pi}\right) ^{\frac{3}{2}}\kappa^{-\frac{3}{5}}\left( \left( 2\kappa^{-\frac{2}{5}}\right) ^{-\frac{3}{2}}-\left( \kappa^{-\frac{2}{5}}+\frac{2\pi}{9}\kappa^{\frac{2}{5}}R_{\kappa}^{2}\right) ^{-\frac{3}{2}}\right) +3\gamma R_{\kappa}^{3}\\
\to&-2(4-3\gamma)\left( \frac{9}{8\pi}\right) ^{\frac{3}{2}}\quad\quad as\quad\quad \kappa\to\infty,
\end{split}
\end{align}
and
\begin{align*}
\begin{split}
\langle 1,y^{4}\bar{\rho}_{\kappa}1\rangle=&\int_{0}^{R_{\kappa}}y^{4}\bar{\rho}_{\kappa}(y)dy
=\int_{0}^{R_{\kappa}}y^{4}\left( \kappa^{-\frac{2}{5}}+\frac{2\pi}{9}\kappa^{\frac{2}{5}}y^{2}\right) ^{-\frac{5}{2}}dy\\
=&\kappa^{-1}\int_{0}^{\kappa^{\frac{1}{5}}R_{\kappa}}z^{4}\left( \kappa^{-\frac{2}{5}}+\frac{2\pi}{9}z^{2}\right) ^{-\frac{5}{2}}dz
=\frac{9}{4\pi}\left( \frac{9}{2\pi}\right) ^{\frac{3}{2}}\kappa^{-1}\int_{\kappa^{-\frac{2}{5}}}^{\kappa^{-\frac{2}{5}}+\frac{2\pi}{9}\kappa^{\frac{2}{5}}R_{\kappa}^{2}}\frac{\left( s-\kappa^{-\frac{2}{5}}\right) ^{\frac{3}{2}}}{s^{\frac{5}{2}}}ds\\
\lesssim&\kappa^{-1}\int_{\kappa^{-\frac{2}{5}}}^{\kappa^{-\frac{2}{5}}+\frac{2\pi}{9}\kappa^{\frac{2}{5}}R_{\kappa}^{2}}s^{-1}ds=\kappa^{-1}\left( \log\left( \kappa^{-\frac{2}{5}}+\frac{2\pi}{9}\kappa^{\frac{2}{5}}R_{\kappa}^{2}\right) -\log\kappa^{-\frac{2}{5}}\right) \lesssim\frac{\log\kappa}{\kappa}.
\end{split}
\end{align*}
According to \(\eqref{5.49}\), we know \(\mu_{0}\sim \frac{\kappa}{\log\kappa}\) when \(\kappa\) is large enough. As discussed in case 1, we consider the relationship between the coefficient \(C_{1}(\bar{\rho}_{\kappa})\) and \(\kappa\). Using the explicit formula \(\eqref{22}\) and \(\eqref{5.54}\) we have
\begin{align*}
\left| C_{1}(\bar{\rho}_{\kappa})\right| \lesssim- \frac{1}{\bar{\rho}_{\kappa}}\partial_{y}\bar{\rho}_{\kappa}^{\gamma} =\frac{4\pi}{3}\kappa^{\frac{2}{5}}y\left( \kappa^{-\frac{2}{5}}+\frac{2\pi}{9}\kappa^{\frac{2}{5}}y^{2}\right) ^{-\frac{3}{2}}\le\max_{0\le\nu\le 1}\frac{4\pi}{3}\kappa^{\frac{2}{5}}\nu R_{\kappa}\left( \kappa^{-\frac{2}{5}}+\frac{2\pi}{9}\kappa^{\frac{2}{5}}\nu^{2}R_{\kappa}^{2}\right) ^{-\frac{3}{2}}.
\end{align*}
Naturally we define the function
\begin{align*}
f(\nu):=\kappa^{\frac{2}{5}}\nu R_{\kappa}\left( \kappa^{-\frac{2}{5}}+\frac{2\pi}{9}\kappa^{\frac{2}{5}}\nu^{2}R_{\kappa}^{2}\right) ^{-\frac{3}{2}}
\end{align*}
It is straightforward to calculate
\begin{align*}
f'(\nu)=\kappa^{\frac{2}{5}}R_{\kappa}\frac{\left( \kappa^{-\frac{2}{5}}+\frac{2\pi}{9}\kappa^{\frac{2}{5}}\nu^{2}R_{\kappa}^{2}\right) ^{\frac{3}{2}}-\frac{2\pi}{3}\kappa^{\frac{2}{5}}\nu^{2}R_{\kappa}^{2}\left( \kappa^{-\frac{2}{5}}+\frac{2\pi}{9}\kappa^{\frac{2}{5}}\nu^{2}R_{\kappa}^{2}\right) ^{\frac{1}{2}}}{\left( \kappa^{-\frac{2}{5}}+\frac{2\pi}{9}\kappa^{\frac{2}{5}}\nu^{2}R_{\kappa}^{2}\right) ^{3}}=\kappa^{\frac{2}{5}}R_{\kappa}\frac{\kappa^{-\frac{2}{5}}-\frac{4\pi}{9}\kappa^{\frac{2}{5}}\nu^{2}R_{\kappa}^{2}}{\left( \kappa^{-\frac{2}{5}}+\frac{2\pi}{9}\kappa^{\frac{2}{5}}\nu^{2}R_{\kappa}^{2}\right) ^{\frac{5}{2}}}
\end{align*}
and the maximum is attained when \(\nu=\frac{3}{2\sqrt{\pi}}\kappa^{-\frac{2}{5}}R_{\kappa}^{-1}\). Therefore we have
\begin{align*}
\max\left|C_{1}(\bar{\rho}_{\kappa}) \right| \lesssim f(\frac{3}{2\sqrt{\pi}}\kappa^{-\frac{2}{5}}R_{\kappa}^{-1})\lesssim\kappa^{\frac{3}{5}}
\end{align*}
Since \(\kappa^{\frac{3}{5}}\ll\frac{\kappa}{\log\kappa}\) when \(\kappa\) is large enough, there exists small enough \(\epsilon>0\) such that \(|C_{1}(\bar{\rho}_{\kappa})|\le \epsilon\mu_{0}\).\\
{\bf Case\ 3:} \(1\le\gamma<\frac{6}{5}\)\\
Let \(\bar{\rho}_{\ast}\) be a gaseous steady state and \(\bar{\rho}_{\kappa}(r)=\kappa\bar{\rho}_{\ast}(\kappa^{1-\gamma/2}r)\). Then there exist \(c>0\) such that
\begin{align}\label{333}
\begin{split}
\left| r^{\frac{2}{2-\gamma}}\bar{\rho}_{\kappa}(r)-\left(\frac{1}{2\pi}\frac{\gamma(4-3\gamma)}{(2-\gamma)^{2}}
\right)^{\frac{1}{2-\gamma}} \right| =&\left| r^{\frac{2}{2-\gamma}}\bar{\rho}_{\kappa}(r)-\upsilon^{\ast}_{1}\right| \lesssim\left( \kappa^{1-\frac{\gamma}{2}}r\right) ^{-c}\\
\left| r^{\frac{2}{2-\gamma}-3}\bar{m}_{\kappa}(r)-\frac{2\gamma}{2-\gamma}\left(\frac{1}{2\pi}\frac{\gamma(4-3\gamma)}{(2-\gamma)^{2}}
\right)^{\frac{\gamma-1}{2-\gamma}}\right| =&\left|r^{\frac{2}{2-\gamma}-3}\bar{m}_{\kappa}(r)-\upsilon^{\ast}_{2}\right| \lesssim \left( \kappa^{1-\frac{\gamma}{2}}r\right) ^{-c}\\
R_{\kappa}\to R_{\infty}:=(\upsilon^{\ast}_{1})^{1-\frac{\gamma}{2}}=&\left(\frac{1}{2\pi}\frac{\gamma(4-3\gamma)}{(2-\gamma)^{2}}
\right)^{\frac{1}{2}}\quad\quad as\quad\quad \kappa\to\infty,
\end{split}
\end{align}
where
\begin{align*}
\bar{m}_{\kappa}(r)=4\pi\int_{0}^{r}s^{2}\bar{\rho}_{\kappa}(s)ds.
\end{align*} 
These are known results (see \cite{L}). Based on the above formula we first consider the bound on the coefficient \(C_{1}(\bar{\rho}_{\kappa})\). From \(\eqref{5.55}\), we have
\begin{align*}
\left| C_{1}(\bar{\rho}_{\kappa})\right| \lesssim- \frac{1}{\bar{\rho}_{\kappa}}\partial_{y}\bar{\rho}_{\kappa}^{\gamma}=\frac{4\pi}{y^{2}}\int_{0}^{y}s^{2}\bar{\rho}_{\kappa}(s)ds
\end{align*}
Let \(y=\nu R_{\kappa}\). when \(0\le \nu\le \kappa^{-(1-\frac{\gamma}{2})}\) we have
\begin{align*}
\left| C_{1}(\bar{\rho}_{\kappa})\right| \lesssim&\frac{1}{\nu^{2} R_{\kappa}^{2}}\int_{0}^{\nu R_{\kappa}}s^{2}\kappa\bar{\rho}_{\ast}(\kappa^{1-\frac{\gamma}{2}}s)ds
=\nu R_{\kappa}\int_{0}^{1}z^{2}\kappa\bar{\rho}_{\ast}(\kappa^{1-\frac{\gamma}{2}}\nu R_{\kappa}z)dz\\
\lesssim&\kappa^{-(1-\frac{\gamma}{2})}\kappa R_{\kappa}\int_{0}^{1}z^{2}\Arrowvert\bar{\rho}_{\ast}\Arrowvert_{\infty}dz
\lesssim \kappa^{\frac{\gamma}{2}}
\end{align*}
When \(\kappa^{-(1-\frac{\gamma}{2})}\le\nu\le 1\) we have
\begin{align*}
\left| C_{1}(\bar{\rho}_{\kappa})\right| \lesssim\frac{1}{\nu^{2} R_{\kappa}^{2}}\int_{0}^{\nu R_{\kappa}}s^{2}\kappa\bar{\rho}_{\ast}(\kappa^{1-\frac{\gamma}{2}}s)ds=
\frac{1}{\nu^{2} R_{\kappa}^{2}}\int_{0}^{\nu R_{\kappa}\kappa^{1-\frac{\gamma}{2}}}\kappa\kappa^{-3(1-\frac{\gamma}{2})}z^{2}\bar{\rho}_{\ast}(z)dz.
\end{align*}
Fix some large \(M>0\) and divide the region of integration into two parts. Then we have
\begin{align*}
\left| C_{1}(\bar{\rho}_{\kappa})\right| \lesssim&\frac{\kappa^{1-3(1-\frac{\gamma}{2})}}{\nu^{2} R_{\kappa}^{2}}\int_{M}^{\nu R_{\kappa}\kappa^{1-\frac{\gamma}{2}}}\bar{\rho}_{\ast}(z)z^{2}dz+\frac{\kappa^{1-3(1-\frac{\gamma}{2})}}{\nu^{2} R_{\kappa}^{2}}\int_{0}^{M}\bar{\rho}_{\ast}(z)z^{2}dz\\
\lesssim&\frac{\kappa^{1-3(1-\frac{\gamma}{2})}}{\nu^{2} R_{\kappa}^{2}}\int_{M}^{\nu R_{\kappa}\kappa^{1-\frac{\gamma}{2}}}z^{2-\frac{2}{2-\gamma}}dz+\frac{\kappa^{1-3(1-\frac{\gamma}{2})}}{\nu^{2} R_{\kappa}^{2}}\int_{0}^{M}\Arrowvert\bar{\rho}_{\ast}\Arrowvert_{\infty}z^{2}dz\\
\lesssim&\frac{\kappa^{1-3(1-\frac{\gamma}{2})}}{\nu^{2} R_{\kappa}^{2}}\left( \nu R_{\kappa}\kappa^{1-\frac{\gamma}{2}}\right) ^{3-\frac{2}{2-\gamma}}+\frac{\kappa^{1-3(1-\frac{\gamma}{2})}}{\nu^{2} R_{\kappa}^{2}}\int_{0}^{M}\Arrowvert\bar{\rho}_{\ast}\Arrowvert_{\infty}z^{2}dz\\
\lesssim&\nu^{1-\frac{2}{2-\gamma}}\kappa^{1-3(1-\frac{\gamma}{2})+(1-\frac{\gamma}{2})(3-\frac{2}{2-\gamma})}+\nu^{-2}\kappa^{1-3(1-\frac{\gamma}{2})}\lesssim\kappa^{\frac{\gamma}{2}}
\end{align*}
According to previous discussion, as long as the linear growth mode \(\mu_{0}\) increase in magnitude with respect to \(\kappa\) is larger than \(\frac{\gamma}{2}\), we will finish the proof. From \(\eqref{111}\) we have
\begin{align*}
\langle L_{\kappa}\chi,\chi\rangle=&\int_{0}^{R_{\kappa}}\gamma\bar{\rho}_{\kappa}^{\gamma}y^{4}(\partial_{y}\chi)^{2}+(4-3\gamma)y^{3}\chi^{2}\partial_{y}\bar{\rho}_{\kappa}^{\gamma}dy+3\gamma R_{\kappa}^{3}\chi(R_{\kappa})^{2}\\
=&\int_{0}^{R_{\kappa}}\gamma\bar{\rho}_{\kappa}^{\gamma}y^{4}(\partial_{y}\chi)^{2}-(4-3\gamma)y\chi^{2}\bar{m}_{\kappa}\bar{\rho}_{\kappa}dy+3\gamma R_{\kappa}^{3}\chi(R_{\kappa})^{2}.
\end{align*}
Fix \(\nu>0\), and suppose
\begin{align}\label{444}
	\chi(y)=\left\{
\begin{aligned}
&\nu^{-a}\kappa^{a(1-\frac{\gamma}{2})},\quad\quad 0\le y\le \nu\kappa^{-(1-\frac{\gamma}{2})}\\
&y^{-a},\quad\quad\nu\kappa^{-(1-\frac{\gamma}{2})}\le y\le R_{\kappa}
\end{aligned}
\right.,
\end{align}
we have
\begin{align}\label{5.60}
\begin{split}
\langle L_{\kappa}\chi,\chi\rangle\le&\int_{\nu\kappa^{-(1-\frac{\gamma}{2})}}^{R_{\kappa}}\gamma\bar{\rho}_{\kappa}^{\gamma}y^{4}(\partial_{y}\chi)^{2}-(4-3\gamma)y\chi^{2}\bar{m}_{\kappa}\bar{\rho}_{\kappa}dy+3\gamma R_{\kappa}^{3-2a}\\
\le&\int_{\nu\kappa^{-(1-\frac{\gamma}{2})}}^{R_{\kappa}}\gamma\left( 1+\epsilon_{1}\right) ^{\gamma}\left( \upsilon^{\ast}_{1}\right) ^{\gamma}y^{5-\frac{2+\gamma}{2-\gamma}}\left( \partial_{y}\chi\right) ^{2}-\frac{2\gamma(4-3\gamma)}{2-\gamma}\left( 1-\epsilon_{2}\right) ^{2}\left( \upsilon^{\ast}_{1}\right) ^{\gamma}y^{4-\frac{4}{2-\gamma}}\chi^{2}dy+3\gamma R_{\kappa}^{3-2a}\\
=&\gamma\left( \upsilon^{\ast}_{1}\right) ^{\gamma}\left( a^{2}\left( 1+\epsilon_{1}\right) ^{\gamma}-2\left( 1-\epsilon_{2}\right) ^{2}\left( 3-\frac{2}{2-\gamma}\right) \right) \int_{\nu\kappa^{-(1-\frac{\gamma}{2})}}^{R_{\kappa}}y^{3-\frac{2+\gamma}{2-\gamma}-2a}dy,
\end{split}
\end{align}
and
\begin{align*}
\langle\chi,y^{4}\bar{\rho}_{\kappa}\chi\rangle=\int_{0}^{R_{\kappa}}y^{4}\bar{\rho}_{\kappa}\chi^{2}dy=\int_{\kappa^{-(1-\frac{\gamma}{2})}}^{R_{\kappa}}y^{4-2a}\bar{\rho}_{\kappa}dy+\int_{0}^{\kappa^{-(1-\frac{\gamma}{2})}}y^{4}\bar{\rho}_{\kappa}\kappa^{a(1-\frac{\gamma}{2})}dy:=A+B.
\end{align*}
We estimate the two parts respectively. With \(\eqref{333}\), we have
\begin{align}\label{5.61}
A\lesssim\int_{\kappa^{-(1-\frac{\gamma}{2})}}^{R_{\kappa}}y^{4-2a-\frac{2}{2-\gamma}}dy
\end{align}
and
\begin{align}\label{5.62}
B=\int_{0}^{\kappa^{-(1-\frac{\gamma}{2})}}y^{4}\kappa^{1+a(1-\frac{\gamma}{2})}\bar{\rho}_{\ast}(\kappa^{1-\frac{\gamma}{2}}y)dy\le\kappa^{1+a(1-\frac{\gamma}{2})}\Arrowvert\bar{\rho}_{\ast}\Arrowvert_{\infty}\int_{0}^{\kappa^{-(1-\frac{\gamma}{2})}}y^{4}dy\lesssim \kappa^{1+(a-5)(1-\frac{\gamma}{2})}.
\end{align}
Combining \(\eqref{5.60}\), \(\eqref{5.61}\) and \(\eqref{5.62}\), if there exists some constant \(a>0\) such that
\begin{align}\label{5.63}
\left\{
\begin{aligned}
&-(1-\frac{\gamma}{2})(4-\frac{2+\gamma}{2-\gamma}-2a)>\frac{\gamma}{2}\\
&a^{2}\left( 1+\epsilon_{1}\right) ^{\gamma}-2\left( 1-\epsilon_{2}\right) ^{2}\left( 3-\frac{2}{2-\gamma}\right)<0\\
&4-\frac{2}{2-\gamma}-2a\ge-1\\
&1+(a-5)(1-\frac{\gamma}{2})\le 0
\end{aligned}
\right.
\end{align}
hold, then it means that \(\mu_{0}\) has a larger magnitude than \(\kappa\). we can take large enough \(\kappa\) such that \(|C_{1}(\bar{\rho}_{\kappa})|\le \epsilon\mu_{0}\) for small enough \(\epsilon>0\). By simplifying \(\eqref{5.63}\), we only need to find the constant \(a>0\) that satisfies
\begin{align*}
2-\frac{1}{2-\gamma}<a<\sqrt{\frac{(1-\epsilon_{2})^{2}}{(1+\epsilon_{1})^{\gamma}}
\left( 6-\frac{4}{2-\gamma}\right) }.
\end{align*}
It is obvious that there exists a constant \(a>0\) satisfying the above inequality when \(\epsilon_{1}\) and \(\epsilon_{2}\) are small enough (equivalent to picking sufficiently large \(\nu\) in \(\eqref{444}\)). This completes the argument.

For the non-top order terms on the left-hand side of \(\eqref{6.1}\), it is straightforward to be bounded by applying Sobolev interpolation inequality (the reason is that the coefficient \(C_{2}(\bar{\rho}_{\kappa})\) may have a faster growth rate than \(\mu_{0}\) as \(\kappa\to\infty\)). The following we roughly illustrate the idea.
\begin{align}\label{5.56}
\int_{0}^{t}\calE_{l-1}ds=\sum_{j=0}^{l-1}\int_{0}^{t}\int_{\calB(s)}|\partial_{t,x}B^{j}u|^{2}dxds+\sum_{j=0}^{l-1}\int_{0}^{t}\int_{\partial\calB(\tau)}|B^{j+1}u|^{2}dSd\tau+\sum_{j=0}^{l}\int_{0}^{t}\int_{\calB(s)}|\partial_{t,x}B^{j}\varepsilon|^{2}dxds.
\end{align}
The first term on the right-hand side of \(\eqref{5.56}\) can be written as
\begin{align}\label{5.57}
\begin{split}
&\sum_{j=0}^{l-1}\int_{0}^{t}\int_{\calB(s)}|\partial_{t,x}B^{j}u|^{2}dxds\\
\le&\theta_{1}\int_{0}^{t}\int_{\calB(s)}|\partial_{t,x}B^{l}u|^{2}dxds+C_{\theta_{1}}\int_{0}^{t}\int_{\calB(s)}|\partial_{t,x}u|^{2}dxds\\
\le&\theta_{1}\int_{0}^{t}\int_{\calB(s)}|\partial_{t,x}B^{l}u|^{2}dxds+\theta_{2}C_{\theta_{1}}\int_{0}^{t}\int_{\calB(s)}|\partial_{t,x}^{2}u|^{2}dxds+C_{\theta_{1}}C_{\theta_{2}}\int_{0}^{t}\int_{\calB(s)}|u|^{2}dxds,
\end{split}
\end{align}
where we apply Sobolev interpolation inequality to time norm and space norm of Lagrange coordinates respectively (in this case, \(B\sim\partial_{t}\) and \(\partial_{x}\sim\nabla\)). For some large \(\kappa\), we take \(\theta_{1},\theta_{2}\) small enough such that
\begin{align}
C_{2}(\bar{\rho})(\theta_{1}+\theta_{2}C_{\theta_{1}})\lesssim\epsilon\mu_{0}.
\end{align} 
Then the terms on the right-hand side of \(\eqref{5.57}\) can be bounded. The remainder terms on the the right-hand side of \(\eqref{5.56}\) is bounded by the trace theorem and the same way. Combined with the previous discussion, we have completed the proof of \(\eqref{6.1}\).

\bibliographystyle{plain}
\bibliography{ee}

\bigskip

\centerline{\scshape Zeming Hao}
\smallskip
{\footnotesize
	\centerline{School of Mathematics and Statistics, Wuhan University}
	\centerline{Wuhan, Hubei 430072, China}
	\centerline{\email{2021202010062@whu.edu.cn}}
}

\medskip

\centerline{\scshape Shuang Miao}
\smallskip
{\footnotesize
	\centerline{School of Mathematics and Statistics, Wuhan University}
	\centerline{Wuhan, Hubei 430072, China}
	\centerline{\email{shuang.m@whu.edu.cn}}
} 
\end{document}